\documentclass[11pt,a4paper,reqno]{amsart}
\pdfoutput=1
\usepackage{times}
\usepackage[margin=1.0in,tmargin=0.9in,bmargin=0.80in]{geometry}
\usepackage[dvipsnames]{xcolor}
\usepackage[english]{babel} 
\definecolor{bred}{rgb}{0.8,0,0}
\usepackage{amsmath,amssymb,graphicx,epsfig}
\usepackage{mathrsfs}
\usepackage{mathtools}
\usepackage{enumerate,amsthm,epstopdf}
\usepackage{enumitem}
\setlist[enumerate]{label={\upshape(\roman*)}}
\usepackage[flushleft]{threeparttable}
\usepackage{multirow}
\usepackage{longtable}
\usepackage[utf8]{inputenc} 
\usepackage[T1]{fontenc}    
\usepackage{bbm}
\usepackage{xargs}
\usepackage{latexsym}
\usepackage{wasysym}
\usepackage{dsfont}
\usepackage{subcaption}
\usepackage{float}
\usepackage{hyperref}       
\usepackage{url}            
\usepackage{booktabs}       
\usepackage{amsfonts}       
\usepackage{nicefrac}       
\usepackage{microtype}      
\usepackage{cleveref}
\usepackage{array}
\usepackage{booktabs}
\usepackage[textsize=footnotesize]{todonotes}
\hypersetup{colorlinks,linkcolor={blue},citecolor={bred},urlcolor={blue}}
\usepackage[numbers]{natbib}
\usepackage[toc,page]{appendix}

\usepackage{stmaryrd}
\usepackage{algorithm,algpseudocode}
\usepackage{tabularx}
\usepackage{algorithmicx}
\allowdisplaybreaks

\numberwithin{equation}{section}

\makeatletter
\newcommand{\multiline}[1]{%
  \begin{tabularx}{\dimexpr\linewidth-\ALG@thistlm}[t]{@{}X@{}}
    #1
  \end{tabularx}
}
\makeatother

\usepackage{scalerel,stackengine}
\stackMath
\newcommand\reallywidehat[1]{%
	\savestack{\tmpbox}{\stretchto{%
			\scaleto{%
				\scalerel*[\widthof{\ensuremath{#1}}]{\kern.1pt\mathchar"0362\kern.1pt}%
				{\rule{0ex}{\textheight}}
			}{\textheight}%
		}{2.4ex}}%
	\stackon[-6.9pt]{#1}{\tmpbox}%
}
\newcommand\reallywidetilde[1]{%
	\savestack{\tmpbox}{\stretchto{%
			\scaleto{%
				\scalerel*[\widthof{\ensuremath{#1}}]{\kern.1pt\mathchar"0366\kern.1pt}%
				{\rule{0ex}{\textheight}}
			}{\textheight}%
		}{2.4ex}}%
	\stackon[-6.9pt]{#1}{\tmpbox}%
}

\newtheorem{theorem}{Theorem}[section]
\newtheorem{proposition}[theorem]{Proposition}
\newtheorem{lemma}[theorem]{Lemma}
\newtheorem{corollary}[theorem]{Corollary}
\newtheorem{remark}[theorem]{Remark}
\newtheorem{definition}[theorem]{Definition}

\newtheorem{assumption}[theorem]{Assumption}

\newcommand{\R}{\mathbb{{R}}}

\newcommand\bE{\mathbb{E}}
\newcommand\bF{\mathbb{F}}

\newcommand\bN{\mathbb{N}}
\newcommand\bR{\mathbb{R}}
\newcommand\bP{\mathbb{P}}

\newcommand\bV{\mathbb{V}}

\newcommand\cA{\mathcal{A}}
\newcommand\cB{\mathcal{B}}

\newcommand\cF{\mathcal{F}}

\newcommand\cL{\mathcal{L}}
\newcommand\cM{\mathcal{M}}

\newcommand\cV{\mathcal{V}}
\newcommand\cX{\mathcal{X}}

\newcommand{\cal}{\curvearrowleft}
\DeclareMathOperator*{\diag}{diag}

\DeclareMathOperator{\re}{Re}

\DeclareMathOperator{\argmin}{arg\,min}
\DeclareMathOperator{\err}{err}

\makeatletter
 \newcommand{\sumstar}
 {\operatornamewithlimits{\sum@\kern-.2em\raise1ex\hbox{*}}}
 \makeatother
 
\makeatletter
\renewenvironment{cases}[1][l]{\matrix@check\cases\env@cases{#1}}{\endarray\right.}
\def\env@cases#1{%
	\let\@ifnextchar\new@ifnextchar
	\left\lbrace\def\arraystretch{1.2}%
	\array{@{}#1@{\quad}l@{}}}
\makeatother

\usepackage{etoolbox}
\newcolumntype{R}[1]{>{\raggedleft\let\newline\\\arraybackslash\hspace{0pt}}m{#1}}

\newcommand{\<}{\langle}
\renewcommand{\>}{\rangle}

\allowdisplaybreaks

\begin{document}

\title[]{Full error analysis of the random deep splitting method\\for nonlinear parabolic PDEs and PIDEs}

\author[A. Neufeld]{Ariel Neufeld}
\address{Nanyang Technological University, Division of Mathematical Sciences, Singapore}
\email{ariel.neufeld@ntu.edu.sg}

\author[P. Schmocker] {Philipp Schmocker}
\address{Nanyang Technological University, Division of Mathematical Sciences, Singapore}
\email{philippt001@e.ntu.edu.sg}

\author[S. Wu]{Sizhou Wu}
\address{Shanghai University of Finance and Economics, School of Mathematics, Shanghai}
\email{sizhou.wu@sufe.edu.cn}

\date{}
\thanks{Financial support by the Nanyang Assistant Professorship Grant (NAP Grant) \emph{Machine Learning based Algorithms in Finance and Insurance} is gratefully acknowledged.}

\subjclass[2010]{}

\begin{abstract}
	In this paper, we present a randomized extension of the deep splitting algorithm introduced in \cite{beck2021deep} using \textit{random} neural networks suitable to approximately solve both high-dimensional nonlinear parabolic PDEs and PIDEs with jumps having (possibly) infinite activity. We provide a full error analysis of our so-called \textit{random deep splitting} method. In particular, we prove that our random deep splitting method converges to the (unique viscosity) solution of the nonlinear PDE or PIDE under consideration. Moreover, we empirically analyze our random deep splitting method by considering several numerical examples including both nonlinear PDEs and nonlinear PIDEs relevant in the context of pricing of financial derivatives under default risk. In particular, we empirically demonstrate in all examples that our random deep splitting method can approximately solve nonlinear PDEs and PIDEs in 10'000 dimensions \textit{within seconds}.
\end{abstract}

\maketitle

\section{\textbf{Introduction}}

In this paper, we develop a deep learning based numerical algorithm using \textit{random} neural networks to solve high-dimensional nonlinear parabolic partial differential equations (PDEs) and partial integro-differential equations (PIDEs) with jumps having possibly infinite activity. Moreover, we provide a full error analysis of our algorithm which guarantees its convergence to the corresponding (unique viscosity) solution of the PDE or PIDE under consideration.

PDEs and PIDEs have many important applications in biology, physics, engineering, economics, and finance; see, e.g., \cite{Boussange2023,ContTankov2004,Cont2005,Delong2013,Oksendal2006} and the references therein.
In general, nonlinear PDEs and PIDEs cannot be solved analytically and hence need to be numerically approximated. However, while classical numerical methods such as, e.g., finite difference or finite elements methods (see, e.g., \cite{LeRoux1989,Pani1992,Sloan1986,Yanik1988} and the references therein) typically provide convergence guarantee, they cannot be applied to numerically solve high-dimensional nonlinear PDEs or PIDEs. Multilevel Monte Carlo methods, especially multilevel Picard approximation algorithms have demonstrated to be able to solve certain high-dimensional nonlinear PDEs \cite{BGJ2020,beck2020overcoming,hutzenthaler2019multilevel,hutzenthaler2021multilevel,giles2019generalised,HJK2022,HJKN2020,HJKNW2020,hutzenthaler2020overcoming,HK2020,hutzenthaler2022multilevel1,NNW2023,neufeld2023multilevel} and PIDEs \cite{NW2022} with theoretical convergence guanratee. However, these methods can only approximate the corresponding PDE or PIDE at one single deterministic time-space point given as input of the algorithm.

In recent years, starting with the seminal work by \cite{han2018solving}, deep learning numerical methods using neural networks have been introduced to solve nonlinear PDEs in up to 10'000 dimensions; see, e.g., \cite{beck2020deep,beck2021deep,BBG+2021,beck2019machine,beck2020overview,berner2020numerically,weinan2021algorithms, EYu2018,FTT2019,han2018solving,han2020convergence,han2019solving,Hen2017,hure2020deep,ito2021neural,jacquier2023deep, KLY2021,lu2021deepxde,Mis2019,NM2019,nguwi2022deep,nguwi2022numerical,nguwi2023deep,Rai2023,raissi2019physics,reisinger2020rectified,sirignano2018dgm,zhang2020learning}. 

Numerical methods to solve high-dimensional nonlinear PIDEs however are still at its infancy, due to the additional non-local term in form of an integral which makes it significantly harder to approximately solving them numerically. In \cite{georgoulis2024deep} a deep learning method to solve linear PIDEs has been presented. In \cite{NW2022} a multilevel Picard approximation algorithm to solve nonlinear PIDEs has been derived, together with its convergence and complexity analysis. Moreover, the deep Galerkin algorithm developed by \cite{sirignano2018dgm} has been generalized to PIDEs by \cite{AlAradi2019}. Furthermore, physics informed neural networks (PINNs) to approximately solve nonlinear PIDEs have been employed in \cite{Yuan2022}. In addition, a deep learning algorithm for solving high-dimensional parabolic PIDEs and forward-backward stochastic differential equations with jumps has been proposed
in \cite{wang2023deep}.

Deep learning algorithms using neural networks for solving nonlinear PIDEs have been developed in \cite{AlAradi2019,Boussange2023,Castro2022,frey2022deep,Gnoatto2022}. However, compared to deep learning algorithms for nonlinear PDEs, the existing deep learning methods have not yet empirically demonstrated to be able to efficiently solve  nonlinear PIDEs in thousands of space dimensions.

Furthermore, for no deep learning based algorithms involving fully trained neural networks to solve either nonlinear PDEs or nonlinear PIDEs one has been able to provide a full error analysis of the algorithm, in particular to prove that the algorithm converges to the PDE or PIDE under consideration. This is mainly caused by the non-convex optimization problem one needs to solve to train the involved neural networks in order to approximate the PDE or PIDE under consideration, for which there is no convergence guarantee \cite{goodfellow2016deep}.

We also refer to the following results proving that deep neural networks can approximate the solutions of PDEs  \cite{AJK+2023, BGJ2020,CHW2022,EGJS2022,GHJVW2023, HJKN2020a, JSW2021, neufeld2024rectified} and PIDEs \cite{GononOptionPrices2021,gonon2023deep, NNW2023}, typically without the curse of dimensionality. However, these results are abstract and only prove the existence of a neural network which can well approximate the given solution of the PDE or PIDE, however it is left open how to construct it.

To overcome these difficulties, we first extend the deep splitting algorithm developed in \cite{beck2021deep} to nonlinear PIDEs with jumps having \emph{infinite} activity and provide an a priori estimate for our algorithm, see Theorem~\ref{ThmMain}. This extends the work \cite{frey2022deep} which in turn has extended the deep splitting algorithm developed in \cite{beck2021deep} to nonlinear PIDEs but with jumps of \emph{finite} activity, as well as extends the a priori estimates for the deep splitting method obtained for nonlinear PDEs \cite{GPW2022} and PIDEs of \emph{finite} activity \cite{frey2022convergence} to the \emph{infinite} activity case.

Second, we then employ \textit{random} neural networks, instead of classical (i.e., deterministic, fully trained) neural networks in order to approximate nonlinear PDEs and PIDEs, which we then call the \textit{random deep splitting method}. We also refer to the following literature where random neural networks have been employed to approximately solve linear and nonlinear PDEs \cite{dong2021local,jacquier2023random,NeufeldSchmocker2023,Wang2023} as well as linear PIDEs \cite{hutzenthaler2021multilevel,GononBlackScholesPDE2021}.

Random neural networks are single-hidden-layer feed-forward neural networks whose weights and biases inside the activation function are randomly initialized. Hence, only the linear readouts need to be trained, which can be performed efficiently, e.g., by the least squares method (see, e.g., \cite{bjorck1996numerical}). This approach is inspired by the seminal works on extreme learning machines (see \cite{Huang2006}), random feature regression (see \cite{Rahimi2007,RahimiUniform2008,Rahimi2008}), and reservoir computing (see \cite{Grigoryeva2018,Herbert2004,Maass2002}). We also refer to \cite{Gonon2023,NeufeldSchmocker2023} for universal approximation theorems for random neural networks.

Random neural networks have the following three crucial advantages compared to classical (i.e., deterministic, fully trained) neural networks. First, the amount of parameters which needs to be trained in a random neural network is significantly lower. As a consequence, their training is significantly faster than deterministic ones. Second, as only the linear readouts are trained, the corresponding optimization problem is \textit{convex} and thus always admits a minimizer. Third, the least squares method directly outputs a minimizer of the optimization problem, which therefore leads to no additional optimization error (as opposed to classical neural networks trained, e.g., via stochastic gradient descent). Hence, the overall generalization error consists only of the prediction error and can therefore be controlled (see \cite{Gonon2023,NeufeldSchmocker2023}).

We employ these advantages of random neural networks to derive a full error analysis of our random deep splitting method both for nonlinear PDEs and PIDEs (with jumps having possibly infinite activity). These results form the main theoretical contribution of the paper, see Theorem~\ref{ThmMainRN} and Corollary~\ref{CorMainRN}. We highlight that in the literature, a full error analysis together with its convergence guarantee has been established in \cite{GononBlackScholesPDE2021} for solving linear Black-Scholes PIDEs using random neural networks. Moreover, \cite{jacquier2023random} derived a convergence analysis for solving path-dependent PDEs by random neural networks in the context of rough volatility. In addition, \cite{NeufeldSchmocker2023} obtained a full error analysis for learning the heat equation using random neural networks. Furthermore, \cite{NeufeldSchmocker2022} applied random neural networks for pricing and hedging financial derivatives via Wiener-Ito chaos expansion. Moreover, \cite{dong2021local,Wang2023} used random neural networks to solve various PDEs in mathematical physics.

Furthermore, we empirically analyze our random deep splitting method by considering several numerical examples including both nonlinear PDEs  and nonlinear PIDEs relevant in the context of pricing of financial derivatives under default risk. We compare the performance of our random deep splitting algorithm to the classical deep learning algorithm involving deterministic neural networks, as well as a multilevel Picard approximation algorithm. In particular, we empirically demonstrate in all examples that our random deep splitting method can approximately solve nonlinear PDEs and PIDEs in 10'000 dimensions \textit{within seconds}, which is significantly faster than the other analyzed methods.

\subsection{Outline}

In Section~\ref{section deri}, we outline the (random) deep splitting algorithm, while we introduce the precise setting and assumptions together with our main theoretical results in Section~\ref{section setting main}. In Section~\ref{section numerics}, we present our numerical results, while we prove all the mathematical results of the paper in Section~\ref{section proof}.

\subsection{Notation}

As usual, we denote by $\mathbb{N} := \lbrace 1, 2, 3, \dots \rbrace$ and $\mathbb{N}_0 := \mathbb{N} \cup \lbrace 0 \rbrace$ the sets of natural numbers, whereas $\mathbb{R}$ and $\mathbb{C}$ represent the sets of real and complex numbers (with imaginary unit $\mathbf{i} := \sqrt{-1} \in \mathbb{C}$), respectively. Hereby, we use the notation $s \wedge t := \min(s,t)$ for $s,t \in \mathbb{R}$.

Moreover, for any $d \in \mathbb{N}$, we denote by $\mathbb{R}^d$ the $d$-dimensional Euclidean space equipped with the norm $\Vert x \Vert = \big( \sum_{i=1}^d \vert x_i \vert^2 \big)^{1/2}$. Hereby, we define the vector $\mathbf{1}_d := (1,\dots,1)^\top \in \mathbb{R}^d$, the componentwise exponential $\mathbb{R}^d \ni x := (x_1,\dots,x_d)^\top \mapsto \exp_d(x) := (\exp(x_1),\dots,\exp(x_d))^\top \in \mathbb{R}^d$, and the componentwise product $x \odot y := (x_1 y_1,\dots,x_d y_d)^\top \in \mathbb{R}^d$ of $x := (x_1,\dots,x_d)^\top \in \mathbb{R}^d$ and $y := (y_1,\dots,y_d)^\top \in \mathbb{R}^d$.

In addition, for any $d,e \in \mathbb{N}$, we denote by $\mathbb{R}^{d \times e}$ the vector space of matrices $A = (a_{i,j})_{i=1,\dots,d}^{j=1,\dots,e} \in \mathbb{R}^{d \times e}$, which is equipped with the Frobenius norm $\Vert A \Vert_F = \big( \sum_{i=1}^d \sum_{j=1}^e \vert a_{i,j} \vert^2 \big)^{1/2}$. Moreover, for $d=e$, we denote by $I_d \in \mathbb{R}^{d \times d}$ the identity matrix.

Furthermore, for any $d \in \mathbb{N}$ and topological space $S$, 
we denote by $C(S;\mathbb{R}^d)$ the vector space of continuous functions 
$f: S \rightarrow \mathbb{R}^d$. If $d = 1$, we abbreviate $C(S) := C(S;\mathbb{R})$. Moreover, we denote by $C_b(\mathbb{R}^d) \subseteq C(\mathbb{R}^d)$ the vector subspace of continuous and bounded functions $f: \mathbb{R}^d \rightarrow \mathbb{R}$, which is a Banach space under the supremum norm $\Vert f \Vert_{C_b(\mathbb{R}^d)} := \sup_{x \in \mathbb{R}^d} \vert f(x) \vert$. In addition, for $\gamma \in (0,\infty)$, we define $\overline{C_b(\mathbb{R}^d)}^\gamma$ as the closure of $C_b(\mathbb{R}^d)$ with respect to the norm
\begin{equation*}
	\Vert f \Vert_{C_{pol,\gamma}(\mathbb{R}^d)} := \sup_{x \in \mathbb{R}^d} \frac{\vert f(x) \vert}{(1 + \Vert x \Vert)^\gamma} < \infty.
\end{equation*}
Then, $(\overline{C_b(\mathbb{R}^d)}^\gamma,\Vert \cdot \Vert_{C_{pol,\gamma}(\mathbb{R}^d)})$ is by definition a Banach space and it holds that $f \in \overline{C_b(\mathbb{R}^d)}^\gamma$ if and only if $f \in C(\mathbb{R}^d)$ and $\lim_{R \rightarrow \infty} \sup_{x \in \mathbb{R}^d, \, \Vert x \Vert > R} \frac{\vert f(x) \vert}{(1+\Vert x \Vert)^\gamma} = 0$ (see \cite[Notation~(v)]{NeufeldSchmocker2024}).

Moreover, we denote by $\mathcal{B}(\mathbb{R}^d)$ the Borel $\sigma$-algebra of $\mathbb{R}^d$ and by $\mathds{1}_E: \mathbb{R}^d \rightarrow \lbrace 0,1 \rbrace$ the indicator function of $E \in \mathcal{B}(\mathbb{R}^d)$. In addition, for a Borel measure $\mu: \mathcal{B}(\mathbb{R}^d) \rightarrow [0,\infty]$, we define $L_2(\mathbb{R}^d,\mathcal{B}(\mathbb{R}^d),\mu)$ as the vector space of (equivalence classes of) Borel-measurable functions $f: \mathbb{R}^d \rightarrow \mathbb{R}$ such that
\begin{equation*}
	\Vert f \Vert_{L_2(\mathbb{R}^d,\mathcal{B}(\mathbb{R}^d),\mu)} := \left( \int_{\mathbb{R}^d} \vert f(x) \vert^2 \mu(dx) \right)^\frac{1}{2} < \infty.
\end{equation*}
Then, $(L_2(\mathbb{R}^d,\mathcal{B}(\mathbb{R}^d),\mu),\Vert \cdot \Vert_{L_2(\mathbb{R}^d,\mathcal{B}(\mathbb{R}^d),\mu)})$ is a Banach space. Similarly, for a given probability space $(\Omega,\mathcal{F},\mathbb{P})$, we denote by $(L_2(\mathbb{P}),\Vert \cdot \Vert_{L_2(\mathbb{P})})$ the Banach space of $\mathcal{F}$-measurable random variables $X: \Omega \rightarrow \mathbb{R}$ such that the norm $\Vert X \Vert_{L^2(\mathbb{P})} := \mathbb{E}\left[ \vert X \vert^2 \right]^{1/2}$ is finite.

Furthermore, we give an overview of the most important quantities used throughout the paper in Appendix~\ref{App}.

\section{\textbf{Derivation of the algorithm}}
\label{section deri}
The deep splitting method has been introduced in \cite{beck2021deep} to solve nonlinear PDEs and later has been extended in \cite{frey2022deep} to nonlinear PIDEs with jumps having \emph{finite} activity. In this section we first outline the extension of the deep splitting method to nonlinear PIDEs with jumps having \emph{infinite} activity and then present its randomized version, where we employ \emph{random} neural networks instead of classical (i.e.\ deterministic, fully trained) ones. The mathematical results for both versions are presented in Section~\ref{section main results}.

Let $T\in(0,\infty)$. For each $d\in \bN$, let $\nu^{d}(dz)$ be a L\'evy measure\footnote{A
L\'evy measure $\nu^d: \mathcal{B}(\mathbb{R}^d) \rightarrow [0,\infty)$ is a Borel measure 
satisfying $\int_{\bR^d}1\wedge\|z\|^2\,\nu^d(dz)<\infty$ and $\nu^d(\{0\})=0$.}
on $\bR^{d}$,
let $\mu^d=(\mu^{d,1},\dots,\mu^{d,d})
\in C([0,T]\times\bR^d;\bR^d)$,
$\sigma^{d}=(\sigma^{d,ij})_{i,j\in\{1,2,\dots,d\}}
\in C([0,T]\times\bR^d;\bR^{d\times d})
$, and
$\eta^{d}=(\eta^{d,1},\dots,\eta^{d,d})\in C([0,T]\times\bR^d\times \bR^{d};\bR^d)$.
Let $g^d\in C(\bR^d)$ and $f^d\in C([0,T]\times \bR^d \times\bR)$.
Moreover, we assume here that all functions 
$\mu^d$, $\sigma^d$, $\eta^d$, $g^d$, and $f^d$
are sufficiently regular.
Then for each $d\in\bN$, 
let $u^d: [0,T] \times \mathbb{R}^d \rightarrow \mathbb{R}$ satisfy for all
$(t,x)\in[0,T)\times\bR^d$ that $u^d(T,x)=g^d(x)$ and
\begin{equation}
	\label{PDE}
	\begin{aligned}
		\frac{\partial}{\partial t} u^d(t,x)
		&+\langle\nabla_x u^d(t,x),\mu^d(t,x)\rangle
		+\frac{1}{2}\operatorname{Trace}
		\Big( \sigma^{d}(t,x)[\sigma^{d}(t,x)]^T\operatorname{Hess}_x u^d(t,x) \Big)
		\\
		& +f^d\big( t,x,u^d(t,x)\big)
		\\
		&
		+\int_{\R^{d}}\left(u^d(t,x+\eta^{d}_t(z,x))-u^d(t,x)
		-\langle\nabla_x u^d(t,x),\eta^{d}_t(x,z)\rangle\right)\,\nu^{d}(dz)
		=0.
	\end{aligned}
\end{equation}

\subsection{Temporal discretization}
Let $d,N\in \bN$, and let $\{t_n\}_{n=0}^N\subseteq [0,T]$ 
be the partition of $[0,T]$ satisfying
$$
t_n=nT/N, \quad n\in\{N,N-1,\dots,1,0\}.
$$
Then by \eqref{PDE}, 
we notice for all $n\in\{N-1,N-2,\dots,1,0\}$, $x\in\bR^d$, 
and $t\in[t_n,t_{n+1})$ that
\begin{equation}
	\begin{aligned}
		u^d(t,x) & = u^d(t_{n+1},x)
		+\int_t^{t_{n+1}}
		f^d\big(s,x,u^d(s,x)\big)
		\,ds \\
		& \quad\quad
		+\int_t^{t_{n+1}}\left[\langle\nabla_x u^d(s,x),\mu^d(s,x)\rangle
		+\frac{1}{2}\operatorname{Trace}
		\left(\sigma^{d}(s,x)[\sigma^{d}(s,x)]^T\operatorname{Hess}_x u^d(s,x)\right)\right]ds \\
		& \quad\quad +\int_t^{t_{n+1}}\int_{\R^{d}}\left(u^d(s,x+\eta^{d}_s(x,z))-u^d(s,x)
		-\langle\nabla_x u^d(s,x),\eta^{d}_s(x,z)\rangle\right)\,\nu^{d}(dz)\,ds.
	\end{aligned}
	\label{approx u 1}
\end{equation}
This suggests for $N$ large enough that for all $n\in\{N-1,N-2,\dots,1,0\}$, $x\in\bR^d$, 
and $t\in[t_n,t_{n+1})$,
\begin{equation}
	\begin{aligned}
		u^d(t,x) & \approx u^d(t_{n+1},x)
		+(t_{n+1}-t_n)
		f^d\big(t_{n+1},x,u^d(t_{n+1},x)\big) \\
		& \quad\quad +\int_t^{t_{n+1}}\left[\langle\nabla_x u^d(s,x),\mu^d(s,x)\rangle
		+\frac{1}{2}\operatorname{Trace}
		\big(\sigma^{d}(s,x)[\sigma^{d}(s,x)]^T\operatorname{Hess}_x u^d(s,x)\big)\right]ds \\
		& \quad\quad +\int_t^{t_{n+1}}\int_{\R^{d}}\left(u^d(s,x+\eta^{d}_s(x,z))-u^d(s,x)
		-\langle\nabla_x u^d(s,x),\eta^{d}_s(x,z)\rangle\right)\,\nu^{d}(dz)\,ds.
	\end{aligned}
	\label{approx u 2}
\end{equation}
Furthermore, let $V^d_N(T,x)=g^d(x)$ for all $x\in\bR^d$, 
and for each $n\in\{N-1,N-2,\dots,1,0\}$ 
let $V^d_n\in C^{1,2}([t_n,t_{n+1})\times\bR^d)$ satisfy 
for all $x\in\bR^d$ and $t\in[t_n,t_{n+1})$ that 
\begin{equation}
	\begin{aligned}
		V^d_n(t,x) = & V^d_{n+1}(t_{n+1},x) 
		 +(t_{n+1}-t_n) f^d\big(t_{n+1},x,V^d_{n+1}(t_{n+1},x)\big) \\
		& +\int_t^{t_{n+1}}\left[\langle\nabla_x V^d_n(s,x),\mu^d(s,x)\rangle
		+\frac{1}{2}\operatorname{Trace}
		\big(\sigma^{d}(s,x)[\sigma^{d}(s,x)]^T
		\operatorname{Hess}_x V^d_n(s,x)\big)\right]ds \\
		& + \int_t^{t_{n+1}}\int_{\R^{d}}\left(V^d_n(s,x+\eta^{d}_s(x,z))-V^d_n(s,x)
		-\langle\nabla_x V^d_n(s,x),\eta^{d}_s(x,z)\rangle\right)\,\nu^{d}(dz)\,ds.
	\end{aligned}
	\label{approx u 4}
\end{equation}
For the existence, uniqueness, and regularity results of PIDEs in the form of
\eqref{approx u 4}, we refer to \cite{GW2021,NW2022}.
Then \eqref{approx u 1}--\eqref{approx u 4} suggest that for $N$ large enough
it holds for all $n\in\{N-1,N-2,\dots,1,0\}$ that
\begin{equation*}
	u^d(t_n,x)\approx V^d_n(t_n,x), \quad x\in\bR^d.
\end{equation*}

\subsection{Feynman-Kac representation}
Let $(\Omega,\cF,\bF,\bP)$ be a filtered probability space equipped with filtration 
$\bF:=(\mathcal{F}_t)_{t\in[0,T]}$
satisfying the usual conditions.
For each $d\in\bN$, let 
$W^{d,m}=(W^{d,m,1},\dots,W^{d,m,d})
:[0,T]\times \Omega\to \bR^d$, $m\in\bN_0$, be independent 
$\bR^d$-valued standard $\bF$-Brownian motions, 
and let $\pi^{d,m}(dz,ds)$, $m\in\bN_0$, be independent
$\bF$-Poisson random measures on $\bR^{d}\times [0,T]$ 
with corresponding identical compensator $\nu^{d}(dz)\otimes ds$, 
and denote by
$
\tilde{\pi}^{d,m}(dz,ds)
:=\pi^{d}(dz,ds)-\nu^{d}(dz)\otimes ds
$ 
the compensated Poisson random measures of $\pi^{d,m}(dz,ds)$.
In addition, for each $d\in\bN$ let $\xi^{d,m}$, $m\in\bN_0$,
be independent $\cF_0$-measurable 
square-integrable random variables. 
Then for each $d\in\bN$, let $(X^{d,m}_t)_{t\in[0,T]}:\Omega\times[0,T]\to\bR^d$, $m\in\bN_0$,
be independent c\`adl\`ag $\bF$-adapted stochastic processes 
satisfying for all $m\in\bN_0$
and, $\mathbb{P}$-a.s., for all $t\in[0,T]$ that
\begin{equation}                                                                             
	\label{SDE deri}
	X^{d,m}_{t}
	=\xi^{d,m}+\int_0^t\mu^d\left(s,X^{d,m}_{s-}\right)\,ds
	+\int_0^t\sigma^{d}\left(s,X^{d,m}_{s-}\right)\,dW^{d,m}_s
	+\int_0^t\int_{\bR^{d}}\eta^{d}_s\left(X^{d,m}_{s-},z\right)
	\,\tilde{\pi}^{d,m}(dz,ds).
\end{equation}
For the existence, uniqueness, and regularity results of SDE \eqref{SDE deri}, 
we refer to \cite[Section~3]{kunita2004stochastic}, 
and \cite[Section~3.1]{rong2006theory}.

To ease notation, for every $d\in\bN$ we define a operator
$F^d:C(\bR^d) \to C([0,T]\times\bR^d)$ by 
\begin{equation*}
	[0,T]\times\bR^d\ni (t,x) \quad \mapsto \quad (F^d\circ v)(t,x):=f^d\big(t,x,v(x)\big),\quad
	v\in C(\bR^d).
\end{equation*}
Then by the Feynman-Kac formula 
(see, e.g., \cite[Theorem~17.4.10]{CE}, or \cite[Theorem~285]{rong2006theory} 
together with the Markov property of
$(X^{d,0}_t)_{t \in [0,T]}$ (see, e.g., \cite[Theorem~17.2.3]{CE}),
one can show for all $d,N\in\bN$ and $n\in\{N-1,N-2,\dots,1,0\}$ that
\begin{equation}
	\label{Markov deri}
	\begin{aligned}
		& \bE\left[
		V^d_{n+1}\left(t_{n+1},X^{d,0}_{t_{n+1}}\right)
		+(t_{n+1}-t_n)\left(F^d\circ V^d_{n+1}\right)
		\left(t_{n+1},X^{d,0}_{t_{n+1}}\right)
		\Big|\cF_{t_n}
		\right] \\
		& \quad\quad =\bE\left[
		V^d_{n+1}\left(t_{n+1},X^{d,0}_{t_{n+1}}\right)
		+(t_{n+1}-t_n)\left(F^d\circ V^d_{n+1}\right)
		\left(t_{n+1},X^{d,0}_{t_{n+1}}\right)
		\Big|X^{d,0}_{t_n}
		\right] \\
		& \quad\quad = V^d_n\left(t_n,X^{d,0}_{t_n}\right).
	\end{aligned}
\end{equation}

\subsection{Time discretization of SDEs with jumps}
\label{subsection time dis}

Since we aim to obtain an implementable numerical algorithm
to solve high dimensional PIDEs, we need an appropriate approximation
of $(X^{d,0}_t)_{t\in[0,T]}$ which can be explicitly simulated. 
To this end, we recall the following approximation procedures in
\cite[Sections 4.2 and 4.4]{gonon2021deep}.
For each $d,N\in\bN$,
let $(\cX^{d,m,N}_s)_{s\in[0,T]}: 
[0,T]\times\Omega\to \bR^d$, $m\in\bN_0$, satisfy for all 
$n\in\{N-1,N-2,\dots,1,0\}$, $s\in (t_n,t_{n+1}]$ that
$\cX^{d,m,N}_{0}=\xi^{d,m}$ and
\begin{equation}
	\label{discrete Euler}
	\begin{aligned}
		\cX^{d,m,N}_{s} & = \cX^{d,m,N}_{t_n}
		+\mu^d\Big(t_n,\cX^{d,m,N}_{t_n}\Big)
		\cdot(s-t_n)
		+\sigma^{d}\Big(t_n,\cX^{d,m,N}_{t_n}\Big)
		\left(W^{d,m}_s-W^{d,m}_{t_n}\right)\nonumber\\
		& \quad\quad
		+\int_{t_n}^s\int_{\bR^{d}}
		\eta^{d}_{t_n}\Big(\cX^{d,m,N}_{t_n},z\Big)
		\,\tilde{\pi}^{d,m}(dz,dr).
	\end{aligned}
\end{equation}
Note that the last term of the stochastic integral with respect to 
the compensated Poisson random measure in \eqref{discrete Euler} 
typically cannot be simulated directly. 
Hence, by following \cite[Sections 4.2 and 4.4]{gonon2021deep}, 
we propose the following truncation.
For each $d\in\bN$ and $\delta\in(0,1)$ we define the set $A^{d}_\delta$
by
\begin{equation*}
	A^{d}_\delta:=\{z\in\bR^{d}:\|z\|\geq \delta\},
\end{equation*}
and also define a probability measure $\nu_{\delta}^{d}$ 
on $(\bR^{d},\cB(\bR^{d}))$ by
\begin{equation*}
	\nu_{\delta}^{d}(B)
	:=\frac{\nu^{d}(B\cap A_\delta^{d})}{\nu^{d}(A_\delta^{d})},
	\quad B\in\cB(\bR^{d}).
\end{equation*}
Then for each $d,N,\cM\in\bN$ and $\delta\in(0,1)$ 
let $V^{d,m,N,\delta,\cM}_{i,j}$, $m\in\bN_0$,
$i=0,1,2,\dots,N$, $j=1,2,\dots,\cM$, 
be independent $\bR^{d}$-valued random variables with
identical distribution $\nu_\delta^{d}$, independent of 
$(W^{d,m},\pi^{d,m})_{(d,m)\in \bN\times\bN_0}$.
For each $d,N,\cM\in\bN$, and $\delta\in(0,1)$, let
$(\cX^{d,m,N,\delta,\cM}_{s})_{s\in[0,T]}: 
[0,T]\times\Omega\to \bR^d$, $m\in\bN_0$,  
be measurable functions satisfying that $\cX^{d,m,N,\delta,\cM}_{0}=\xi^{d,m}$ 
and, $\mathbb{P}$-a.s., for all $n\in\{N-1,N-2,\dots,1,0\}$ and $s\in(t_n,t_{n+1}]$ that
\begin{equation}
	\label{def time discretization}
	\begin{aligned}
		\cX^{d,m,N,\delta,\cM}_{s} & = \cX^{d,m,N,\delta,\cM}_{t_n} +
		\mu^d\left( {t_n,\cX^{d,m,N,\delta,\cM}_{t_n}} \right)(s-t_n) + \sigma^{d}\left( {t_n,\cX^{d,m,N,\delta,\cM}_{t_n}} \right)
		(W^{d,m}_s-W^{d,m}_{t_n}) \\
		& \quad\quad +\int_{A_\delta^{d}}
		\eta^{d}_{t_n}\left( \cX^{d,m,N,\delta,\cM}_{t_n}, z \right)
		\,\pi^{d,m}(dz,(s,t_n]) \\
		& \quad\quad -\frac{(s-t_n)\nu^{d}(A_\delta^{d})}{\cM}\sum_{j=1}^{\cM}
		\eta^{d}_{t_n}\left( \cX^{d,m,N,\delta,\cM}_{t_n},
		V^{d,m,N,\delta,\cM}_{n,j} \right).
	\end{aligned}
\end{equation}
Next, for every $d,N,\cM\in\bN$ and $\delta\in(0,1)$ 
we define Borel functions $\cV^{d}_n:\bR^d\to\bR$, 
$n\in\{N,N-1,\dots,1,0\}$, recursively, 
by setting for all $x\in\bR^d$ that $\cV^{d}_N(x)=g^d(x)$ and
\begin{equation}
	\label{def cV argmin}
	\cV^{d}_n(x)
	:=
	\bE\Big[
	\cV^{d}_{n+1}\big(\cX^{d,0,N,\delta,\cM}_{t_{n+1}}\big)
	+(t_{n+1}-t_n)f^d\Big(t_{n+1},\cX^{d,0,N,\delta,\cM}_{t_{n+1}},
	\cV^{d}_{n+1}\big(\cX^{d,0,N,\delta,\cM}_{t_{n+1}}\big)\Big)
	\Big|\cX^{d,0,N,\delta,\cM}_{t_{n}}=x
	\Big]
\end{equation}
for $n\in\{N-1,N-2,\dots,1,0\}$.
This implies for all $d,N,\cM\in\bN$, $\delta\in(0,1)$, 
and $n\in\{0,1,\dots,N-1\}$ that
\begin{equation}
	\begin{aligned}
	    &
		\cV^{d}_n\big(\cX^{d,0,N,\delta,\cM}_{t_{n}}\big)
		\\
		&
		= \bE\left[
		\cV^{d}_{n+1}\big(\cX^{d,0,N,\delta,\cM}_{t_{n+1}}\big)
		+(t_{n+1}-t_n)f^d\Big(t_{n+1},\cX^{d,0,N,\delta,\cM}_{t_{n+1}},
		\cV^{d}_{n+1}\big(\cX^{d,0,N,\delta,\cM}_{t_{n+1}}\big)\Big)
		\Big|\cX^{d,0,N,\delta,\cM}_{t_{n}}
		\right] \\
		&
		= \bE\left[
		\cV^{d}_{n+1}\big(\cX^{d,0,N,\delta,\cM}_{t_{n+1}}\big)
		+(t_{n+1}-t_n)f^d\Big(t_{n+1},\cX^{d,0,N,\delta,\cM}_{t_{n+1}},
		\cV^{d}_{n+1}\big(\cX^{d,0,N,\delta,\cM}_{t_{n+1}}\big)\Big)
		\Big|\cF_{t_n}
		\right].
	\end{aligned}
	\label{markov Euler intro}
\end{equation}
Combining \eqref{Markov deri} and \eqref{markov Euler intro} suggests for large enough $\cM,N \in \mathbb{N}$ and small enough $\delta \in (0,1)$ that for all $d \in \bN$ and $n\in\{N,N-1,...,1,0\}$ it holds that
\begin{equation*}
	V^d_n\left(t_n,X^{d,0}_{t_n}\right) \approx 
	\cV^{d}_n\left(\cX^{d,0,N,\delta,\cM}_{t_{n}}\right).
\end{equation*}
Moreover, for every $d,N,\cM\in\bN$, $\delta\in(0,1)$ and $n\in\{N-1,N-2,\dots,1,0\}$, 
the $L_2$-projection property of conditional expectations 
(see, e.g., \cite[Theorem~3.1.14]{Kry2002}) and as $\cV^d_n(\cdot)$ is continuous imply that
\begin{equation}
\label{def cV}
	\begin{aligned}
		\cV^d_n(\cdot) 
		& = \mathop{\arg\min}_{v \in C(\mathbb{R}^d)}
		\bE\bigg[
		\Big\vert
		\cV^d_{n+1}\left(\cX^{d,0,N,\delta,\cM}_{t_{n+1}}\right)
		\\
		& \quad\quad
		+(t_{n+1}-t_n)\left(F^d\circ \cV^d_{n+1}\right)
		\left(t_{n+1},\cX^{d,0,N,\delta,\cM}_{t_{n+1}}\right)
		-v\left(\cX^{d,0,N,\delta,\cM}_{t_n}\right)
		\Big\vert^2
		\bigg].
	\end{aligned}
\end{equation}

\subsection{Deep Splitting Algorithm with Deterministic Neural Networks}
\label{section Deter NN}
In order to introduce the deep splitting algorithm with deterministic neural networks, we fix some large $\cM\in\bN$, small $\delta\in(0,1)$, 
$d,N,K \in \mathbb{N}$,
$k \in \mathbb{N} \cap [3,\infty)$, a weakly differentiable function $\rho: \mathbb{R} \rightarrow \mathbb{R}$, and let $r=(1+kK-K)(K+1)+K(d+1)$. Then, for each $\theta=(\theta_1,\dots,\theta_r)\in\bR^r$, $i,j\in\bN$, $v\in\bN_0$
with $v+i(j+1)\leq r$, we define the function
$A^{\theta,v}_{i,j}:\bR^i\to\bR^j$ for every $x=(x_1,\dots,x_i)\in\bR^i$ by
\begin{equation*}
	A^{\theta,v}_{i,j}(x):=
	\begin{pmatrix}
		\theta_{v+1} & \theta_{v+2} & \dots & \theta_{v+i}\\
		\theta_{v+i+1} & \theta_{v+i+2} & \dots & \theta_{v+2i}\\
		\theta_{v+2i+1} & \theta_{v+2i+2} & \dots & \theta_{v+3i}\\
		\vdots & \vdots & \ddots & \vdots\\
		\theta_{v+i(j-1)+1} & \theta_{v+i(j-1)+2} & \dots & \theta_{v+ij}
	\end{pmatrix}
	\begin{pmatrix}
		x_1 \\ x_2\\ x_3\\ \vdots \\ x_i
	\end{pmatrix}
	+\begin{pmatrix}
		\theta_{v+ij+1} \\ 
		\theta_{v+ij+2} \\ 
		\theta_{v+ij+3} \\
		\vdots\\
		\theta_{v+ij+j}
	\end{pmatrix}.
\end{equation*}
Moreover, let
$\bV^d_n:\bR^r\times\bR^d\to\bR$, $n\in\{N,N-1,\dots,1,0\}$, 
be the functions satisfying
for every $n\in\{N-1,N-2,\dots,1,0\}$, $\theta\in\bR^r$, and $x\in\bR^d$ that
$\bV^d_N(\theta,x)=g^d(x)$ and
\begin{equation}
\bV^d_n(\theta,x):=
\Big( A^{\theta,(k-1)K(K+1)+K(d+1)}_{K,1}\circ \rho \circ 
A^{\theta,(k-2)K(K+1)+l(d+1)}_{K,K} \circ \cdots \circ
\rho \circ A^{\theta,K(d+1)}_{K,K} \circ 
\rho \circ A^{\theta,0}_{d,K}
\Big)(x),
\label{function NN deri}
\end{equation}
where $\rho: \mathbb{R} \rightarrow \mathbb{R}$ is applied componentwise.
Note that for each $n\in\{N-1,N-2,\dots,1,0\}$, 
the function $\bV^d_n$ defined by \eqref{function NN deri}
is the network function of a deterministic (fully trained)
neural network with $k+1$ layers ($1$ input layer with $d$ neurons, 
$k-1$ hidden layers with $K$ neurons respectively,
and $1$ output layer with $1$ neuron) and 
$\rho: \mathbb{R} \rightarrow \mathbb{R}$ as activation function.

To find a suitable parameter $\theta$, we apply 
the following stochastic gradient descent algorithm. Let
$\{\Theta^m_n\}_{m=0}^\infty:\bN_0\times\Omega\to\bR^r$,
$n\in\{0,1,\dots,N\}$, be discrete-time stochastic processes 
defined for every $n\in\{0,1,\dots,N\}$, $m\in\bN_0$, and $\omega \in \Omega$
by\footnote{
	Instead of the plain vanilla gradient descent, one could also consider
	any suitable stochastic gradient descent method
	such as, e.g., the Adam algorithm (see \cite{KingmaBa2015}).}
\begin{equation*}
	\Theta^{m+1}_n(\omega)=\Theta^m_n(\omega)
	-\beta \cdot \nabla_\theta\phi^m_n\left(\Theta^m_n(\omega),\omega\right),
\end{equation*}
where
\begin{equation}
	\begin{aligned}
		\phi^m_n(\theta,\omega) & := \frac{1}{J} \sum_{j=1}^J \Bigg|\bV^d_{n}
		\left(\theta,\cX^{d,Jm+j,N,\delta,\cM}_{t_n}(\omega)\right) 
		- \bigg[
		\bV^d_{n+1}\left( \Theta^m_{n+1}(\omega),
		\cX^{d,Jm+j,N,\delta,\cM}_{t_{n+1}}(\omega) \right) \\
		& \quad\quad\quad\quad\quad 
		+(t_{n+1}-t_n)
		\left( F^d\circ \bV^d_{n+1}\left(\Theta^M_{n+1}(\omega),\cdot\right)\right)
		\left(t_{n+1},\cX^{d,Jm+j,N,\delta,\cM}_{t_{n+1}}(\omega)\right) 
		\bigg] \Bigg|^2,
	\end{aligned}
	\label{error theta deri}
\end{equation}
for some $J \in \mathbb{N}$ denoting the sample size and some $\beta \in (0,\infty)$ representing the learning rate. Then, 
for sufficiently large $J,N,M,\cM\in\bN$ and sufficiently small $\beta \in (0,\infty)$ and $\delta \in (0,1)$, we obtain for every $n\in\{N,N-1,\dots,1,0\}$ the following approximations: $\bV^d_N(\cdot) = g^d(\cdot)$ and
\begin{equation*}
	\bV^d_n\left(\Theta^M_n,\cX^{d,0,N,\delta,\cM}_{t_n}\right) \approx u^d\left(t_n,X^{d,0}_{t_n}\right).
\end{equation*}

\subsection{Deep Splitting Algorithm with Random Neural Networks}
\label{section RNN}
Fixing a large $\cM,N\in\bN$, a small $\delta\in(0,1)$, and some $d \in \mathbb{N}$, we aim to find suitable approximations of the functions $V^d(t_n,\cdot)$, $n\in\{N,N-1,\dots,1,0\}$, by using \emph{random} neural networks instead of deterministic (fully trained) neural networks. Roughly speaking, a random neural network is a single-hidden-layer feedforward neural network whose weights and biases inside the activation function are randomly initialized and only the linear readouts need to be trained. To this end, we follow \cite{NeufeldSchmocker2023} and introduce random neural networks as $\overline{C_b(\mathbb{R}^d)}^\gamma$-valued random variables (see \cite[Section~1.2]{Hytoenen2016} for an introduction to Banach space-valued random variables).

More precisely, we assume that $(\Omega,\mathcal{F},\mathbb{P})$ supports two i.i.d.~sequences of random variables $(A_k)_{k \in \mathbb{N}}: \Omega \rightarrow \mathbb{R}^d$ and $(B_k)_{k \in \mathbb{N}}: \Omega \rightarrow \mathbb{R}$, which are independent of the Brownian motions $(W^{d,m})_{m\in\bN_0}$, 
the componensated Poisson random measures $(\tilde{\pi}^{d,m})_{m\in\bN_0}$, and the jumps $(V^{d,m,N,\delta,\mathcal{M}}_{i,j})_{m \in \mathbb{N}_0, \, i = 0,\dots,N, \, j = 1,\dots,\mathcal{M}}$, and thus also from the time discretization $(\cX^{d,j,N,\cM,\delta})_{m \in \mathbb{N}_0, \, i = 0,\dots,N, \, j = 1,\dots,\mathcal{M}}$.

\begin{definition}
	\label{DefRN}
	For $\gamma \in (0,\infty)$, $\rho \in \overline{C_b(\mathbb{R})}^\gamma$, $K \in \mathbb{N}$, and $y := (y_1,\dots,y_k)^\top \in \mathbb{R}^K$, we call a map of the form\footnote{The notation $y_k \rho\left( A_k(\omega)^\top \cdot - B_k(\omega) \right)$ refers to the function $\mathbb{R}^d \ni x \mapsto y_k \rho\left( A_k(\omega)^\top x - B_k(\omega) \right) \in \mathbb{R}$.}
	\begin{equation}
		\label{EqDefRN}
		\Omega \ni \omega \quad \mapsto \quad \mathscr{V}^d(y,\cdot) := \sum_{k=1}^K y_k \rho\left( A_k(\omega)^\top \cdot - B_k(\omega) \right) \in \overline{C_b(\mathbb{R}^d)}^\gamma
	\end{equation}
	a \emph{random neural network}, where $K \in \mathbb{N}$ denotes the number of neurons and $\rho \in \overline{C_b(\mathbb{R})}^\gamma$ represents the activation function. Moreover, $A_1,\dots,A_K: \Omega \rightarrow \mathbb{R}^d$ are the \emph{random weights}, $B_1,\dots,B_K: \Omega \rightarrow \mathbb{R}$ are the \emph{random biases}, whereas $y_1,\dots,y_K \in \mathbb{R}$ are the \emph{linear readouts} that need to be trained.
\end{definition}

For the implementation of random neural networks, we fix some $K \in \mathbb{N}$ and randomly initialize the weights and biases $(A_k,B_k)_{k=1,\dots,K}$. Then, by following \cite[Section~5.1]{NeufeldSchmocker2023}, we use the randomized least square method recursively to find the optimal linear readouts, i.e.~we set $\mathscr{V}^d_N(y,\cdot) := g^d(\cdot)$ for all $y \in \mathbb{R}^K$, and for every $n\in\{N-1,N-2,\dots,1,0\}$ that
\begin{equation}
	\label{EqLeastSquares}
	\begin{aligned}
		\Upsilon_n(\omega) & := \argmin_{y \in \mathbb{R}^K} \frac{1}{J} \sum_{j=1}^J \Bigg\vert \mathscr{V}^d_n\left( y, \cX^{d,j,N,\cM,\delta}_{t_n}(\omega) \right) - \bigg[ \mathscr{V}^d_{n+1}\left( \Upsilon_{n+1}(\omega), \cX^{d,j,N,\cM,\delta}_{t_{n+1}}(\omega) \right) \\
		& \quad\quad\quad\quad\quad\quad\quad\quad\quad\quad\quad + (t_{n+1} - t_n) \left( F^d \circ \mathscr{V}^d_{n+1}\left( \Upsilon_{n+1}(\omega), \cdot \right) \right)\left(t_{n+1}, \cX^{d,j,N,\cM,\delta}_{t_{n+1}}(\omega) \right) \bigg] \Bigg\vert^2.
	\end{aligned}
\end{equation}
Then, for sufficiently large $K \in \mathbb{N}$, we obtain for every $n\in\{N,N-1,\dots,1,0\}$ the following approximations: $\mathscr{V}^d_N(y,\cdot) := g^d(\cdot)$ for all $y \in \mathbb{R}^K$ and
\begin{equation*}
	\mathscr{V}^d_n\left(\Upsilon_n,\cX^{d,0,N,\cM,\delta}_{t_n}\right) \approx u^d\left(t_n,X^{d,0}_{t_n}\right).
\end{equation*}

\subsection{Pseudocode}
\label{section time discr}
In this subsection, we summarize the deep splitting algorithm 
with both deterministic (fully trained) 
neural networks and random neural networks in Algorithm~\ref{AlgDS} 
in a pseudocode form.

\begin{algorithm}
	\scriptsize{
		\caption{Deep Splitting Algorithm}
		\label{AlgDS}
		\begin{algorithmic}[1]	
			\State Fix an error tolerance $\widetilde{\varepsilon} \in (0,1)$;
			\State Choose $\vartheta > 0$ large enough such that $\frac{\overline{C}}{\vartheta^{q-2}} \bE\big[\big(d^p+\|\xi^{d,0}\|^2\big)^{q/2}\big] \leq \widetilde{\varepsilon}/4$, see \eqref{EqDefCBar} for $\overline{C} > 0$;
			\State Choose $N \in \mathbb{N}$ large enough such that $\frac{1}{N} \cdot 64 \widehat{C} d^p \left( d^p + \bE\left[ \Vert \xi^{d,0} \Vert^2 \right] \right) \leq \widetilde{\varepsilon}/12$, see \eqref{EqDefCHat} for $\widehat{C} > 0$;
			\State Choose $\delta > 0$ small enough such that $\delta^q d^p \cdot 64 \widehat{C} d^p \left( d^p + \bE\left[ \Vert \xi^{d,0} \Vert^2 \right] \right) \leq \widetilde{\varepsilon}/12$, see \eqref{EqDefCHat} for $\widehat{C} > 0$;
			\State Choose $\cM \in \mathbb{N}$ large enough such that $\frac{d^p}{\delta^2 \cM} \cdot 64 \widehat{C} d^p \left( d^p + \bE\left[ \Vert \xi^{d,0} \Vert^2 \right] \right) \leq \widetilde{\varepsilon}/12$, see \eqref{EqDefCHat} for $\widehat{C} > 0$;
			\State Let $K \in \mathbb{N}$ be large enough such that there exists $Y_n: \Omega \rightarrow \mathbb{R}^K$ with $64 \mathbb{E}\Big[ \big\vert \mathcal{V}^d_n(\cX^{d,0,N,\delta,\cM}_{t_n}) - \mathscr{V}^d_n(Y_n, \cX^{d,0,N,\delta,\cM}_{t_n}) \big\vert^2 \Big] \leq \widetilde{\varepsilon}/4$;
			\State \% Note that such $K \in \mathbb{N}$ exists according to Theorem~\ref{ThmMainRN}~\ref{ThmMainRN2}.
			\State Choose $J \in \mathbb{N}$ large enough such that $2 C_0 \vartheta^2 \frac{(\ln(J)+1) K}{J} \leq \widetilde{\varepsilon}/4$, see \eqref{EqDefC0} for $C_0 > 0$;\\
					
			\Function{SDE}{$n,J,N,t$}			
			\State \% Generate paths of the numerical solution of SDE \eqref{SDE deri}
			\State \% Input: simulation time steps $n$; batchsize $J$; total time steps $N$; initial time $t$; initial value $x$
			
			\State Generate $J$ realizations $\xi(j)$, $j \in \lbrace 1,...,J \rbrace$, of i.i.d.\,random variables distributed as the (random) initial value $\xi^{d,0}: \Omega \rightarrow \mathbb{R}^d$;
			
			\For{$j\leftarrow 1$ to $J$}
			
			\State $\mathcal{X}(j,0) \leftarrow \xi(j)$;
			\State \multiline{Generate $n$ realizations $W(k)\in\bR^d$,
				$k \in \lbrace 0,\dots,n-1 \rbrace$, of i.i.d.\,standard normal random vectors;}
			\State \multiline{Generate $n$ realizations $P(k)\in\bN_0$,
				$k \in \lbrace 0,\dots,n-1 \rbrace$, of i.i.d.\,Poisson random variables with
				parameter $\nu^d(A^d_\delta) T/N$;}
			\For{$k\leftarrow 0$ to $n-1$}
			\State \multiline{Generate $P(k)$ realizations $Z(l)\in\bR^d$,
				$l \in \lbrace 1,\dots,P(k) \rbrace$,
				of i.i.d.\,random vectors with distribution 
				$\nu^d_\delta(dz)$;}
			\State \multiline{Generate $\cM$ realizations $V(l)\in\bR^d$, 
				$l \in \lbrace 1,\dots,\cM \rbrace$, of i.i.d.\,random vectors 
				with distribution $\nu^d_\delta(dz)$;}
			\State \multiline{
				$
				\mathcal{X}(j,k+1)\leftarrow \mathcal{X}(j,k)+\mu^d(t+kT/N,\mathcal{X}(j,k)) T/N
				+\sigma^d(t+kT/N,\mathcal{X}(j,k))
				\sqrt{T/N} \cdot W(k)
				$
				\\
				\qquad\quad
				$
				\quad\quad\quad\quad\;\;+\sum_{i=1}^{P(k)}\eta^d_{t+kT/N}(\mathcal{X}(j,k),Z(i))
				-(N \cM)^{-1} T \nu^d(A^d_\delta)
				\sum_{i=1}^{\cM}\eta^d_{t+kT/N}(\mathcal{X}(j,k),V(i)),
				$
				\\
				cf.~\eqref{def time discretization};
			}
			\EndFor
			
			\EndFor\\
			
			\quad\, \Return $\mathcal{X}$
			
			\EndFunction\\
			
			\Function{DS\_NN}{$J,N,M,t$}
			\State \% Deep Splitting Algorithm with Deterministic Neural Networks
			\State \% Input: batchsize $J$; total time steps $N$; learning steps $M$; initial time $t$; initial value $x$
			
			\State Set $\mathbb{V}_N(\theta,\cdot) = g^d(\cdot)$ for all $\theta \in \mathbb{R}^r$;			
			\For{$n \leftarrow N-1$ to $0$}
			\State Set $\Theta(N,m) \leftarrow 0$ for all $m = 1,\dots,M$;
			\State Generate paths $\mathcal{X} \leftarrow \text{SDE}(n,J,N,t)$;
			\For{$m \leftarrow 0$ to $M-1$}
			\State \multiline{$\text{Grad} \leftarrow J^{-1} \sum_{j=1}^J \Big( 2\nabla_{\theta}\bV^d_n(\Theta(n,m),\mathcal{X}(j,n)) \cdot \big\{\bV^d_n(\Theta(n,m),\mathcal{X}(j,n))$\\
				\quad\quad\quad\quad\quad $- \big[\bV^d_{n+1}\big(\Theta(n+1,M),\mathcal{X}(j,n+1)\big)
				+T/N \big(F^d\circ \bV^d_{n+1}\big(\Theta(n+1,M),\cdot\big)\big)
				\big(t_{n+1},\mathcal{X}(j,n+1)\big)\big]\big\} \Big);
				$}
			
			\State $\Theta(n,m+1) \leftarrow \Theta(n,m) - \beta \cdot \text{Grad}$;
			\EndFor
			\EndFor\\
			
			\quad\, \Return Networks at intermediate time steps $\bV^d_n\left(\Theta(n,M),\cdot\right)$ for $n\in\{N,N-1,\dots,1,0\}$;\\
			
			\quad\, \Return Approximation $\bV^d_0\left(\Theta(n,M),\xi^{d,0}\right)$ of the solution $u^d(0,\xi^{d,0})$;
			
			\EndFunction\\
			
			\Function{DS\_RN}{$J,N,t$}
			\State \% Deep Splitting Algorithm with Random Neural Networks
			\State \% Input: batchsize $J$; total time steps $N$; initial time $t$; initial value $x$
			
			\State Generate paths $\mathcal{X} \leftarrow \text{SDE}(N,J,N,t)$;
			\State Generate $K$ realizations $A(k) \in \mathbb{R}^d$, $k \in \lbrace 1,\dots,K \rbrace$, of random weights;
			\State Generate $K$ realizations $B(k) \in \mathbb{R}$, $k \in \lbrace 1,\dots,K \rbrace$, of random biases;
			\State Set $\mathscr{V}^d_N(y,\cdot) = g^d(\cdot)$ for all $y \in \mathbb{R}^K$;
			\State Set $\Upsilon(N) = 0$;	
			\For{$n \leftarrow N-1$ to $0$}
			\State Set $R(j,k) \leftarrow \rho\left( A(k)^\top \mathcal{X}(j,n) - B(k) \right)$ for $j=1,\dots,J$ and $k = 1,\dots,K$;
			\State \multiline{Set $Q(j) \leftarrow \mathscr{V}^d_{n+1}\big(\Upsilon(n+1),\mathcal{X}(j,n+1)\big) +T/N \big(F^d\circ \mathscr{V}^d_{n+1}\big(\Upsilon(n+1),\cdot\big)\big) \big(t_{n+1},\mathcal{X}(j,n+1)\big)$ for $j=1,\dots,J$;}
			\State \multiline{Solve the least squares problem $\Upsilon(n) \leftarrow \argmin_{y \in \mathbb{R}^K} \Vert R y - Q \Vert^2$, i.e.~solve the corresponding normal equation \\
				\qquad $R^\top R \Upsilon(n) = R^\top Q$ via Cholesky decomposition and forward/backward substitution (see \cite[Section~2.2.2]{bjorck1996numerical});}
			\EndFor\\
						
			\quad\, \Return Networks at intermediate time steps $\mathscr{V}^d_n(\Upsilon(n),\cdot)$ for $n\in\{N,N-1,\dots,1,0\}$;\\
			
			\quad\, \Return Approximation $\mathscr{V}^d_n(\Upsilon(n),\xi^{d,0})$ of the solution $u^d(0,\xi^{d,0})$;
			
			\EndFunction\\
			
			\State \% In this case, the random neural networks $\mathscr{V}^d_n(\Upsilon(n),\cdot)$, $n\in\{N,N-1,\dots,1,0\}$, satisfy $\mathscr{V}^d_N(\Upsilon(N),\cdot) = g^d(\cdot)$ and
			\begin{equation*}
				\sup_{n\in\{N-1,N-2,\dots,1,0\}} \mathbb{E}\left[ \left\vert u^d\left(t_n,X^{d,0}_{t_n}\right) - \mathscr{V}^d_n\left(\Upsilon_n,\cX^{d,0,N,\delta,\cM}_{t_n}\right) \right\vert^2 \right] \leq \widetilde{\varepsilon},
			\end{equation*}		
			\State \% see also Corollary~\ref{CorMainRN}.	
		\end{algorithmic}
	}
\end{algorithm}

\section{\textbf{Setting and main results}}
\label{section setting main}

\subsection{Assumptions and Setting}
\label{section assump}
In the section, we present our setting and the main
mathematical results for the deep splitting algorithms, see 
Theorem \ref{ThmMain} for fully trained deterministic neural 
networks, and Theorem \ref{ThmMainRN} as well as Corollary 
\ref{CorMainRN} for random neural networks. To this end, we consider the PIDE of the form \eqref{PDE} with coefficients $\mu^d$, $\sigma^d$, and $\eta^d$, functions $f^d$ and $g^d$, as well as L\'evy measure $\nu^d$ introduced in Section~\ref{section deri}, where we now impose the following conditions. 

\begin{assumption}[Lipschitz and linear growth conditions]
	\label{assumption Lip and growth}
	There exists constants $L,p\in(0,\infty)$ satisfying 
	for all $d\in\bN$, $x,y\in\bR^d$, $v,w\in\bR$, and $t,s\in[0,T]$ that
	\begin{equation}                                    \label{assumption Lip f g}
		T|f^d(t,x,v)-f^d(s,y,w)|+|g^d(x)-g^d(y)|
		\leq L^{1/2}\big(T|t-s|^{1/2}+T|v-w|+T^{-1/2}\|x-y\|\big),
	\end{equation}
	\begin{equation}                              \label{assumption Lip mu sigma eta}
		\|\mu^d(t,x)-\mu^d(t,y)\|^2
		+\|\sigma^{d}(t,x)-\sigma^{d}(t,y)\|_F^2
		+\int_{\bR^{d}}\big\|\eta^{d}_t(x,z)-\eta^{d}_t(y,z)\big\|^2
		\,\nu^{d}(dz)
		\leq L\|x-y\|^2,
	\end{equation}
	\begin{equation}                                        
	\label{assumption growth}
		\|\mu^d(t,x)\|^2+\|\sigma^{d}(t,x)\|_F^2
		+\int_{\bR^{d}}\big\|\eta^{d}_t(x,z)\big\|^2\,\nu^{d}(dz)\leq L(d^p+\|x\|^2),
	\end{equation}
	and
	\begin{equation}                                  \label{assumption growth f g}
		|T\cdot f^d(t,x,0)|^2+|g^d(x)|^2\leq L(d^p+\|x\|^2).
	\end{equation}
\end{assumption}

\begin{assumption}[Pointwise Lipschitz and integrability conditions]
	\label{assumption pointwise}
	For each $d\in \bN$ there exits a constant $C_{d}\in(0,\infty)$ such that
	for all $x,x'\in \bR^d$, $t\in[0,T]$, and $z\in\bR^{d}$ 
	\begin{equation}                                        \label{assumption eta}                                     
		\|\eta^{d}_t(x,z)\|^2\leq Ld^p(1\wedge \|z\|^2), 
		\quad \|\eta^{d}_t(x,z)-\eta^{d}_t(x',z)\|^2
		\leq C_{d}\|x-x'\|^2(1\wedge \|z\|^2),
	\end{equation} 
	where $L,p\in(0,\infty)$ are the constants introduced 
	in Assumption \ref{assumption Lip and growth}.
\end{assumption}

\begin{assumption}                                     \label{assumption jacobian}
	For all $d\in\bN$, $(t,x)\in[0,T]\times\bR^d$ and $z\in\bR^d$ we assume that
	$D_x \eta^d_t(x,z)$ exists,
	where $D_x\eta^d_t(x,z)$ denotes the Jacobian matrix of $\eta^d_t(x,z)$
	with respect to $x$.
	Moreover, we assume that there exists a constant
	$\lambda>0$ such that for all $(t,x)\in[0,T]\times\bR^d$, $z\in\bR^d$,
	and $\theta\in[0,1]$  
	\begin{equation}                                     \label{jacobian cond}
		\lambda\leq |\det(I_d+\theta D_x \eta^d_t(x,z))|.
	\end{equation}
\end{assumption}

\begin{assumption}
	There exist constants $C_\eta,q\in(0,\infty)$ satisfying
	for all $d\in \bN$, $\delta\in(0,1)$, and $(t,x)\in[0,T]\times\bR^d$ that
	\begin{equation}                                       \label{int z & q bound}
		\int_{\bR^{d}}(1\wedge\|z\|^2)\,\nu^{d}(dz)\leq C_\eta d^p
		\quad \text{and} \quad
		\int_{\|z\|\leq \delta}\|\eta^{d}_t(x,z)\|^2
		\,\nu^{d}(dz)\leq C_\eta d^p\delta^q(1+\|x\|^2).
	\end{equation}
\end{assumption}

\begin{assumption}[Temporal $1/2$-H\"older condition]
	\label{assumption time Holder}
	There exist constants $L_1,L_2>0$ satisfying 
	for all $d\in \bN$,
	$x\in\bR^d$, $z\in\bR^{d}$, and $(t,t')\in[0,T]^2$ that 
	\begin{equation}                                        \label{Holder mu sigma d}
		\|\mu^d(t,x)-\mu^d(t',x)\|^2+\|\sigma^{d}(t,x)-\sigma^{d}(t',x)\|_F^2
		\leq L_1|t-t'|,
	\end{equation}
	and
	\begin{equation}                                           \label{Holder eta d}
		\int_{\bR^{d}}\|\eta^{d}_t(x,z)-\eta^{d}_{t'}(x,z)\|^2\,\nu^{d}(dz)
		\leq L_2|t-t'|.
	\end{equation}
\end{assumption}

Let $(\Omega,\cF,\bF,\bP)$ be a filtered probability space equipped with filtration 
$\bF:=(\mathcal{F}_t)_{t\in[0,T]}$
satisfying the usual conditions. For each $d\in\bN$, let 
$W^{d}=(W^{d,1},\dots,W^{d,d})
:[0,T]\times \Omega\to \bR^d$, be a 
$\bR^d$-valued standard $\bF$-Brownian motions, 
and let $\pi^{d}(dz,dt)$ be an
$\bF$-Poisson random measure on $\bR^{d}\times [0,T]$ 
with compensator $\nu^{d}(dz)\otimes dt$, 
and denote by
$
\tilde{\pi}^{d}(dz,dt)
:=\pi^{d}(dz,dt)-\nu^{d}(dz)\otimes dt
$ 
the compensated Poisson random measures of $\pi^{d}(dz,dt)$.
For each $d\in\bN$ and $(t,x)\in[0,T]\times\bR^d$, let
$(X^{d,t,x}_{s})_{s\in[t,T]}: 
[t,T]\times\Omega\to \bR^d$ be 
$\cB([t,T])\otimes\cF/\cB(\bR^d)$-measurable 
c\`adl\`ag stochastic processes satisfying for all $s\in[t,T]$, $\mathbb{P}$-a.s., that
\begin{equation}                                               \label{SDE}                              
	X^{d,t,x}_{s}
	=x+\int_t^s\mu^d(r,X^{d,t,x}_{r-})\,dr
	+\int_t^s\sigma^{d}(r,X^{d,t,x}_{r-})\,dW^{d}_r
	+\int_t^s\int_{\bR^{d}}\eta^{d}_r(X^{d,t,x}_{r-},z)
	\,\tilde{\pi}^{d}(dz,dr).
\end{equation}
Note that Assumptions \ref{assumption Lip and growth} and 
\ref{assumption pointwise} guarantee that for each $d\in\bN$,
the SDE in \eqref{SDE} has a unique solution satisfying 
$\bE\Big[\sup_{s\in[t,T]}\big\|X^{d,t,x}_{s}\big\|^2\Big] < \infty$
(see, e.g., \cite[Theorem 3.1+3.2]{kunita2004stochastic} or \cite[Lemma~114+117]{rong2006theory}). 
Moreover, \cite[Equation~(5.111) and Theorem 2.9~(ii)-(iv)]{NW2022} ensure that 
for each $d\in\bN$, PIDE \eqref{PDE} 
has a unique viscosity solution $u^d\in C([0,T]\times\bR^d)$ satisfying for all 
$(t,x)\in[0,T]\times\bR^d$ that
\begin{equation}                                       \label{E integrability}
	\bE\Big[\big|g^d(X^{d,t,x}_{T})\big|\Big]
	+\int_t^T\bE\Big[
	\big|
	f^d(s,X^{d,t,x}_{s},u^d(s,X^{d,t,x}_{s}))
	\big|
	\Big]\,ds<\infty,
\end{equation}
\begin{equation}                                           \label{Feynman Kac u}
	u^d(t,x)=\bE\Big[g^d(X^{d,t,x}_{T})\Big]
	+\int_t^T\bE\Big[f^d(s,X^{d,t,x}_{s},u^d(s,X^{d,t,x}_{s}))\Big]\,ds,
\end{equation}
and
\begin{equation}                                               \label{u est}
	\sup_{x\in\bR^d}\frac{|u^d(t,x)|}{\sqrt{d^p+\|x\|^2}}\leq 2L^{1/2}
	\exp\big\{L^{1/2}(T-t)\big\}.
\end{equation}

Moreover, we assume that the initial values $(\xi^{d,m})_{m\in\bN_0}$ of the SDE \eqref{SDE deri} are sufficiently integrable.

\begin{assumption}
	\label{AssXiLq}
	For some $q\in(2,\infty)$ and every $d \in \mathbb{N}$, let $\xi^{d,m}:\Omega\to\bR^d$, $m\in\bN_0$, be i.i.d.\,$\mathcal{F}_0/\cB(\bR^d)$-measurable random variables such that $\bE\left[\big\|\xi^{d,m}\big\|^q\right]<\infty$.
\end{assumption}

For each $d\in\bN$, 
let $W^{d,m}=(W^{d,m,1},\dots,W^{d,m,d})
:[0,T]\times \Omega\to \bR^d, m\in\bN_0$, be independent 
$\bR^d$-valued standard $\bF$-Brownian motions, let $\pi^{d,m}(dz,dt),m\in \bN_0$, be
independent $\bF$-Poisson random measures on $\bR^{d}\times [0,T]$ 
with identical corresponding compensator $\nu^{d}(dz)\otimes dt$, 
and denote by $\tilde{\pi}^{d,m}(dz,dt) :=\pi^{d,m}(dz,dt)-\nu^{d}(dz)\otimes dt$
the compensated Poisson random measure of $\pi^{d,m}(dz,dt)$ for every $m\in \bN_0$.
For each $d\in\bN$, let $(X^{d,m}_s)_{s\in[0,T]}: 
[0,T]\times\Omega\to \bR^d$, $m\in\bN_0$, be 
$\cB([0,T])\otimes\cF/\cB(\bR^d)$-measurable 
c\`adl\`ag stochastic processes satisfying for all $m\in\bN_0$ and $s\in[0,T]$, $\mathbb{P}$-a.s., that
\begin{equation}                                               \label{SDE d}                              
	X^{d,m}_{s}
	=\xi^{d,m}+\int_0^s\mu^d(r,X^{d,m}_{r-})\,dr
	+\int_0^s\sigma^{d}(r,X^{d,m}_{r-})\,dW^{d,m}_r
	+\int_0^s\int_{\bR^{d}}\eta^{d}_r(X^{d,m}_{r-},z)
	\,\tilde{\pi}^{d,m}(dz,dr).
\end{equation}
If $\bE\left[\|\xi^{d,m}\|^2\right]<\infty$, then Assumptions \ref{assumption Lip and growth} and 
\ref{assumption pointwise} guarantee that for each $d\in\bN$ and $m \in \bN_0$,
the SDE in \eqref{SDE d} has a unique solution satisfying $\bE\Big[\sup_{s\in[0,T]}\big\|X^{d,m}_s\big\|^2\Big] < \infty$ (see, e.g., \cite[Theorem 3.1+3.2]{kunita2004stochastic} or \cite[Lemma~114+117]{rong2006theory}). Moreover, Assumption~\ref{AssXiLq} (with respect to some fixed $q \in (2,\infty)$) additionally ensures that there exists a constant $C > 0$ (depending only on $q$, $T$, and $L$) such that for all $d \in \bN$, $m\in\bN_0$, and $s \in [0,T]$ it holds that
\begin{equation*}
	\bE\Big[ \big\|X^{d,m}_s\big\|^q \Big] \leq C \bE\left[\big(d^p+\big\|\xi^{d,m}\big\|^2\big)^{q/2}\right],
\end{equation*}
see also Corollary~\ref{corollary Lyaponov q} for more details.

In addition, recall that we denote by $(\cX^{d,m,N,\delta,\cM}_s)_{s\in[0,T]}$
the time discritization of SDE \eqref{SDE d}
(see \eqref{def time discretization} in Section \ref{subsection time dis}). Then, we 
notice that by \cite[Lemma 3.6-3.8]{NW2022}, we have the following result.

\begin{remark}
	\label{RemTimeDiscr}
	Note that if Assumptions~\ref{assumption Lip and growth}--\ref{assumption time Holder} hold, then there exists by \cite[Lemma 3.6-3.8]{NW2022} some constant $\widetilde{C} > 0$
	(depending only on $T$, $L$, $L_1$, $L_2$, and $C_\eta$, and explicitly given in \eqref{EqDefCTilde}) 
	such that for all $d,N\in\bN$, $m \in \bN_0$,
	$\delta\in(0,1)$, and $\cM\in\bN$ with $\cM\geq \delta^{-2} C_\eta d^p$
	it holds that 
	\begin{equation}                             
		\label{error Euler}
		\bE\left[
		\sup_{t\in[0,T]}\big\|
		X^{d,m}_t-\cX^{d,m,N,\delta,\cM}_t
		\big\|^2
		\right]
		\leq \widetilde{C} \left( d^p+\bE\left[ \big\|\xi^{d,m}\big\|^2 \right] \right)
		 e^{d,N,\delta,\cM},
	\end{equation}
	where 
	\begin{equation*}
		e^{d,N,\delta,\cM}:=N^{-1}+\delta^qd^p+\delta^{-2}d^p\cM^{-1}.
	\end{equation*}
	Indeed, though in \cite[Lemma 3.6-3.8]{NW2022} 
	only SDEs with deterministic initial
	value are considered, one can just follow the proof of 
	\cite[Lemma 3.6-3.8]{NW2022} and replace the initial
	value by $\xi^{d,m}$ to obtain \eqref{error Euler}. 
\end{remark}

Furthermore, we impose the following condition on the moments of $\eta^d$ under $\nu^d$.

\begin{assumption}[Moments of $\eta^d$ under $\nu^d$]
	\label{AssEtaMoments}
	Let $c \in \mathbb{N}_0$, $d \in \mathbb{N}$, and $\delta \in (0,1)$ be given. Then for every $(t,x) \in [0,T]\times\bR^d$ 
	we have $\int_{A^d_\delta} \max\{ 1, \Vert \eta^d_t(x,z) \Vert^{2c} \}\, \nu^d(dz) < \infty$, where
	$A^{d}_\delta:=\{z\in\bR^{d}:\|z\|\geq \delta\}$.
\end{assumption}

\subsection{Main Results}
\label{section main results}

In this section, we present the main results of this paper, 
which provides us with a full error analysis of the deep splitting algorithm.
First, we start with the original deep splitting method extended to PIDEs
using deterministic neural networks, followed by its randomized version.

First, we analyze the convergence of the deep splitting method with deterministic fully trained neural network. To this end, we assume that the activation function $\rho \in \overline{C_b(\mathbb{R})}^\gamma$ is non-polynomial, i.e.~algebraically not a polynomial over $\mathbb{R}$. Since $\rho \in \overline{C_b(\mathbb{R})}^\gamma$ induces the tempered distribution $\left( g \mapsto T_\rho(g) := \int_{\mathbb{R}} \rho(s) g(s) ds \right) \in \mathcal{S}'(\mathbb{R};\mathbb{C})$ (see e.g.~\cite[Equation~9.26]{Folland1992}), it follows that $\rho \in \overline{C_b(\mathbb{R})}^\gamma$ is non-polynomial if and only if the Fourier transform $\widehat{T_\rho} \in \mathcal{S}'(\mathbb{R};\mathbb{C})$ of $T_\rho \in \mathcal{S}'(\mathbb{R};\mathbb{C})$ in the sense of distribution is supported\footnote{For the definition of the support of $\widehat{T_\rho} \in \mathcal{S}'(\mathbb{R};\mathbb{C})$, we refer, e.g., to \cite[Notation~(viii)]{NeufeldSchmocker2024}.} at a non-zero point (see \cite[Examples~7.16]{Rudin1991}).

Then, by splitting up the approximation error into the time discretization error, the universal approximation error, and the generalization error, we are able to present our main result on the convergence of the original deep splitting algorithm with deterministic (fully trained) neural networks.

\begin{theorem}
	\label{ThmMain}
	For $\gamma\in(0,\infty)$, let $\rho \in \overline{C_b(\mathbb{R})}^\gamma$ be weakly differentiable and non-polynomial. Moreover, let Assumptions \ref{assumption Lip and growth}--\ref{assumption time Holder} hold, and let Assumption~\ref{AssEtaMoments} hold with some $c \in \mathbb{N}_0 \cap [\gamma,\infty)$. Then, there exists a constant $\widehat{C} > 0$ (depending only on $T$, $L$, $L_1$, $L_2$, and $C_\eta$, and explicitly given in \eqref{EqDefCHat}) such that for every $d,N,\cM \in \bN$, $\delta,\varepsilon \in (0,1)$ with $\cM\geq \delta^{-2} C_\eta d^p$, and $n\in\{N-1,N-2,\dots,1,0\}$, the following holds true:
	\begin{enumerate}
		\item\label{ThmMain1} The \emph{time discretization error} is bounded by 
		\begin{flalign*}
			\quad\quad\quad\,\, \err^{d,N,\delta,\mathcal{M}}_{discr,n} := \mathbb{E}\left[ \left\vert u^d\left( t_n, X^{d,0}_{t_n} \right) - \mathcal{V}^d_n\left( \cX^{d,0,N,\delta,\cM}_{t_n} \right) \right\vert^2 \right] \leq \widehat{C} d^p \left( d^p+\bE\left[ \big\|\xi^{d,0}\big\|^2 \right] \right) e^{d,N,\delta,\cM}. & &
		\end{flalign*}
		\item\label{ThmMain2} There exist network parameters $\theta_n \in \mathbb{R}^r$ such that the \emph{universal approximation error} satisfies
		\begin{flalign*}
			\quad\quad\quad\,\, \err^{d,N,\delta,\mathcal{M}}_{UAT,n} := \mathbb{E}\left[ \left\vert \mathcal{V}^d_n\left( \cX^{d,0,N,\delta,\cM}_{t_n} \right) - \mathbb{V}^d_n\left( \theta_n, \cX^{d,0,N,\delta,\cM}_{t_n} \right) \right\vert^2 \right] \leq \varepsilon. & &
		\end{flalign*}
	\end{enumerate}
	As a consquence, we therefore conclude for every $d,N,M,\cM\in \bN$ and $\delta,\varepsilon \in (0,1)$ with $\cM\geq \delta^{-2} C_\eta d^p$ that
	\begin{equation}
		\label{EqThmMain1}
		\begin{aligned}
			& \sup_{n\in\{N-1,N-2,\dots,1,0\}} \mathbb{E}\left[ \left\vert u^d\left( t_n, \cX^{d,0,N,\delta,\cM}_{t_n} \right) - \mathbb{V}^d_n\left( \Theta^M_n, \cX^{d,0,N,\delta,\cM}_{t_n} \right) \right\vert^2 \right] \\
			& \quad\quad \leq 3 \widehat{C} d^p \left( d^p+\bE\left[ \big\|\xi^{d,0}\big\|^2 \right] \right) e^{d,N,\delta,\cM} + 3 \varepsilon + 3 \sup_{n\in\{N-1,N-2,\dots,1,0\}} \err^{d,N,\delta,\mathcal{M}}_{gen,n},
		\end{aligned}
	\end{equation}
	where $e^{d,N,\delta,\cM} := N^{-1}+\delta^qd^p+\delta^{-2}d^p\cM^{-1} > 0$ was defined in Remark~\ref{RemTimeDiscr}, and where
	\begin{flalign*}
		\quad\quad\quad\,\, \err^{d,N,\delta,\mathcal{M}}_{gen,n} := \mathbb{E}\left[ \left\vert \mathbb{V}^d_n\left( \theta_n, \cX^{d,0,N,\delta,\cM}_{t_n} \right) - \mathbb{V}^d_n\left( \Theta^M_n, \cX^{d,0,N,\delta,\cM}_{t_n} \right) \right\vert^2 \right] \geq 0 & &
	\end{flalign*}
	denotes the \emph{generalization error}.
\end{theorem}

Theorem~\ref{ThmMain} shows an a priori estimate for the approximation error of the deep splitting algorithm with deterministic neural networks. Let us point out the following remarks concerning Theorem~\ref{ThmMain}.

\begin{remark}
	\hfill
	\begin{enumerate}
		\item The time discretization error $\err^{d,N,\delta,\mathcal{M}}_{discr,n}$ and the universal approximation error $\err^{d,N,\delta,\mathcal{M}}_{UAT,n}$ can be made arbitrarily small thanks to Remark~\ref{RemTimeDiscr} and the universal approximation property of deterministic neural networks (see e.g.~\cite{Cybenko1989,Hornik1991,Hornik1989,NeufeldSchmocker2024}).
		\item However, the generalization error $\err^{d,N,\delta,\mathcal{M}}_{gen,n}$ consists of the prediction error (how well the neural networks predicts outside the training set) and the optimization error (how well the stochastic gradient descent algorithm learns the optimal neural network), where the latter is in general difficult to estimate due to the non-convexity of the network's loss function and the stochastic nature of the SGD algorithm (see also \cite{Bartlett2002,Neyshabur2017,Vapnik2000,Zhang2021}).
		\item The a priori estimate in Theorem~\ref{ThmMain} extends the a priori estimate in \cite{GPW2022} obtained for PDEs (without non-local term) and the a priori estimate in \cite{frey2022convergence} obtained for PIDEs (with non-local term) but with finite activity.
	\end{enumerate}
\end{remark}

Next, we derive the convergence result for the random deep splitting method, where the deterministic neural networks are now replaced by random neural networks introduced in Definition~\ref{DefRN}. Compared to Theorem~\ref{ThmMain}, the use of random neural networks allows us to also control the generalization error and hence to obtain a full error analysis for our algorithm.

To this end, we impose the following condition on the i.i.d.~random initialization $(A_k,B_k)_{k \in \mathbb{N}}$.

\begin{assumption}[Full Support]
	\label{AssCDF}
	For every $d \in \mathbb{N}$, the random vector $(A_1,B_1): \Omega \rightarrow \mathbb{R}^d \times \mathbb{R}$ satisfies for every $a \in \mathbb{R}^d$, $b \in \mathbb{R}$, and $r > 0$ that $\mathbb{P}\left[ \left\lbrace \omega \in \Omega: \Vert (A_1(\omega),B_1(\omega))-(a,b) \Vert < r \right\rbrace \right] > 0$.
\end{assumption}

Now, we present our second main result about the random deep splitting algorithm. To this end, we use a result on non-parametric function regression (see \cite{Gyoerfi2002}) to estimate the generalization error, where $\mathbb{R} \ni s \mapsto T_\vartheta(s) := \min(\max(s,-\vartheta),\vartheta) \in \mathbb{R}$ denotes the truncation at level $\vartheta > 0$.

\begin{theorem}
	\label{ThmMainRN}
	For $\gamma\in(0,\infty)$, let $\rho \in \overline{C_b(\mathbb{R})}^\gamma$ be non-polynomial. Moreover, let Assumptions~\ref{assumption Lip and growth}--\ref{assumption time Holder} hold, let Assumption~\ref{AssXiLq} hold with some $q \in (2,\infty)$, let Assumption~\ref{AssEtaMoments} hold with some $c \in \mathbb{N}_0 \cap [\gamma,\infty)$, and let $(A_n,B_n)_{n \in \mathbb{N}}$ satisfy Assumption~\ref{AssCDF}. Then, there exist some constants $\widehat{C},\overline{C} > 0$ (depending only on $q$, $T$, $L$, $L_1$, $L_2$, and $C_\eta$, and explicitly given in \eqref{EqDefCHat} and \eqref{EqDefCBar}, respectively) and a universal constant $C_0 > 0$ (explicitly given in \eqref{EqDefC0}) such that for every $d,N,\cM \in \bN$, $\delta,\varepsilon \in(0,1)$ with $\cM\geq \delta^{-2} C_\eta d^p$, and $n\in\{N-1,N-2,\dots,1,0\}$, the following holds true:
	\begin{enumerate}
		\item\label{ThmMainRN1} The \emph{time discretization error} is bounded by 
		\begin{flalign*}
			\quad\quad\quad\,\, \err^{d,N,\delta,\mathcal{M}}_{discr,n} := \mathbb{E}\left[ \left\vert u^d\left( t_n, X^{d,0}_{t_n} \right) - \mathcal{V}^d_n\left( \cX^{d,0,N,\delta,\cM}_{t_n} \right) \right\vert^2 \right] \leq \widehat{C} d^p \left( d^p+\bE\left[ \big\|\xi^{d,0}\big\|^2 \right] \right) e^{d,N,\delta,\cM}. & &
		\end{flalign*}
		\item\label{ThmMainRN2} There exists some $K_\varepsilon \in \mathbb{N}$ and a random variable $Y_n: \Omega \rightarrow \mathbb{R}^{K_\varepsilon}$ such that the \emph{universal approximation error} satisfies
		\begin{flalign*}
			\quad\quad\quad\,\, \err^{d,N,\delta,\mathcal{M}}_{UAT,n} := \mathbb{E}\left[ \left\vert \mathcal{V}^d_n\left( \cX^{d,0,N,\delta,\cM}_{t_n} \right) - \mathscr{V}^d_n\left( Y_n, \cX^{d,0,N,\delta,\cM}_{t_n} \right) \right\vert^2 \right] \leq \varepsilon. & &
		\end{flalign*}
		\item\label{ThmMainRN3} For every $J \in \mathbb{N}$ and $\vartheta > 0$ the \emph{generalization error}\footnote{Hereby, we assume that $\Upsilon_n$ is learned with truncation, i.e.~$\Upsilon_n(\omega) := \argmin_{y \in \mathbb{R}^{K_\varepsilon}} \frac{1}{J} \sum_{j=1}^J \big\vert \mathscr{V}^d_n(y, \cX^{d,j,N,\delta,\cM}_{t_n}(\omega)) - T_\vartheta\big( \mathscr{V}^d_{n+1}(\Upsilon_{n+1}(\omega),\cX^{d,j,N,\delta,\cM}_{t_{n+1}}(\omega)) + (t_{n+1} - t_n) (F^d \circ \mathscr{V}^d_{n+1}(\Upsilon_{n+1}(\omega), \cdot))(t_{n+1},\cX^{d,j,N,\delta,\cM}_{t_{n+1}}(\omega)) \big) \big\vert^2$ instead of \eqref{EqLeastSquares}.} is bounded by
		\begin{flalign*}
			\quad\quad\quad\,\, \err^{d,N,\delta,\mathcal{M}}_{gen,n} := \mathbb{E}\left[ \left\vert T_\vartheta\left( \mathscr{V}^d_n\left( Y_n, \cX^{d,0,N,\delta,\cM}_{t_n} \right) \right) - T_\vartheta \left( \mathscr{V}^d_n\left( \Upsilon_n, \cX^{d,0,N,\delta,\cM}_{t_n} \right) \right) \right\vert^2 \right] & &
		\end{flalign*}
		\begin{flalign*}
			\quad\quad\quad\quad\quad\quad\quad\quad \leq C_0 \vartheta^2 \frac{(\ln(J)+1) K_\varepsilon}{J} + 8 \mathbb{E}\left[ \mathds{1}_{\lbrace \vert \mathscr{V}^d_n(Y_n,\cX^{d,0,N,\delta,\cM}_{t_n}) \vert > \vartheta \rbrace} \left\vert \mathscr{V}^d_n\left(Y_n,\cX^{d,0,N,\delta,\cM}_{t_n}\right) \right\vert^2 \right]. & &
		\end{flalign*}
	\end{enumerate}
	As a consquence, we therefore conclude that for every $d,J,N,\cM \in \bN$ and $\delta,\varepsilon \in (0,1)$ with $\cM\geq \delta^{-2} C_\eta d^p$, and $\vartheta > 0$ there exists some $K_\varepsilon \in \mathbb{N}$ such that
	\begin{equation}
		\label{EqThmMainRN1}
		\begin{aligned}
			& \sup_{n\in\{N-1,N-2,\dots,1,0\}} \mathbb{E}\left[ \left\vert u^d\left( t_n, X^{d,0}_{t_n} \right) - \mathscr{V}^d_n\left( \Upsilon_n, \cX^{d,0,N,\delta,\cM}_{t_n} \right) \right\vert^2 \right] \\
			& \quad\quad \leq \frac{\overline{C}}{\vartheta^{q-2}} \bE\left[\big(d^p+\big\|\xi^{d,0}\big\|^2\big)^{q/2}\right] + 64 \widehat{C} d^p \left( d^p+\bE\left[ \big\|\xi^{d,0}\big\|^2 \right] \right) e^{d,N,\delta,\cM}\\
			& \quad\quad\quad\quad + 2 C_0 \vartheta^2 \frac{(\ln(J)+1) K_\varepsilon}{J} + 64 \varepsilon,
		\end{aligned}
	\end{equation}
	where $e^{d,N,\delta,\cM} := N^{-1}+\delta^qd^p+\delta^{-2}d^p\cM^{-1} > 0$ was defined in Remark~\ref{RemTimeDiscr}.
\end{theorem}

Theorem~\ref{ThmMainRN} provides us with a full error analysis of the random splitting method. By using random neural networks, we are in particular able to control the generalization error in \ref{ThmMainRN3}, which cannot be controlled for the splitting method with deterministic neural networks.

\begin{corollary}
	\label{CorMainRN}
	For $\gamma\in(0,\infty)$, let $\rho \in \overline{C_b(\mathbb{R})}^\gamma$ be non-polynomial. Moreover, let Assumptions~\ref{assumption Lip and growth}--\ref{assumption time Holder} hold, let Assumption~\ref{AssXiLq} hold with some $q \in (2,\infty)$, let Assumption~\ref{AssEtaMoments} hold with some $c \in \mathbb{N}_0 \cap [\gamma,\infty)$, and let $(A_n,B_n)_{n \in \mathbb{N}}$ satisfy Assumptions~\ref{AssCDF}. Then, for every $d \in \mathbb{N}$ and $\widetilde{\varepsilon} \in (0,1)$, there exists some $K \in \mathbb{N}$ (see lines~6-7 of Algorithm~\ref{AlgDS}) such that Algorithm~\ref{AlgDS} outputs random neural networks $\mathscr{V}^d_n(\Upsilon_n,\cdot)$, $n \in \lbrace N,N-1,\dots,1,0 \rbrace$ satisfying $\mathscr{V}^d_N(\Upsilon_N,\cdot) = g^d(\cdot)$ and
	\begin{equation*}
		\sup_{n\in\{N-1,N-2,\dots,1,0\}} \mathbb{E}\left[ \left\vert u^d\left(t_n,X^{d,0}_{t_n}\right) - \mathscr{V}^d_n\left(\Upsilon_n,\cX^{d,0,N,\delta,\cM}_{t_n}\right) \right\vert^2 \right] \leq \widetilde{\varepsilon}.
	\end{equation*}
\end{corollary}

Corollary~\ref{CorMainRN} emphasizes the main result of this paper, namely that the random deep splitting algorithm can \emph{provably} converge to the true solution $u^d(t,x)$ of the PIDE in \eqref{PDE}.

\section{\textbf{Numerical examples}}
\label{section numerics}

In this section, we illustrate in four numerical examples\footnote{All numerical experiments have been implemented in \texttt{Python} on an average laptop with GPU (Lenovo ThinkPad X13 Gen2a with Processor AMD Ryzen 7 PRO 5850U and Radeon Graphics, 1901 Mhz, 8 Cores, 16 Logical Processors). The code for the (random) deep splitting method can be found here: \url{https://github.com/psc25/RandomDeepSplitting}. On the other hand, the code for the MLP algorithm can be found here: \url{https://github.com/psc25/MLPJumps}. Note that GPU only accelerates the training of deterministic neural networks via the \texttt{TensorFlow} package. However, for MLP and random neural networks which are trained with the packages \texttt{NumPy} and \texttt{SciPy}, GPU offers no advantage over CPU.} how to apply the deep splitting method with deterministic (fully trained) and random neural networks to solve different nonlinear PDEs and PIDEs. 

More precisely, in each of the following numerical examples, we learn for different dimensions $d \in \mathbb{N}$ deterministic (fully trained) neural networks and random neural networks (both with one hidden layer of size $K := \min(d,2000)$ and the tangens hyperbolicus as activation function) on $J = 500$ samples of the Euler discretization $(\cX^{d,0}_{t_n})_{n=0,\dots,N}$, where $N = 12$. The deterministic neural networks are trained using the Adam optimizer (see \cite{KingmaBa2015}) over $M = 2000$ learning steps (epochs) with variable learning rate $\beta_m = 10^{-2} \mathds{1}_{[0,500)}(m) + 10^{-3} \mathds{1}_{[500,1000)}(m) + 10^{-4} \mathds{1}_{[1000,2000]}(m)$, whereas the random neural networks (with random initialization $(A_1,B_1) \sim \mathcal{N}_{K+1}(0,\mathbf{I}_{K+1})$) are learned with the least squares method described in \cite[Algorithm~2]{NeufeldSchmocker2023} and some batch normalization in front of the activation function.

Both methods are then compared to the reference solution obtained via the MLP algorithm in \cite{NW2022} (with Euler discretization $N = 12$, number of samples $M = 5$, and number of levels $l = 5$).

\subsection{Pricing under default risk without jumps}
\label{SecDefaultBS}

In the first example, we aim to price a financial derivative within the nonlinear Black-Scholes model under default risk (see \cite{beck2021deep,BSZ2017,Duffie1996,EHJK2019,han2018solving}). For this purpose, we suppose that the stock prices follow geometric Brownian motions, i.e.~for every $t \in [0,T]$, we assume that
\begin{equation}
	\label{EqBlackScholes}
	X^{d,m}_{t} = x \odot \exp_d\left( \mu_0 \mathbf{1}_d t + \sigma_0 W^{d,m}_t \right).
\end{equation}
Hereby, the parameters are $\mu_0 \in \mathbb{R}$ and $\sigma_0 > 0$, and the initial value is $x \in (0,\infty)^d$. By applying Ito's formula, it follows for every $t \in [0,T]$ that
\begin{equation*}
	dX^{d,m}_{t} = \left( \mu_0 + \tfrac{\sigma_0^2}{2} \right) X^{d,m}_{t} \, dt + \sigma_0 \diag(X^{d,m}_{t}) \, dW^{d,m}_{t}.
\end{equation*}
Hence, the stochastic process $(X^{d,m}_t)_{t \in [0,T]}$ is of the form \eqref{SDE deri} with $\mu^d(s,x) = (\mu_0 + \sigma_0^2/2) x$, $\sigma^d(s,x) = \sigma_0 \diag(x)$, $\eta^d_s(x,z) = 0$, and $\nu^d(dz) = 0$.

In order to include the default risk, we assume that one only gets $\alpha \in [0,1)$ of the contract's current value if a default occurs. Hereby, the first default time is modelled as a Poisson process with intensity $Q(v) = \mathds{1}_{(-\infty,v^h)}(v) \gamma^h + \mathds{1}_{[v^l,\infty)}(v) \gamma^l + \mathds{1}_{[v^h,v^l)}(v) \big( \frac{\gamma^h-\gamma^l}{v^h-v^l} \left( v-v^h \right) + \gamma^h \big)$. Hence, the price of the contract with respect to the payoff function $g^d(x) = \min_{i=1,\dots,d} x_i$ is described by the PIDE \eqref{PDE} with $f^d(t,x,v) = -(1-\alpha) Q(v) v - R v$. For the numerical simulations, we choose the parameters $\mu_0 = -0.01$, $\sigma_0 = 0.15$, $\alpha = 2/3$, $\gamma^h = 0.2$, $\gamma^l = 0.02$, $v^h = 25$, $v^l = 50$, $R = 0.02$, and $x := (30,\dots,30)^\top \in \mathbb{R}^d$. Moreover, note that the coefficients $\mu^d$, $\sigma^d$, and $\eta^d$ as well as the functions $f^d$ and $g^d$ satisfy Assumptions~\ref{assumption Lip and growth}--\ref{AssXiLq}+\ref{AssEtaMoments}. 

For different dimensions $d \in \mathbb{N}$, we run the algorithm $10$ times either with deterministic neural networks or random neural networks to compute an approximation of $u^d(T,30,\dots,30)$ with $T = 1/3$, whereas $\mathcal{M} \in \mathbb{N}$ and $\delta > 0$ do not need to be specified as $\nu^d(dz) = 0$ (for the other parameters $K,J,N \in \mathbb{N}$, we refer to the beginning of Section~\ref{section numerics}). The results are reported in Table~\ref{FigDefaultBS}, which are compared to the reference solution obtained via the MLP method in \cite{NW2022}.

\begin{table}[h!]
	\begin{minipage}[t][][t]{0.49\textwidth}
		\includegraphics[height = 6.0cm]{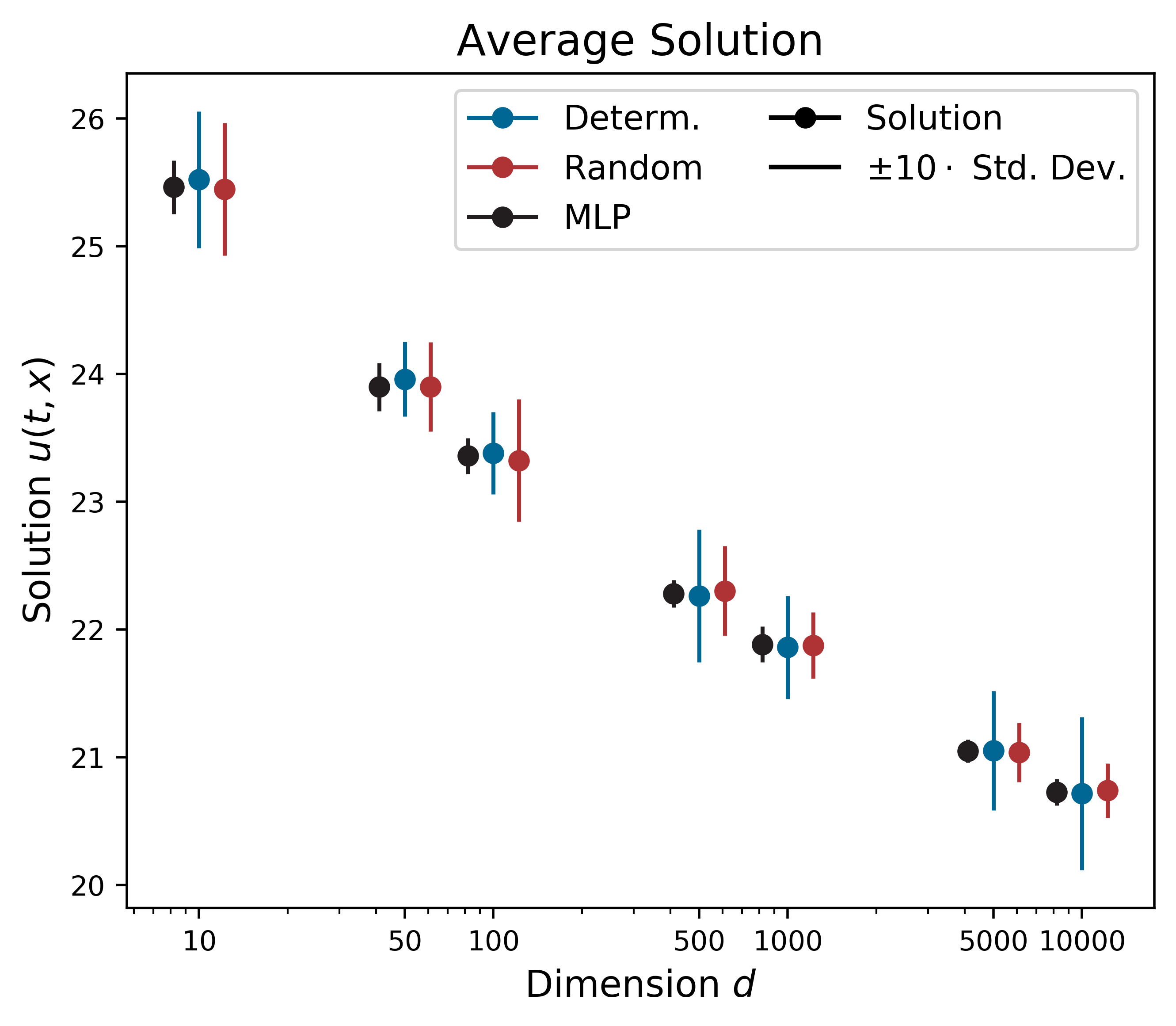}
		\vspace{0.2cm}
	\end{minipage}
	\begin{minipage}[t][][t]{0.49\textwidth}
		\includegraphics[height = 6.0cm]{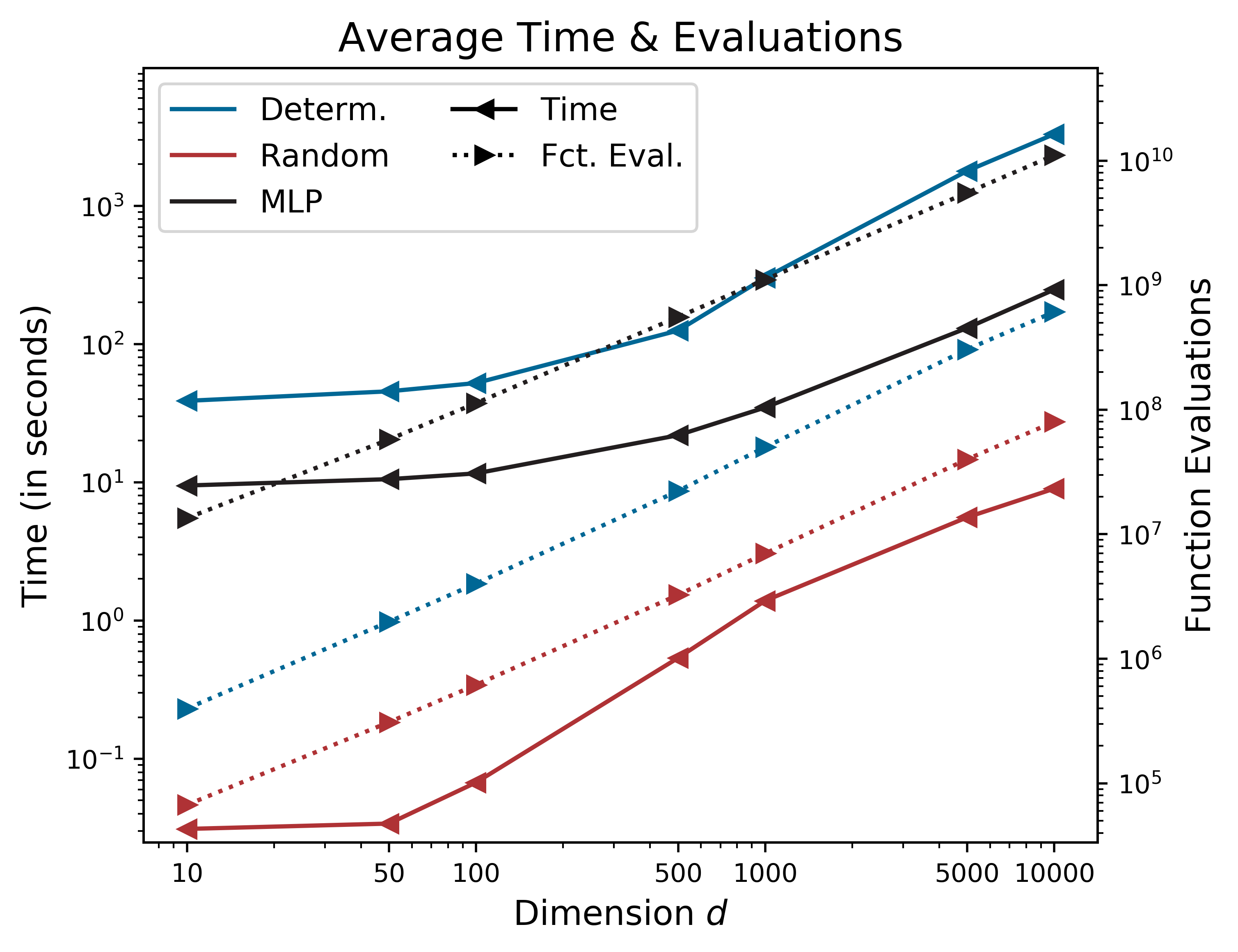}
		\vspace{0.2cm}
	\end{minipage}
	\begin{tabular}{r|R{1.7cm}R{1.7cm}R{1.7cm}|R{1.7cm}R{1.7cm}R{1.7cm}|}
		& \multicolumn{3}{c|}{\textbf{Average Solution}} & \multicolumn{3}{c|}{\textbf{Average Time \& Evaluations}} \\
		$d$ & \textcolor{MidnightBlue}{Determ.} & \textcolor{Maroon}{Random} & MLP & \textcolor{MidnightBlue}{Determ.} & \textcolor{Maroon}{Random} & MLP \\
		\hline
		10 & \textcolor{MidnightBlue}{25.5224} & \textcolor{Maroon}{25.4464} & \textcolor{Black}{25.4633} & \textcolor{MidnightBlue}{38.80} & \textcolor{Maroon}{0.03} & \textcolor{Black}{9.48} \\ 
 & \textcolor{MidnightBlue}{0.0534} & \textcolor{Maroon}{0.0520} & \textcolor{Black}{0.0210} & \textcolor{MidnightBlue}{$3.98 \cdot 10^{5}$} & \textcolor{Maroon}{$6.72 \cdot 10^{4}$} & \textcolor{Black}{$1.35 \cdot 10^{7}$} \\ 
\hline 
50 & \textcolor{MidnightBlue}{23.9589} & \textcolor{Maroon}{23.8994} & \textcolor{Black}{23.8985} & \textcolor{MidnightBlue}{45.46} & \textcolor{Maroon}{0.03} & \textcolor{Black}{10.53} \\ 
 & \textcolor{MidnightBlue}{0.0292} & \textcolor{Maroon}{0.0350} & \textcolor{Black}{0.0189} & \textcolor{MidnightBlue}{$1.99 \cdot 10^{6}$} & \textcolor{Maroon}{$3.10 \cdot 10^{5}$} & \textcolor{Black}{$5.77 \cdot 10^{7}$} \\ 
\hline 
100 & \textcolor{MidnightBlue}{23.3821} & \textcolor{Maroon}{23.3233} & \textcolor{Black}{23.3587} & \textcolor{MidnightBlue}{52.17} & \textcolor{Maroon}{0.07} & \textcolor{Black}{11.59} \\ 
 & \textcolor{MidnightBlue}{0.0322} & \textcolor{Maroon}{0.0480} & \textcolor{Black}{0.0140} & \textcolor{MidnightBlue}{$4.02 \cdot 10^{6}$} & \textcolor{Maroon}{$6.17 \cdot 10^{5}$} & \textcolor{Black}{$1.13 \cdot 10^{8}$} \\ 
\hline 
500 & \textcolor{MidnightBlue}{22.2624} & \textcolor{Maroon}{22.3021} & \textcolor{Black}{22.2800} & \textcolor{MidnightBlue}{125.13} & \textcolor{Maroon}{0.54} & \textcolor{Black}{21.89} \\ 
 & \textcolor{MidnightBlue}{0.0520} & \textcolor{Maroon}{0.0351} & \textcolor{Black}{0.0106} & \textcolor{MidnightBlue}{$2.23 \cdot 10^{7}$} & \textcolor{Maroon}{$3.26 \cdot 10^{6}$} & \textcolor{Black}{$5.55 \cdot 10^{8}$} \\ 
\hline 
1000 & \textcolor{MidnightBlue}{21.8595} & \textcolor{Maroon}{21.8746} & \textcolor{Black}{21.8822} & \textcolor{MidnightBlue}{301.72} & \textcolor{Maroon}{1.39} & \textcolor{Black}{34.73} \\ 
 & \textcolor{MidnightBlue}{0.0402} & \textcolor{Maroon}{0.0259} & \textcolor{Black}{0.0141} & \textcolor{MidnightBlue}{$5.00 \cdot 10^{7}$} & \textcolor{Maroon}{$7.01 \cdot 10^{6}$} & \textcolor{Black}{$1.11 \cdot 10^{9}$} \\ 
\hline 
5000 & \textcolor{MidnightBlue}{21.0511} & \textcolor{Maroon}{21.0381} & \textcolor{Black}{21.0463} & \textcolor{MidnightBlue}{1785.30} & \textcolor{Maroon}{5.58} & \textcolor{Black}{130.62} \\ 
 & \textcolor{MidnightBlue}{0.0468} & \textcolor{Maroon}{0.0231} & \textcolor{Black}{0.0090} & \textcolor{MidnightBlue}{$3.05 \cdot 10^{8}$} & \textcolor{Maroon}{$4.00 \cdot 10^{7}$} & \textcolor{Black}{$5.53 \cdot 10^{9}$} \\ 
\hline 
10000 & \textcolor{MidnightBlue}{20.7160} & \textcolor{Maroon}{20.7384} & \textcolor{Black}{20.7260} & \textcolor{MidnightBlue}{3294.22} & \textcolor{Maroon}{8.99} & \textcolor{Black}{247.45} \\ 
 & \textcolor{MidnightBlue}{0.0599} & \textcolor{Maroon}{0.0214} & \textcolor{Black}{0.0103} & \textcolor{MidnightBlue}{$6.10 \cdot 10^{8}$} & \textcolor{Maroon}{$8.00 \cdot 10^{7}$} & \textcolor{Black}{$1.11 \cdot 10^{10}$}
	\end{tabular}
	\caption{(Random) deep splitting approximation of the pricing problem under default risk with geometric Brownian motions \eqref{EqBlackScholes} for different $d \in \mathbb{N}$, by either using deterministic neural networks (label ``Determ.'') or random neural networks (label ``Random''). The average solutions over $10$ runs together with their standard deviations (italic font) are compared to the MLP algorithm (label ``MLP''). On the right-hand side, the average running time on a laptop with GPU (upper row, in seconds) and the average number of function evaluations (lower row, in scientific format\textsuperscript{\ref{footnote1}}) are displayed.}
	\label{FigDefaultBS}
\end{table}

\subsection{Pricing under default risk with jumps}

In the second example, we consider the same pricing problem as in Section~\ref{SecDefaultBS} consisting of the same nonlinearity $f^d$ and payoff function $g^d$ as in Section~\ref{SecDefaultBS}, but now with jumps in the underlying process $X^{d,m}$. For this purpose, we consider Merton's jump-diffusion model (see \cite{Merton1976} and \cite[Section~4.3]{ContTankov2004}), i.e.~for every $t \in [0,T]$, we assume that
\begin{equation}
	\label{EqMertonModel}
	X^{d,m}_{t} = x \odot \exp_d\left( L^{d,m}_t \right), \quad\quad L^{d,m}_t := \mu_0 \mathbf{1}_d t + \sigma_0 W^{d,m}_t + \sum_{n=1}^{N^{d,m}_t} Z^{d,m}_n - \lambda \mu_z \mathbf{1}_d t,
\end{equation}
where $\mu_0 \in \mathbb{R}$ and $\sigma_0 > 0$ are parameters, and where $x \in (0,\infty)^d$ is the initial value. Moreover, $(N^{d,m}_t)_{t \in [0,\infty)}$ is a Poisson process with intensity $\lambda > 0$ and $(Z^{d,m}_n)_{n \in \mathbb{N}} \sim \mathcal{N}_d(\mu_z \mathbf{1}_d,\sigma_z^2 \, I_d)$ is an i.i.d.~sequence of multivariate normal random variables. Define the Poisson random measure
\begin{equation*}
	\pi^{d,m}((0,t] \times E) := \#\left\lbrace s \in (0,t]: Z^{d,m}_s - Z^{d,m}_{s-} \in E \right\rbrace, \quad\quad t \in (0,T], \quad E \in \mathcal{B}(\mathbb{R}^d \setminus \lbrace 0 \rbrace),
\end{equation*}
and its compensated Poisson random measure $\tilde{\pi}^{d,m}(dz,dt) := \pi^{d,m}(dz,dt) - \nu^d(dz) \otimes dt$, where $\nu^d(dz) = \frac{\lambda}{\left( 2\pi \sigma_z^2 \right)^{d/2}} \exp\left( \frac{\Vert z - \mu_z \mathbf{1}_d \Vert^2}{2\sigma_z^2} \right) dz$. Then, for every $t \in [0,T]$, we have
\begin{equation*}
	X^{d,m}_{t} = x \odot \exp_d\left( \mu_0 \mathbf{1}_d t + \sigma_0 W^{d,m}_t + \int_0^t \int_{\mathbb{R}^d} z \, \tilde{\pi}^{d,m}(dz,ds) \right).
\end{equation*}
\pagebreak
Moreover, by using that $\nu^d$ is a finite measure together with
\begin{equation*}
	\mathbb{E}\left[ \int_0^t \int_{\mathbb{R}^d} \left\Vert X^{d,m}_{s} \odot \left( \exp_d(z) - \mathbf{1}_d - z \right) \right\Vert^2 \, \nu^d(dz) \, ds \right] < \infty,
\end{equation*}
we conclude for every $t \in [0,T]$ that
\begin{equation}
	\label{EqMertonSplit}
	\begin{aligned}
		& \int_0^t \int_{\mathbb{R}^d} X^{d,m}_{s-} \odot \left( \exp_d(z) - \mathbf{1}_d - z \right) \, \pi^{d,m}(dz,ds) \\
		& \quad\quad = \int_0^t \int_{\mathbb{R}^d} X^{d,m}_{s-} \odot \left( \exp_d(z) - \mathbf{1}_d - z \right) \, \tilde{\pi}^{d,m}(dz,ds) \\
		& \quad\quad\quad\quad + \int_0^t \int_{\mathbb{R}^d} X^{d,m}_{s-} \odot \left( \exp_d(z) - \mathbf{1}_d - z \right) \, \nu^d(dz) \, ds
	\end{aligned}
\end{equation}
Hence, by using Ito's formula in \cite[Theorem~3.1]{GW2021} and \eqref{EqMertonSplit}, it follows for every $t \in [0,T]$ that
\begin{equation*}
	\begin{aligned}
		dX^{d,m}_{t} & = \left( \mu_0 + \frac{\sigma_0^2}{2} \right) X^{d,m}_{t-} \, dt + \sigma_0 \diag(X^{d,m}_{t-}) \, dW^{d,m}_{t} \\
		& \quad\quad + \int_{\mathbb{R}^d} X^{d,m}_{t-} \odot z \, \tilde{\pi}^{d,m}(dz,dt) + \int_{\mathbb{R}^d} X^{d,m}_{t-} \odot \left( \exp_d(z) - \mathbf{1}_d - z \right) \, \pi^{d,m}(dz,dt) \\
		& = \left( \mu_0 + \frac{\sigma_0^2}{2} \right) X^{d,m}_{t-} \, dt + \sigma_0 \diag(X^{d,m}_{t-}) \, dW^{d,m}_{t} \\
		& \quad\quad + \int_{\mathbb{R}^d} X^{d,m}_{t-} \odot \left( \exp_d(z) - \mathbf{1}_d \right) \, \tilde{\pi}^{d,m}(dz,dt) + \int_{\mathbb{R}^d} X^{d,m}_{t-} \odot \left( \exp_d(z) - \mathbf{1}_d - z \right) \, \nu^d(dz) \\
		& = \left( \mu_0 + \frac{\sigma_0^2}{2} + \lambda \left( e^{\mu_z+\frac{\sigma_z^2}{2}} - 1 - \mu_z \right) \right) X^{d,m}_{t-} \, dt + \sigma_0 \diag(X^{d,m}_{t-}) \, dW^{d,m}_{t} \\
		& \quad\quad + \int_{\mathbb{R}^d} X^{d,m}_{t-} \odot \left( \exp_d(z) - \mathbf{1}_d \right) \, \tilde{\pi}^{d,m}(dz,dt),
	\end{aligned}
\end{equation*}
Thus, the stochastic process $(X^{d,0,x}_t)_{t \in [0,T]} := (X^{d,m}_t)_{t \in [0,T]}$ is of the form \eqref{SDE deri} with $\mu^d(s,x) = \left( \mu_0 + \frac{\sigma_z^2}{2} + \lambda \left( \exp\left( \mu_z + \frac{\sigma_z^2}{2} \right) - 1 - \mu_z \right) \right) x$, $\sigma^d(s,x) = \sigma_0 \diag(x)$, $\eta^d_s(x,z) = x \odot \left( \exp_d(z) - \mathbf{1}_d \right)$, and $\nu^d(dz) = \frac{\lambda}{\left( 2\pi \sigma_z^2 \right)^{d/2}} \exp\left( \frac{\Vert z - \mu_z \mathbf{1}_d \Vert^2}{2\sigma_z^2} \right) dz$.

Now, we consider the same pricing problem as in Section~\ref{SecDefaultBS} consisting of the nonlinearity $f^d(t,x,v) = -(1-\alpha) Q(v) v - R v$ and payoff function $g^d(x) = \min_{i=1,\dots,d} x_i$, but with the additional integro-part in the PIDE \eqref{PDE}, where we choose the same parameters as in Section~\ref{SecDefaultBS}, except $\lambda = 0.2$, $\mu_z = -0.05$, and $\sigma_z = 0.1$. For different dimensions $d \in \mathbb{N}$, we run the algorithm $10$ times either with deterministic neural networks or random neural networks to compute an approximation of $u^d(T,30,\dots,30)$ with $T = 1/3$, $\mathcal{M} = 200$, and $\delta = 0.1$ (for the other parameters $K,J,N \in \mathbb{N}$, we refer to the beginning of Section~\ref{section numerics}). The results are reported in Table~\ref{FigDefaultMerton}, which are compared to the reference solution obtained via the MLP method in \cite{NW2022}.

\footnotetext{\label{footnote1}By scientific format, we refer to numbers of the form $1.00 \cdot 10^0$.}

\begin{table}[h!]
	\begin{minipage}[t][][t]{0.49\textwidth}
		\includegraphics[height = 6.0cm]{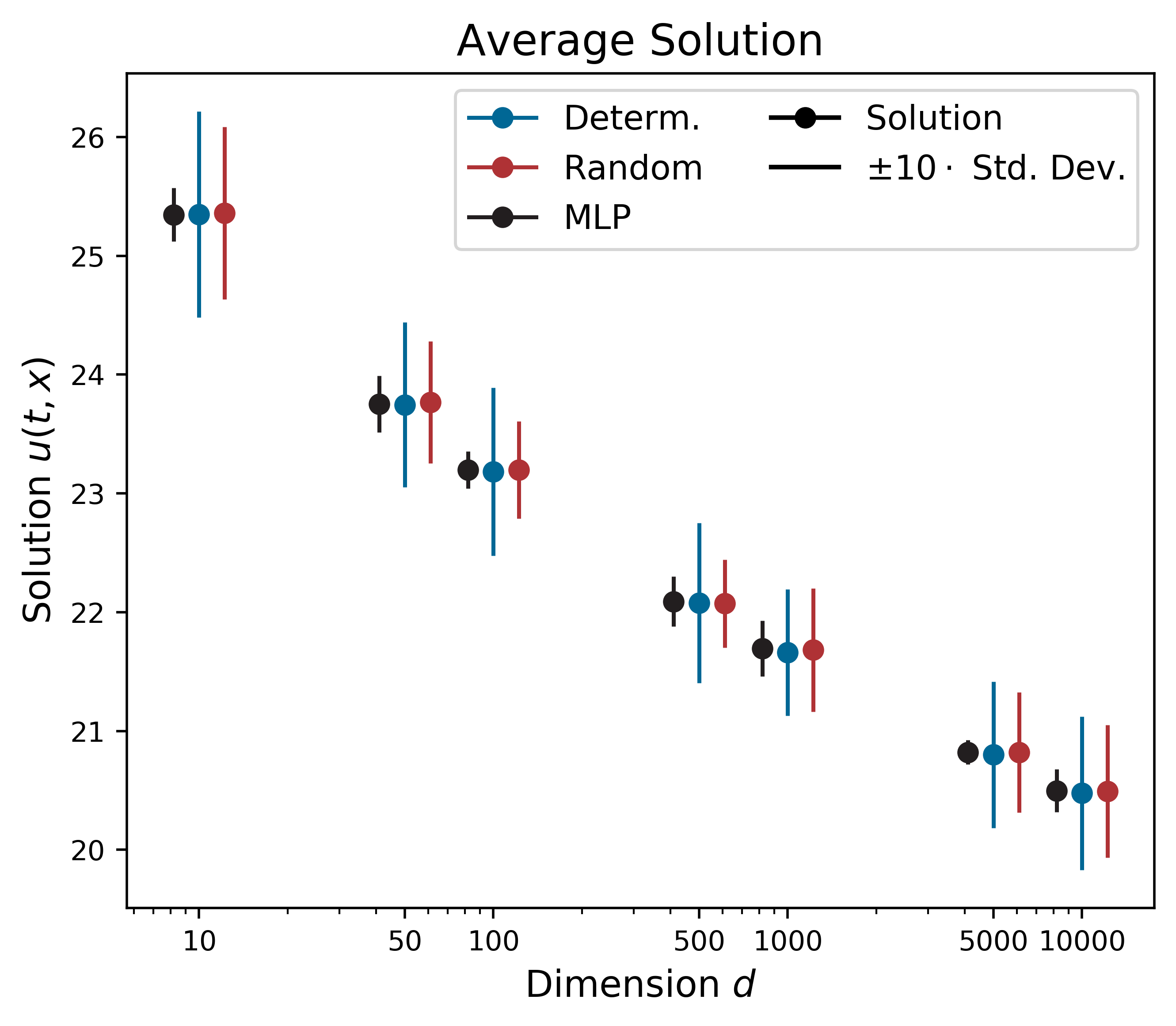}
		\vspace{0.2cm}
	\end{minipage}
	\begin{minipage}[t][][t]{0.49\textwidth}
		\includegraphics[height = 6.0cm]{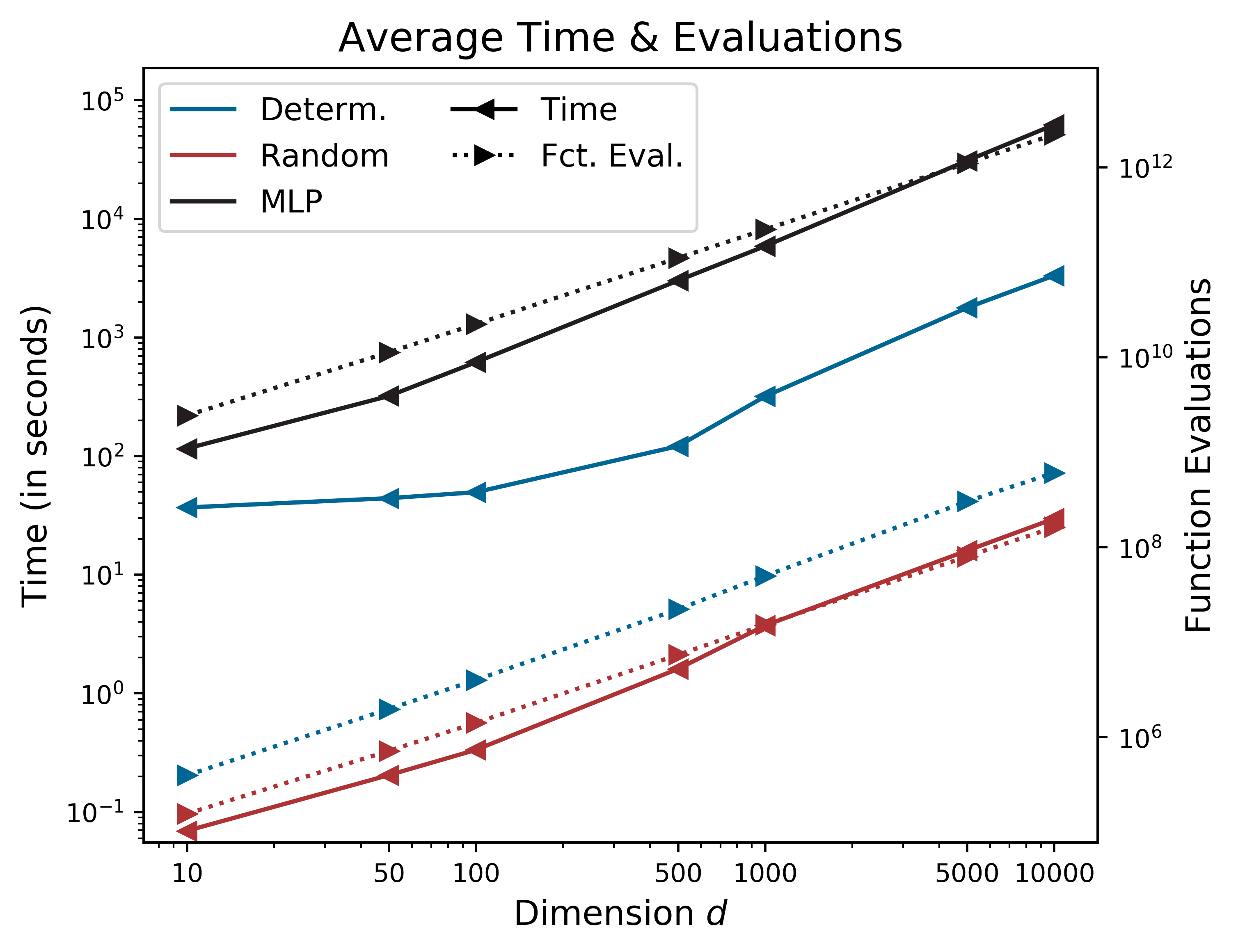}
		\vspace{0.2cm}
	\end{minipage}
	\begin{tabular}{r|R{1.7cm}R{1.7cm}R{1.7cm}|R{1.7cm}R{1.7cm}R{1.7cm}|}
		& \multicolumn{3}{c|}{\textbf{Average Solution}} & \multicolumn{3}{c|}{\textbf{Average Time \& Evaluations}} \\
		$d$ & \textcolor{MidnightBlue}{Determ.} & \textcolor{Maroon}{Random} & MLP & \textcolor{MidnightBlue}{Determ.} & \textcolor{Maroon}{Random} & MLP \\
		\hline
		10 & \textcolor{MidnightBlue}{25.3496} & \textcolor{Maroon}{25.3606} & \textcolor{Black}{25.3464} & \textcolor{MidnightBlue}{36.91} & \textcolor{Maroon}{0.07} & \textcolor{Black}{115.89} \\ 
 & \textcolor{MidnightBlue}{0.0866} & \textcolor{Maroon}{0.0725} & \textcolor{Black}{0.0224} & \textcolor{MidnightBlue}{$3.98 \cdot 10^{5}$} & \textcolor{Maroon}{$1.57 \cdot 10^{5}$} & \textcolor{Black}{$2.45 \cdot 10^{9}$} \\ 
\hline 
50 & \textcolor{MidnightBlue}{23.7462} & \textcolor{Maroon}{23.7667} & \textcolor{Black}{23.7526} & \textcolor{MidnightBlue}{44.21} & \textcolor{Maroon}{0.20} & \textcolor{Black}{321.51} \\ 
 & \textcolor{MidnightBlue}{0.0694} & \textcolor{Maroon}{0.0515} & \textcolor{Black}{0.0239} & \textcolor{MidnightBlue}{$1.99 \cdot 10^{6}$} & \textcolor{Maroon}{$7.20 \cdot 10^{5}$} & \textcolor{Black}{$1.13 \cdot 10^{10}$} \\ 
\hline 
100 & \textcolor{MidnightBlue}{23.1829} & \textcolor{Maroon}{23.1973} & \textcolor{Black}{23.1982} & \textcolor{MidnightBlue}{49.64} & \textcolor{Maroon}{0.33} & \textcolor{Black}{619.23} \\ 
 & \textcolor{MidnightBlue}{0.0708} & \textcolor{Maroon}{0.0410} & \textcolor{Black}{0.0157} & \textcolor{MidnightBlue}{$4.02 \cdot 10^{6}$} & \textcolor{Maroon}{$1.42 \cdot 10^{6}$} & \textcolor{Black}{$2.24 \cdot 10^{10}$} \\ 
\hline 
500 & \textcolor{MidnightBlue}{22.0774} & \textcolor{Maroon}{22.0723} & \textcolor{Black}{22.0899} & \textcolor{MidnightBlue}{121.45} & \textcolor{Maroon}{1.61} & \textcolor{Black}{3022.03} \\ 
 & \textcolor{MidnightBlue}{0.0673} & \textcolor{Maroon}{0.0370} & \textcolor{Black}{0.0211} & \textcolor{MidnightBlue}{$2.23 \cdot 10^{7}$} & \textcolor{Maroon}{$7.38 \cdot 10^{6}$} & \textcolor{Black}{$1.11 \cdot 10^{11}$} \\ 
\hline 
1000 & \textcolor{MidnightBlue}{21.6606} & \textcolor{Maroon}{21.6815} & \textcolor{Black}{21.6945} & \textcolor{MidnightBlue}{319.06} & \textcolor{Maroon}{3.73} & \textcolor{Black}{5894.59} \\ 
 & \textcolor{MidnightBlue}{0.0530} & \textcolor{Maroon}{0.0520} & \textcolor{Black}{0.0235} & \textcolor{MidnightBlue}{$5.00 \cdot 10^{7}$} & \textcolor{Maroon}{$1.53 \cdot 10^{7}$} & \textcolor{Black}{$2.22 \cdot 10^{11}$} \\ 
\hline 
5000 & \textcolor{MidnightBlue}{20.8000} & \textcolor{Maroon}{20.8193} & \textcolor{Black}{20.8211} & \textcolor{MidnightBlue}{1783.06} & \textcolor{Maroon}{15.99} & \textcolor{Black}{30977.20} \\ 
 & \textcolor{MidnightBlue}{0.0616} & \textcolor{Maroon}{0.0507} & \textcolor{Black}{0.0102} & \textcolor{MidnightBlue}{$3.05 \cdot 10^{8}$} & \textcolor{Maroon}{$8.02 \cdot 10^{7}$} & \textcolor{Black}{$1.11 \cdot 10^{12}$} \\ 
\hline 
10000 & \textcolor{MidnightBlue}{20.4759} & \textcolor{Maroon}{20.4922} & \textcolor{Black}{20.4966} & \textcolor{MidnightBlue}{3330.89} & \textcolor{Maroon}{29.74} & \textcolor{Black}{62182.71} \\ 
 & \textcolor{MidnightBlue}{0.0645} & \textcolor{Maroon}{0.0560} & \textcolor{Black}{0.0181} & \textcolor{MidnightBlue}{$6.10 \cdot 10^{8}$} & \textcolor{Maroon}{$1.63 \cdot 10^{8}$} & \textcolor{Black}{$2.22 \cdot 10^{12}$}
	\end{tabular}
	\caption{(Random) deep splitting approximation of the pricing problem under default risk with Merton's jump-diffusion model \eqref{EqMertonModel} for different $d \in \mathbb{N}$, by either using deterministic neural networks (label ``Determ.'') or random neural networks (label ``Random''). The average solutions over $10$ runs together with their standard deviations (italic font) are compared to the MLP algorithm (label ``MLP''). On the right-hand side, the average running time on a laptop with GPU (upper row, in seconds) and the average number of function evaluations (lower row, in scientific format\textsuperscript{\ref{footnote1}}) are displayed.}
	\label{FigDefaultMerton}
\end{table}

\subsection{Pricing with counterparty credit risk in Vasi\v{c}ek jump model}
\label{SecCounterpartyVasicek}

In the third example, we consider the problem of pricing a financial derivative with counterparty credit risk, which was already considered in \cite{EHJK2019} (see also \cite{BurgardKjaer2011,HernyLabordere2012} for the corresponding PDE). For this purpose, we assume that the stock prices follow a Vasi\v{c}ek jump model (see also \cite{WuLiang2018}), i.e.~for every $t \in [0,T]$, we assume that
\begin{equation}
	\label{EqJVasicek}
	X^{d,m}_{t} = x + \alpha \int_0^t \left( \mu_0 \mathbf{1}_d - X^{d,m}_{s} \right) \, ds + \sigma_0 W^{d,m}_t + \sum_{n=1}^{N_t^{d,m}} Z^{d,m}_n - \frac{\lambda t}{2} \mathbf{1}_d.
\end{equation}
where $\alpha, \mu_0, \sigma_0 > 0$ are parameters, and where $x \in \mathbb{R}^d$ is the initial value. Moreover, $(N^{d,m}_t)_{t \in [0,\infty)}$ is a Poisson process with intensity $\lambda > 0$ and $(Z^{d,m}_n)_{n \in \mathbb{N}} \sim \mathcal{U}_d([0,1]^d)$ is an i.i.d.~sequence of multivariate uniformly distributed random variables. Define the Poisson random measure
\begin{equation*}
	\pi^{d,m}((0,t] \times E) := \#\left\lbrace s \in (0,t]: X_s - X_{s-} \in E \right\rbrace, \quad\quad t \in (0,T], \quad E \in \mathcal{B}(\mathbb{R}^d \setminus \lbrace 0 \rbrace),
\end{equation*}
and its compensated Poisson random measure $\tilde{\pi}^{d,m}(dz,dt) := \pi^{d,m}(dz,dt) - \nu^d(dz) \otimes dt$, where $\nu^d(dz) := \lambda \mathbf{1}_{[0,1]^d}(z) dz$. Then, for every $t \in [0,T]$, we have
\begin{equation*}
	X^{d,m}_{t} = x + \alpha \int_0^t \left( \mu_0 - X^{d,m}_{s} \right) \, ds + \sigma_0 \int_0^t \, dW^{d,m}_{s} + \int_{\mathbb{R}^d} z \, \tilde{\pi}^{d,m}(dz,dt).
\end{equation*}
Hence, the stochastic process $(X^{d,0,x}_t)_{t \in [0,T]} := (X^{d,m}_t)_{t \in [0,T]}$ is of the form \eqref{SDE deri} with $\mu^d(s,x) := \alpha (\mu_0 - x)$, $\sigma^d(s,x) := \sigma_0 I_d$, $\eta^d_s(x,z) := z \mathbf{1}_{[0,1]^d}(z)$, and $\nu^d(dz) := \lambda \mathbf{1}_{[0,1]^d}(z) dz$.

Then, by following the derivations in \cite{BurgardKjaer2011,HernyLabordere2012}, the pricing problem with respect to the payoff function $g^d(x) = \max\left( \min_{i=1,\dots,d} x_i - K_1, 0 \right) - \max\left( \min_{i=1,\dots,d} x_i - K_2, 0 \right) - L$ is given by the PIDE \eqref{PDE} with $f^d(t,x,v) = -\zeta \min(v,0)$. For the numerical simulations, we choose the parameters $\alpha = 0.01$, $\mu_0 = 100$, $\sigma_0 = 2$, $\zeta = 0.03$, $K_1 = 80$, $K_2 = 100$, and $L = 5$. Moreover, note that the coefficients $\mu^d$, $\sigma^d$, and $\eta^d$ as well as the functions $f^d$ and $g^d$ satisfy Assumptions~\ref{assumption Lip and growth}--\ref{AssXiLq}+\ref{AssEtaMoments}. 

For different dimensions $d \in \mathbb{N}$, we run the algorithm $10$ times either with deterministic neural networks or random neural networks to compute an approximation of $u^d(T,100,\dots,100)$ with $T = 1/2$, $\mathcal{M} = 200$, and $\delta = 0.1$ (for the other parameters $K,J,N \in \mathbb{N}$, we refer to the beginning of Section~\ref{section numerics}). The results are reported in Table~\ref{FigCounterpartyVasicek}, which are compared to the reference solution obtained via the MLP method in \cite{NW2022}.

\begin{table}[h!]
	\begin{minipage}[t][][t]{0.49\textwidth}
		\includegraphics[height = 6.0cm]{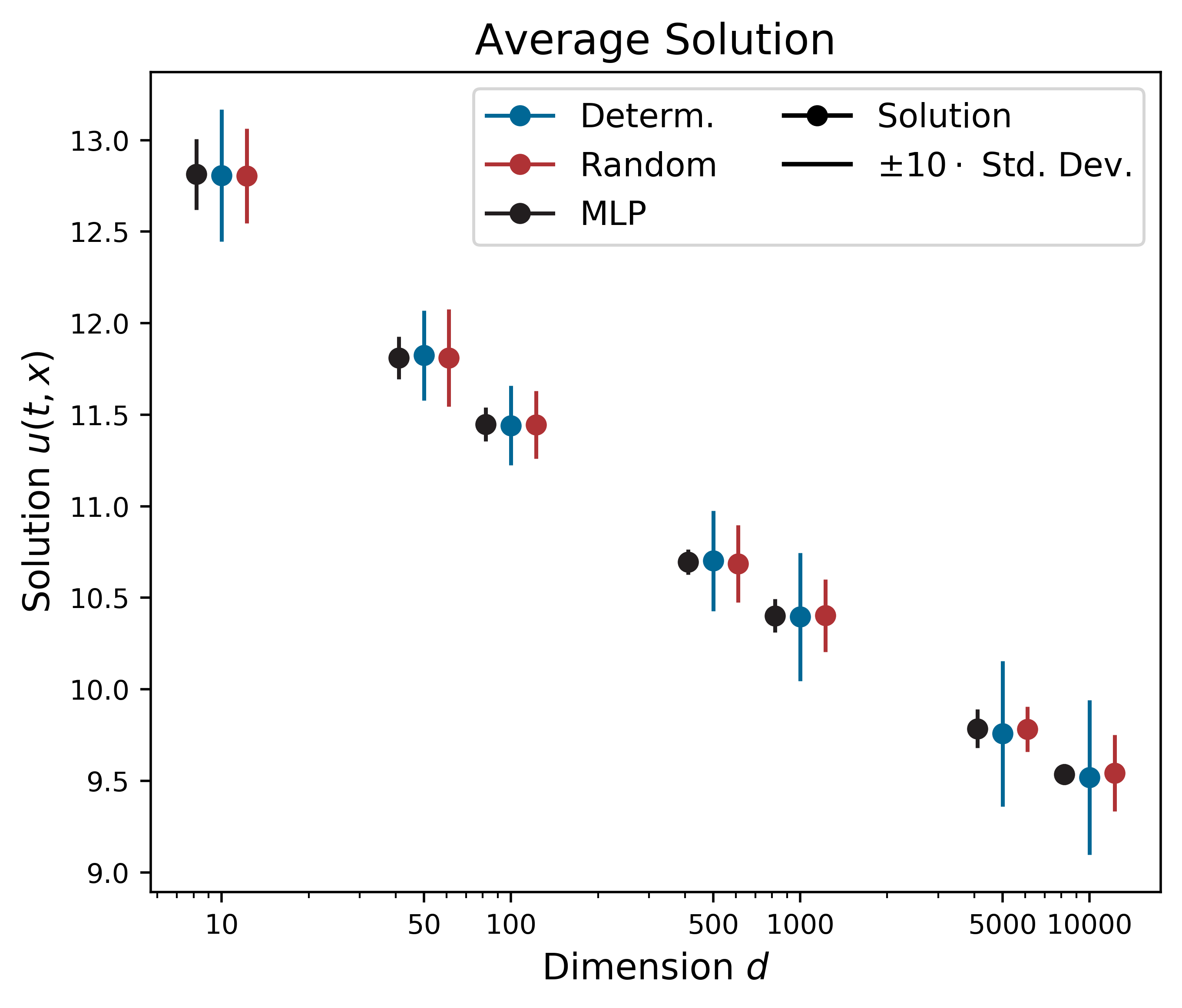}
		\vspace{0.2cm}
	\end{minipage}
	\begin{minipage}[t][][t]{0.49\textwidth}
		\includegraphics[height = 6.0cm]{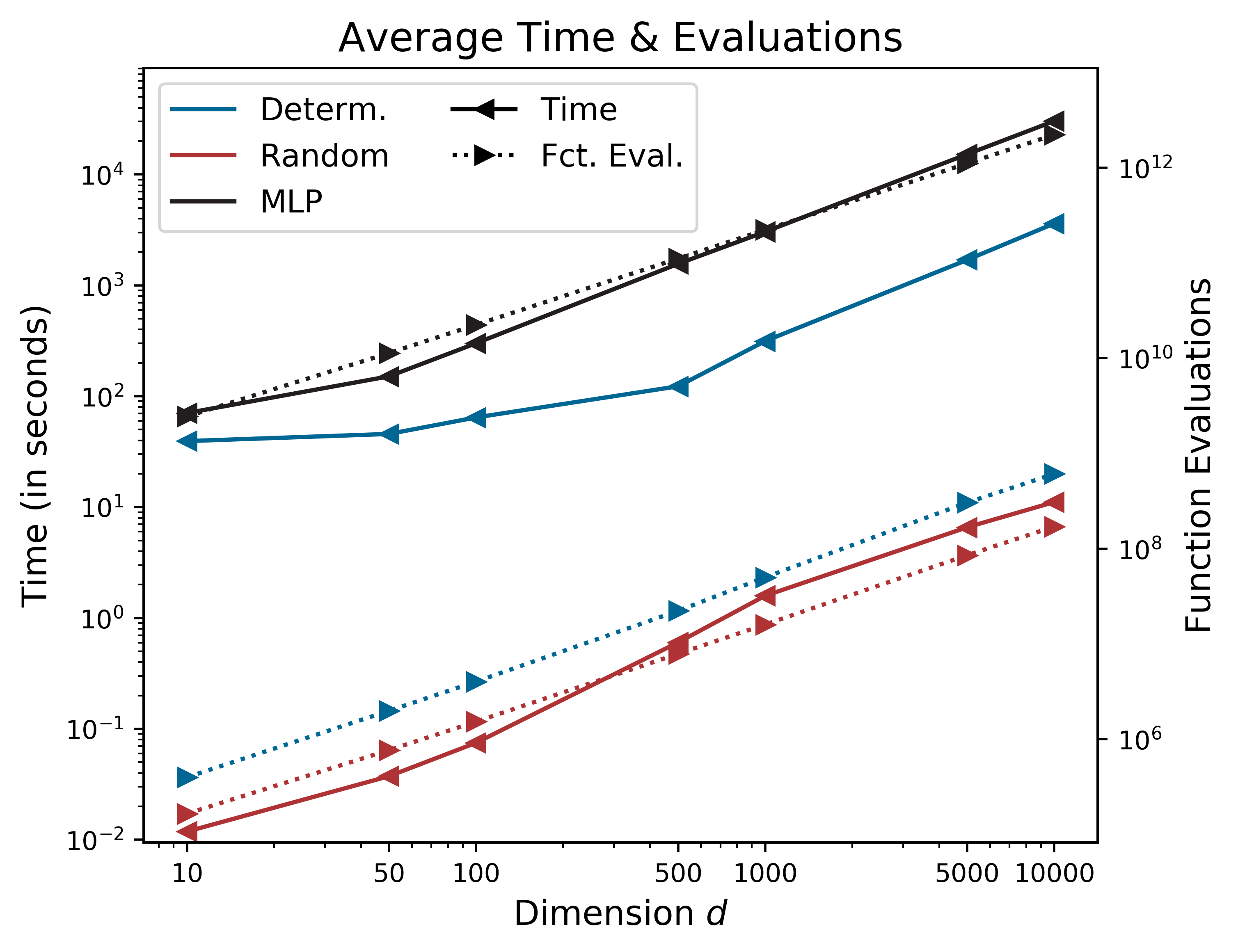}
		\vspace{0.2cm}
	\end{minipage}
	\begin{tabular}{r|R{1.7cm}R{1.7cm}R{1.7cm}|R{1.7cm}R{1.7cm}R{1.7cm}|}
		& \multicolumn{3}{c|}{\textbf{Average Solution}} & \multicolumn{3}{c|}{\textbf{Average Time \& Evaluations}} \\
		$d$ & \textcolor{MidnightBlue}{Determ.} & \textcolor{Maroon}{Random} & MLP & \textcolor{MidnightBlue}{Determ.} & \textcolor{Maroon}{Random} & MLP \\
		\hline
		10 & \textcolor{MidnightBlue}{12.8074} & \textcolor{Maroon}{12.8043} & \textcolor{Black}{12.8130} & \textcolor{MidnightBlue}{39.38} & \textcolor{Maroon}{0.01} & \textcolor{Black}{70.52} \\ 
 & \textcolor{MidnightBlue}{0.0360} & \textcolor{Maroon}{0.0259} & \textcolor{Black}{0.0193} & \textcolor{MidnightBlue}{$3.98 \cdot 10^{5}$} & \textcolor{Maroon}{$1.65 \cdot 10^{5}$} & \textcolor{Black}{$2.45 \cdot 10^{9}$} \\ 
\hline 
50 & \textcolor{MidnightBlue}{11.8245} & \textcolor{Maroon}{11.8108} & \textcolor{Black}{11.8102} & \textcolor{MidnightBlue}{45.69} & \textcolor{Maroon}{0.04} & \textcolor{Black}{151.15} \\ 
 & \textcolor{MidnightBlue}{0.0246} & \textcolor{Maroon}{0.0266} & \textcolor{Black}{0.0115} & \textcolor{MidnightBlue}{$1.99 \cdot 10^{6}$} & \textcolor{Maroon}{$7.67 \cdot 10^{5}$} & \textcolor{Black}{$1.13 \cdot 10^{10}$} \\ 
\hline 
100 & \textcolor{MidnightBlue}{11.4409} & \textcolor{Maroon}{11.4444} & \textcolor{Black}{11.4468} & \textcolor{MidnightBlue}{64.36} & \textcolor{Maroon}{0.07} & \textcolor{Black}{300.37} \\ 
 & \textcolor{MidnightBlue}{0.0217} & \textcolor{Maroon}{0.0185} & \textcolor{Black}{0.0092} & \textcolor{MidnightBlue}{$4.02 \cdot 10^{6}$} & \textcolor{Maroon}{$1.53 \cdot 10^{6}$} & \textcolor{Black}{$2.24 \cdot 10^{10}$} \\ 
\hline 
500 & \textcolor{MidnightBlue}{10.7014} & \textcolor{Maroon}{10.6857} & \textcolor{Black}{10.6960} & \textcolor{MidnightBlue}{122.91} & \textcolor{Maroon}{0.60} & \textcolor{Black}{1563.30} \\ 
 & \textcolor{MidnightBlue}{0.0274} & \textcolor{Maroon}{0.0211} & \textcolor{Black}{0.0069} & \textcolor{MidnightBlue}{$2.23 \cdot 10^{7}$} & \textcolor{Maroon}{$7.88 \cdot 10^{6}$} & \textcolor{Black}{$1.11 \cdot 10^{11}$} \\ 
\hline 
1000 & \textcolor{MidnightBlue}{10.3957} & \textcolor{Maroon}{10.4027} & \textcolor{Black}{10.4021} & \textcolor{MidnightBlue}{314.18} & \textcolor{Maroon}{1.59} & \textcolor{Black}{3039.66} \\ 
 & \textcolor{MidnightBlue}{0.0350} & \textcolor{Maroon}{0.0198} & \textcolor{Black}{0.0092} & \textcolor{MidnightBlue}{$5.00 \cdot 10^{7}$} & \textcolor{Maroon}{$1.60 \cdot 10^{7}$} & \textcolor{Black}{$2.23 \cdot 10^{11}$} \\ 
\hline 
5000 & \textcolor{MidnightBlue}{9.7573} & \textcolor{Maroon}{9.7819} & \textcolor{Black}{9.7854} & \textcolor{MidnightBlue}{1698.46} & \textcolor{Maroon}{6.58} & \textcolor{Black}{15233.95} \\ 
 & \textcolor{MidnightBlue}{0.0397} & \textcolor{Maroon}{0.0123} & \textcolor{Black}{0.0105} & \textcolor{MidnightBlue}{$3.05 \cdot 10^{8}$} & \textcolor{Maroon}{$8.54 \cdot 10^{7}$} & \textcolor{Black}{$1.11 \cdot 10^{12}$} \\ 
\hline 
10000 & \textcolor{MidnightBlue}{9.5182} & \textcolor{Maroon}{9.5419} & \textcolor{Black}{9.5347} & \textcolor{MidnightBlue}{3613.56} & \textcolor{Maroon}{11.11} & \textcolor{Black}{30334.40} \\ 
 & \textcolor{MidnightBlue}{0.0422} & \textcolor{Maroon}{0.0208} & \textcolor{Black}{0.0057} & \textcolor{MidnightBlue}{$6.10 \cdot 10^{8}$} & \textcolor{Maroon}{$1.71 \cdot 10^{8}$} & \textcolor{Black}{$2.22 \cdot 10^{12}$}
	\end{tabular}
	\caption{(Random) deep splitting approximation of the pricing problem under counterparty credit risk with Vasi\v{c}ek jump model \eqref{EqJVasicek} for different $d \in \mathbb{N}$, by either using deterministic neural networks (label ``Determ.'') or random neural networks (label ``Random''). The average solutions over $10$ runs together with their standard deviations (italic font) are compared to the MLP algorithm (label ``MLP''). On the right-hand side, the average running time on a laptop with GPU (upper row, in seconds) and the average number of function evaluations (lower row, in scientific format\textsuperscript{\ref{footnote1}}) are displayed.}
	\label{FigCounterpartyVasicek}
\end{table}

\subsection{Pricing with counterparty credit risk in exponential Variance-Gamma model}
\label{SecCounterpartyExpVG}

In the fourth example, we consider the same pricing problem as in Section~\ref{SecCounterpartyVasicek} consisting of the same nonlinearity $f^d$ and payoff function $g^d$ as in Section~\ref{SecCounterpartyVasicek}, but now with a different underlying process $X^{d,m}$. For this purpose, we consider an exponential Variance-Gamma model (see \cite{Buchmann2017} and also \cite{Madan1998,Madan1990}), i.e.~for every $t \in [0,T]$, we assume that
\begin{equation}
	\label{EqExpVGModel}
	X^{d,m}_{t} = x \odot \exp_d\left( L^{d,m}_t \right), \quad\quad L^{d,m}_t := \mu_0 \mathbf{1}_d t + \sigma_0 W^{d,m}_t + Z^{d,m}_t,
\end{equation}
where $\mu_0 \in \mathbb{R}$ and $\sigma_0 > 0$ are parameters, and where $x \in (0,\infty)^d$ is the initial value. Hereby, $Z^{d,m} := (Z^{d,m}_t)_{t \in [0,\infty)}$ is a $d$-dimensional Madan-Seneta Variance-Gamma process (see \cite[p.~2213]{Buchmann2017} and also \cite{Madan1990}), which is obtained from a $d$-dimensional Brownian motion $B^{d,m} := (B^{d,m}_t)_{t \in [0,\infty)}$ (with zero drift and covariane matrix $\Sigma := \kappa \, \mathbf{I}_d \in \mathbb{R}^{d \times d}$, for some $\kappa > 0$) subordinated by a standard Gamma process $\tau^{d,m} := (\tau^{d,m}_t)_{t \in [0,\infty)}$ with parameter $\alpha > 0$, i.e. $Z^{d,m}_t = B^{d,m}_{\tau^{d,m}_t}$ for all $t \in [0,\infty)$. Then, $Z^{d,m}$ is a pure jump L\'evy process of infinite activity with L\'evy measure $\nu^d(dz) = \frac{(2\alpha)^{d/4+1}}{(2\pi)^{d/2} \kappa^{d/4} \Vert z \Vert^{d/2}} K_{d/2}\left( \sqrt{2\alpha/\kappa} \Vert z \Vert \right) dz$ (see \cite[Equation~2.11]{Buchmann2017}), where $K_{d/2}(\cdot)$ denotes the modified Bessel function of the second kind with parameter $d/2 > 0$ (see \cite[Section~9.6]{Abramowitz1970}). Define the Poisson random measure
\begin{equation*}
	\pi^{d,m}((0,t] \times E) := \#\left\lbrace s \in (0,t]: Z^{d,m}_s - Z^{d,m}_{s-} \in E \right\rbrace, \quad\quad t \in (0,T], \quad E \in \mathcal{B}(\mathbb{R}^d \setminus \lbrace 0 \rbrace),
\end{equation*}
and its compensated Poisson random measure $\tilde{\pi}^{d,m}(dz,dt) := \pi^{d,m}(dz,dt) - \nu^d(dz) \otimes dt$, where $\nu^d(dz) = \frac{(2\alpha)^{d/4+1}}{(2\pi)^{d/2} \kappa^{d/4} \Vert z \Vert^{d/2}} K_{d/2}\left( \sqrt{\frac{2\alpha}{\kappa}} \Vert z \Vert \right) dz$. Then, for every $t \in [0,T]$, we have
\begin{equation*}
	X^{d,m}_{t} = x \odot \exp_d\left( \mu_0 \mathbf{1}_d t + \sigma_0 W^{d,m}_t + \int_0^t \int_{\mathbb{R}^d} z \, \pi^{d,m}(dz,ds) \right).
\end{equation*}
In the following, we now always assume that $\kappa < \alpha/2$. Then, by using that for every $r > 0$ there exists some $C_{d,r} > 0$ such that for every $z \in B_r(0) := \left\lbrace z \in \mathbb{R}^d: \Vert z \Vert \leq r \right\rbrace$ it holds that $\Vert \exp_d(z) - \mathbf{1}_d \Vert \leq C_{d,r} \Vert z \Vert$ and for every $z \in B_r(0)^c := \mathbb{R}^d \setminus B_r(0)$ it holds that $\Vert \exp_d(z) - \mathbf{1}_d \Vert \leq C_{d,r} \exp(\Vert z \Vert)$, the substitution into spherical coordinates (with Jacobi determinant $dz = \frac{2\pi^{d/2} s^{d-1}}{\Gamma(d/2)} ds$, where $\Gamma$ denotes the Gamma function, see \cite[Equation~6.1.1]{Abramowitz1970}), the substitution $t \mapsto \sqrt{2\alpha/\kappa} s$, that $1 \leq e^t$ for any $t \in [0,r]$, that $K_{d/2}(t) \leq e^{r-t} K_{d/2}(r)$ for any $t \in [r,\infty)$ (see \cite[Equation~3.2]{Baricz2010}), that $\int_0^r e^t t^{d/2} K_{d/2}(t) dt < \infty$ (see \cite[Equation~11.3.15]{Abramowitz1970}), and that $\kappa < \alpha/2 < 2\alpha$, we observe that
\begin{equation}
	\label{EqExpVGInt1}
	\begin{aligned}
		& \int_{\mathbb{R}^d} \left\Vert \exp_d(z) - \mathbf{1}_d \right\Vert \, \nu^d(dz) = \int_{B_r(0)} \left\Vert \exp_d(z) - \mathbf{1}_d \right\Vert \, \nu^d(dz) + \int_{B_r(0)^c} \left\Vert \exp_d(z) - \mathbf{1}_d \right\Vert \, \nu^d(dz) \\
		& \quad \leq C_{d,r} \int_{B_r(0)} \frac{(2\alpha)^{\frac{d}{4}+1} \Vert z \Vert}{(2\pi)^\frac{d}{2} \kappa^\frac{d}{4} \Vert z \Vert^\frac{d}{2}} K_\frac{d}{2}\left( \sqrt{\frac{2\alpha}{\kappa}} \Vert z \Vert \right) dz + C_{d,r} \int_{B_r(0)^c} \frac{(2\alpha)^{\frac{d}{4}+1} e^{\Vert z \Vert}}{(2\pi)^\frac{d}{2} \kappa^\frac{d}{4} \Vert z \Vert^\frac{d}{2}} K_\frac{d}{2}\left( \sqrt{\frac{2\alpha}{\kappa}} \Vert z \Vert \right) dz \\
		& \quad \leq C_{d,r} \frac{(2\alpha)^{\frac{d}{4}+1}}{(2\pi)^\frac{d}{2} \kappa^\frac{d}{4}} \int_0^r s^{1-\frac{d}{2}} K_\frac{d}{2}\left( \sqrt{\frac{2\alpha}{\kappa}} s \right) \frac{2\pi^\frac{d}{2} s^{d-1}}{\Gamma(d/2)} ds + C_{d,r} \frac{(2\alpha)^{\frac{d}{4}+1}}{(2\pi)^\frac{d}{2} \kappa^\frac{d}{4}} \int_r^\infty \frac{e^s}{s^\frac{d}{2}} K_\frac{d}{2}\left( \sqrt{\frac{2\alpha}{\kappa}} s \right) \frac{2\pi^\frac{d}{2} s^{d-1}}{\Gamma(d/2)} ds \\
		& \quad \leq C_{d,r} \frac{2 \kappa}{2^\frac{d}{2} \Gamma(d/2)} \int_0^{\sqrt{\frac{2\alpha}{\kappa}} r} t^\frac{d}{2} K_\frac{d}{2}(t) dt + C_{d,r} \frac{2\alpha}{2^\frac{d}{2} \Gamma(d/2)} \int_{\sqrt{\frac{2\alpha}{\kappa}} r}^\infty \exp\left( \sqrt{\frac{\kappa}{2\alpha}} t \right) t^{\frac{d}{2}-1} K_\frac{d}{2}(t) dt \\
		& \quad \leq C_{d,r} \frac{2 \kappa}{2^\frac{d}{2} \Gamma(d/2)} \int_0^{\sqrt{\frac{2\alpha}{\kappa}} r} e^t t^\frac{d}{2} K_\frac{d}{2}(t) dt + C_{d,r} \frac{2\alpha K_{d/2}(r) e^r}{2^\frac{d}{2} \Gamma(d/2)} \int_{\sqrt{\frac{2\alpha}{\kappa}} r}^\infty t^{\frac{d}{2}-1} \exp\left( \left( \sqrt{\frac{\kappa}{2\alpha}} - 1 \right) t \right) dt < \infty.
	\end{aligned}
\end{equation}
Moreover, by using the same arguments as for \eqref{EqExpVGInt1} together with $\int_0^r t^{d/2+1} K_{d/2}(t) dt < \infty$ (see \cite[Equation~11.3.27]{Abramowitz1970}) and $\kappa < \alpha/2$, we conclude that
\begin{equation}
	\label{EqExpVGInt2}
	\begin{aligned}
		& \int_{\mathbb{R}^d} \left\Vert \exp_d(z) - \mathbf{1}_d \right\Vert^2 \, \nu^d(dz) = \int_{B_r(0)} \left\Vert \exp_d(z) - \mathbf{1}_d \right\Vert^2 \, \nu^d(dz) + \int_{B_r(0)^c} \left\Vert \exp_d(z) - \mathbf{1}_d \right\Vert^2 \, \nu^d(dz) \\
		& \quad \leq C_{d,r}^2 \int_{B_r(0)} \frac{(2\alpha)^{\frac{d}{4}+1} \Vert z \Vert^2}{(2\pi)^\frac{d}{2} \kappa^\frac{d}{4} \Vert z \Vert^\frac{d}{2}} K_\frac{d}{2}\left( \sqrt{\frac{2\alpha}{\kappa}} \Vert z \Vert \right) dz + C_{d,r} \int_{B_r(0)^c} \frac{(2\alpha)^{\frac{d}{4}+1} e^{2 \Vert z \Vert}}{(2\pi)^\frac{d}{2} \kappa^\frac{d}{4} \Vert z \Vert^\frac{d}{2}} K_\frac{d}{2}\left( \sqrt{\frac{2\alpha}{\kappa}} \Vert z \Vert \right) dz \\
		& \quad \leq C_{d,r}^2 \frac{2 \kappa}{2^\frac{d}{2} \Gamma(d/2)} \int_0^{\sqrt{\frac{2\alpha}{\kappa}} r} t^{\frac{d}{2}+1} K_\frac{d}{2}(t) dt + C_{d,r}^2 \frac{2\alpha}{2^\frac{d}{2} \Gamma(d/2)} \int_{\sqrt{\frac{2\alpha}{\kappa}} r}^\infty \exp\left( 2 \sqrt{\frac{\kappa}{2\alpha}} t \right) t^{\frac{d}{2}-1} K_\frac{d}{2}(t) dt \\
		& \quad \leq C_{d,r}^2 \frac{2 \kappa}{2^\frac{d}{2} \Gamma(d/2)} \int_0^{\sqrt{\frac{2\alpha}{\kappa}} r} t^{\frac{d}{2}+1} K_\frac{d}{2}(t) dt + C_{d,r}^2 \frac{2\alpha K_{d/2}(r) e^r}{2^\frac{d}{2} \Gamma(d/2)} \int_{\sqrt{\frac{2\alpha}{\kappa}} r}^\infty t^{\frac{d}{2}-1} \exp\left( \left( \sqrt{\frac{\kappa}{\alpha/2}} - 1 \right) t \right) dt < \infty.
	\end{aligned}
\end{equation} 
Hence, by using the localizing sequence $(\tau_k)_{k \in \mathbb{N}}$ of stopping times $\tau_k := \inf\big\lbrace t \in [0,T]: \big\Vert X^{d,m}_{s-} \big\Vert \geq k \big\rbrace$ $\wedge T$, $k \in \mathbb{N}$, we observe that for every $k \in \mathbb{N}$ the stopped integrand $[0,\infty) \times \mathbb{R}^d \ni (s,z) \mapsto X^{d,m}_{s- \wedge \tau_k} \odot \left( \exp_d(z) - \mathbf{1}_d \right) \in \mathbb{R}^d$ is $\mathbb{F}$-predictable, while the inequalities \eqref{EqExpVGInt1} and \eqref{EqExpVGInt2} ensure for every $n \in \lbrace 1,2 \rbrace$ and $t \in [0,\infty)$ that
\begin{equation*}
	\begin{aligned}
		& \mathbb{E}\left[ \int_0^{t \wedge \tau_k} \int_{\mathbb{R}^d} \left\Vert X^{d,m}_{s-} \odot \left( \exp_d(z) - \mathbf{1}_d \right) \right\Vert^n \nu^d(dz) ds \right] \\
		& \quad\quad \leq \mathbb{E}\left[ \int_0^t \int_{\mathbb{R}^d} \left\Vert X^{d,m}_{s- \wedge \tau_k} \right\Vert^n \left\Vert \exp_d(z) - \mathbf{1}_d \right\Vert^n \nu^d(dz) ds \right] \\
		& \quad\quad \leq k^n t \int_{\mathbb{R}^d} \left\Vert \exp_d(z) - \mathbf{1}_d \right\Vert^n \, \nu^d(dz) < \infty.
	\end{aligned}
\end{equation*}
Therefore, \cite[Equation~3.8, p.~62]{IW2011} (see also \cite[Equation~1.14]{rong2006theory}) show for every $t \in [0,T]$ that
\begin{equation}
	\label{EqVGItoSplit}
	\begin{aligned}
		& \int_0^t \int_{\mathbb{R}^d} X^{d,m}_{s-} \odot \left( \exp_d(z) - \mathbf{1}_d \right) \, \pi^{d,m}(dz,ds) \\
		& \quad = \int_0^t \int_{\mathbb{R}^d} X^{d,m}_{s-} \odot \left( \exp_d(z) - \mathbf{1}_d \right) \, \tilde{\pi}^{d,m}(dz,ds) + \int_0^t \int_{\mathbb{R}^d} X^{d,m}_{s-} \odot \left( \exp_d(z) - \mathbf{1}_d \right) \, \nu^d(dz) \, ds.
	\end{aligned}
\end{equation}
Thus, we can apply Ito's formula in \cite[Theorem~3.1]{GW2021}, use \eqref{EqVGItoSplit}, and define the vector-valued integral $I_\nu := \int_{\mathbb{R}^d} \left( \exp_d(z) - \mathbf{1}_d \right) \nu^d(dz) \in \mathbb{R}^d$ (which exists by \eqref{EqExpVGInt1}) to conclude for every $t \in [0,T]$ that
\begin{equation}
	\label{EqVGIto}
	\begin{aligned}
		dX^{d,m}_{t} & = \left( \mu_0 + \frac{\sigma_0^2}{2} \right) X^{d,m}_{t-} \, dt + \sigma_0 \diag(X^{d,m}_{t-}) \, dW^{d,m}_{t} + \int_{\mathbb{R}^d} X^{d,m}_{t-} \odot \left( \exp_d(z) - \mathbf{1}_d \right) \, \pi^{d,m}(dz,dt) \\
		& = \left( \mu_0 + \frac{\sigma_0^2}{2} \right) X^{d,m}_{t-} \, dt + \sigma_0 \diag(X^{d,m}_{t-}) \, dW^{d,m}_{t} + \int_{\mathbb{R}^d} X^{d,m}_{t-} \odot \left( \exp_d(z) - \mathbf{1}_d \right) \, \tilde{\pi}^{d,m}(dz,dt) \\
		& \quad\quad + \int_{\mathbb{R}^d} X^{d,m}_{t-} \odot \left( \exp_d(z) - \mathbf{1}_d \right) \, \nu^d(dz) \\
		& = \left( \left( \mu_0 + \frac{\sigma_0^2}{2} \right) \mathbf{1}_d + I_\nu \right) \odot X^{d,m}_{t-} \, dt + \sigma_0 \diag(X^{d,m}_{t-}) \, dW^{d,m}_{t} \\
		& \quad\quad + \int_{\mathbb{R}^d} X^{d,m}_{t-} \odot \left( \exp_d(z) - \mathbf{1}_d \right) \, \tilde{\pi}^{d,m}(dz,dt).
	\end{aligned}
\end{equation}
Altogether, if we assume that $\kappa < \alpha/2$, the stochastic process $(X^{d,0,x}_t)_{t \in [0,T]} := (X^{d,m}_t)_{t \in [0,T]}$ is of the form \eqref{SDE deri} with $\mu^d(s,x) = \left( \mu_0 \mathbf{1}_d + I_\nu \right) \odot x$ (where $I_\nu := \int_{\mathbb{R}^d} \left( \exp_d(z) - \mathbf{1}_d \right) \nu^d(dz) \in \mathbb{R}^d$), $\sigma^d(s,x) = \sigma_0 \diag(x)$, $\eta^d_s(x,z) = x \odot \left( \exp_d(z) - \mathbf{1}_d \right)$, and $\nu^d(dz) = \frac{(2\alpha)^{d/4+1}}{(2\pi)^{d/2} \kappa^{d/4} \Vert z \Vert^{d/2}} K_{d/2}(\sqrt{2\alpha/\kappa} \, \Vert z \Vert) dz$. 

Now, we consider the same pricing problem as in Section~\ref{SecCounterpartyVasicek} consisting of the nonlinearity $f^d(t,x,v) = -\zeta \min(v,0)$ and payoff function $g^d(x) = \max\left( \min_{i=1,\dots,d} x_i - K_1, 0 \right) - \max\left( \min_{i=1,\dots,d} x_i - K_2, 0 \right) - L$, but now with the exponential Variance-Gamma model \eqref{EqExpVGModel} instead of the Vasi\v{c}ek jump model. For $f^d$ and $g^d$, we use the same parameters as in Section~\ref{SecCounterpartyVasicek}, whereas for the underlying process $X^{d,m}$, we choose $\mu_0 = -0.0001$, $\sigma_0 = 0.01$, $\alpha = 0.1$, and $\kappa = 0.0001$. For different dimensions $d \in \mathbb{N}$, we run the algorithm $10$ times either with deterministic neural networks or random neural networks to compute an approximation of $u^d(T,100,\dots,100)$ with $T = 1/2$, $\mathcal{M} = 200$, and $\delta = 0.1$ (for the other parameters $K,J,N \in \mathbb{N}$, we refer to the beginning of Section~\ref{section numerics}). The results are reported in Table~\ref{FigCounterpartyExpVG}, which are compared to the reference solution obtained via the MLP method in \cite{NW2022}.

\begin{table}[h!]
	\begin{minipage}[t][][t]{0.49\textwidth}
		\includegraphics[height = 6.0cm]{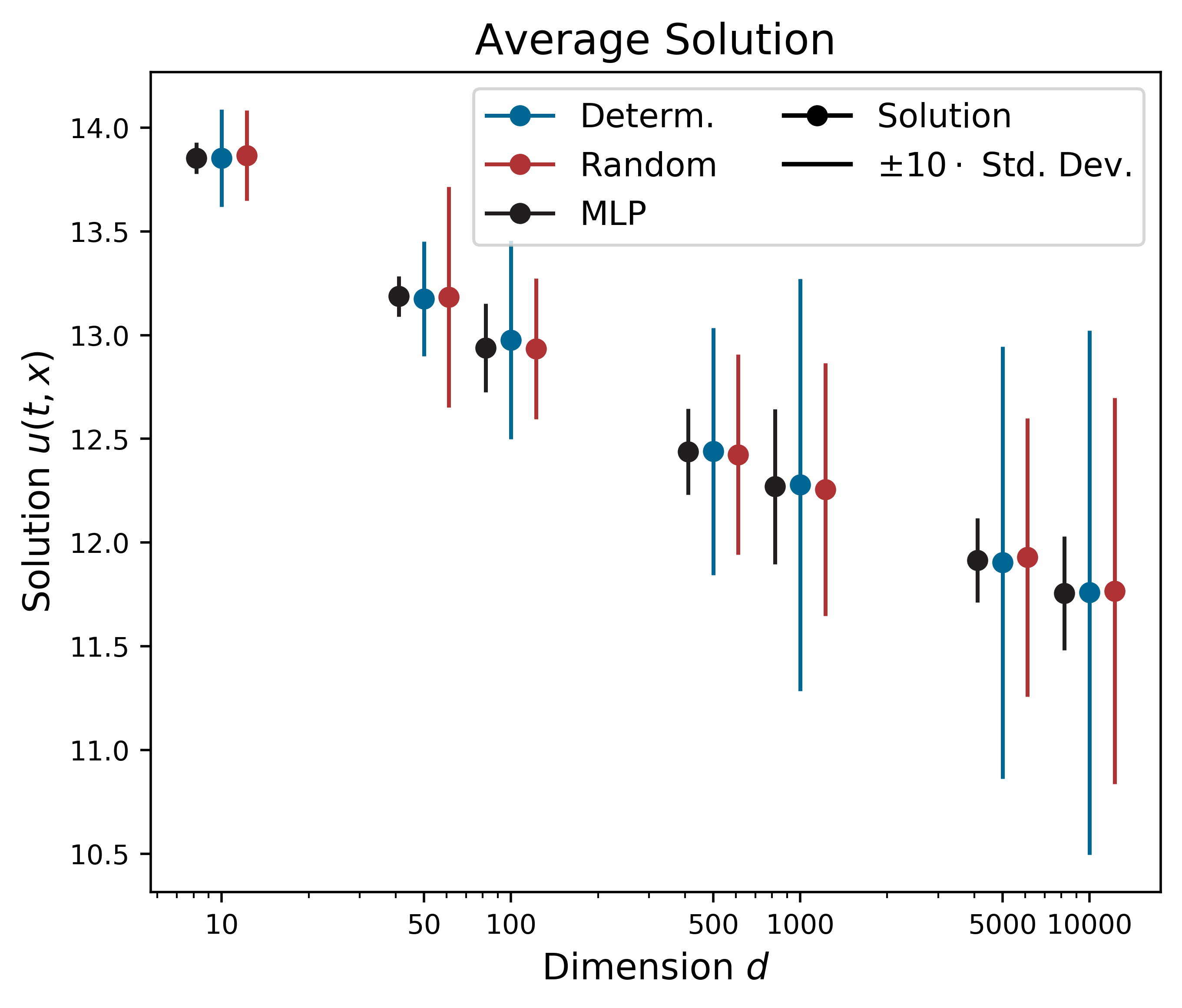}
		\vspace{0.1cm}
	\end{minipage}
	\begin{minipage}[t][][t]{0.49\textwidth}
		\includegraphics[height = 6.0cm]{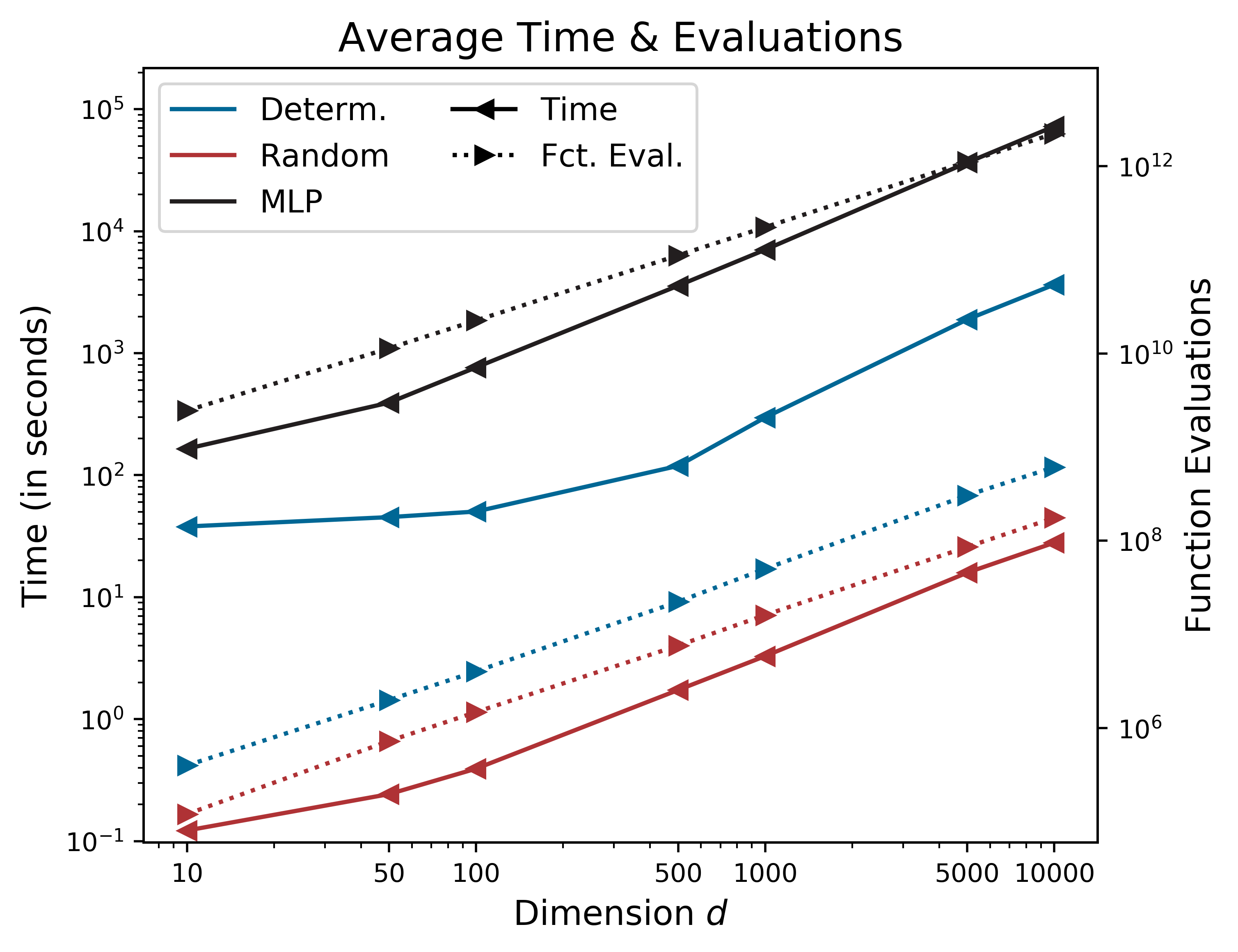}
		\vspace{0.1cm}
	\end{minipage}
	\begin{tabular}{r|R{1.7cm}R{1.7cm}R{1.7cm}|R{1.7cm}R{1.7cm}R{1.7cm}|}
		& \multicolumn{3}{c|}{\textbf{Average Solution}} & \multicolumn{3}{c|}{\textbf{Average Time \& Evaluations}} \\
		$d$ & \textcolor{MidnightBlue}{Determ.} & \textcolor{Maroon}{Random} & MLP & \textcolor{MidnightBlue}{Determ.} & \textcolor{Maroon}{Random} & MLP \\
		\hline
		10 & \textcolor{MidnightBlue}{13.8538} & \textcolor{Maroon}{13.8654} & \textcolor{Black}{13.8536} & \textcolor{MidnightBlue}{37.96} & \textcolor{Maroon}{0.12} & \textcolor{Black}{164.68} \\ 
 & \textcolor{MidnightBlue}{0.0234} & \textcolor{Maroon}{0.0218} & \textcolor{Black}{0.0076} & \textcolor{MidnightBlue}{$3.98 \cdot 10^{5}$} & \textcolor{Maroon}{$1.20 \cdot 10^{5}$} & \textcolor{Black}{$2.45 \cdot 10^{9}$} \\ 
\hline 
50 & \textcolor{MidnightBlue}{13.1757} & \textcolor{Maroon}{13.1835} & \textcolor{Black}{13.1868} & \textcolor{MidnightBlue}{45.28} & \textcolor{Maroon}{0.24} & \textcolor{Black}{392.89} \\ 
 & \textcolor{MidnightBlue}{0.0276} & \textcolor{Maroon}{0.0532} & \textcolor{Black}{0.0098} & \textcolor{MidnightBlue}{$1.99 \cdot 10^{6}$} & \textcolor{Maroon}{$7.25 \cdot 10^{5}$} & \textcolor{Black}{$1.13 \cdot 10^{10}$} \\ 
\hline 
100 & \textcolor{MidnightBlue}{12.9772} & \textcolor{Maroon}{12.9341} & \textcolor{Black}{12.9387} & \textcolor{MidnightBlue}{50.33} & \textcolor{Maroon}{0.39} & \textcolor{Black}{764.19} \\ 
 & \textcolor{MidnightBlue}{0.0478} & \textcolor{Maroon}{0.0339} & \textcolor{Black}{0.0214} & \textcolor{MidnightBlue}{$4.02 \cdot 10^{6}$} & \textcolor{Maroon}{$1.48 \cdot 10^{6}$} & \textcolor{Black}{$2.25 \cdot 10^{10}$} \\ 
\hline 
500 & \textcolor{MidnightBlue}{12.4399} & \textcolor{Maroon}{12.4241} & \textcolor{Black}{12.4385} & \textcolor{MidnightBlue}{119.17} & \textcolor{Maroon}{1.73} & \textcolor{Black}{3573.26} \\ 
 & \textcolor{MidnightBlue}{0.0596} & \textcolor{Maroon}{0.0483} & \textcolor{Black}{0.0207} & \textcolor{MidnightBlue}{$2.23 \cdot 10^{7}$} & \textcolor{Maroon}{$7.58 \cdot 10^{6}$} & \textcolor{Black}{$1.11 \cdot 10^{11}$} \\ 
\hline 
1000 & \textcolor{MidnightBlue}{12.2784} & \textcolor{Maroon}{12.2563} & \textcolor{Black}{12.2705} & \textcolor{MidnightBlue}{296.88} & \textcolor{Maroon}{3.29} & \textcolor{Black}{7047.26} \\ 
 & \textcolor{MidnightBlue}{0.0993} & \textcolor{Maroon}{0.0610} & \textcolor{Black}{0.0373} & \textcolor{MidnightBlue}{$5.00 \cdot 10^{7}$} & \textcolor{Maroon}{$1.60 \cdot 10^{7}$} & \textcolor{Black}{$2.23 \cdot 10^{11}$} \\ 
\hline 
5000 & \textcolor{MidnightBlue}{11.9042} & \textcolor{Maroon}{11.9291} & \textcolor{Black}{11.9159} & \textcolor{MidnightBlue}{1889.58} & \textcolor{Maroon}{15.88} & \textcolor{Black}{36502.57} \\ 
 & \textcolor{MidnightBlue}{0.1041} & \textcolor{Maroon}{0.0671} & \textcolor{Black}{0.0203} & \textcolor{MidnightBlue}{$3.05 \cdot 10^{8}$} & \textcolor{Maroon}{$8.57 \cdot 10^{7}$} & \textcolor{Black}{$1.11 \cdot 10^{12}$} \\ 
\hline 
10000 & \textcolor{MidnightBlue}{11.7593} & \textcolor{Maroon}{11.7670} & \textcolor{Black}{11.7559} & \textcolor{MidnightBlue}{3649.25} & \textcolor{Maroon}{27.94} & \textcolor{Black}{72363.96} \\ 
 & \textcolor{MidnightBlue}{0.1263} & \textcolor{Maroon}{0.0931} & \textcolor{Black}{0.0275} & \textcolor{MidnightBlue}{$6.10 \cdot 10^{8}$} & \textcolor{Maroon}{$1.75 \cdot 10^{8}$} & \textcolor{Black}{$2.22 \cdot 10^{12}$}
	\end{tabular}
	\caption{(Random) deep splitting approximation of the pricing problem under counterparty credit risk with exponential Variance-Gamma model \eqref{EqExpVGModel} for different $d \in \mathbb{N}$, by either using deterministic neural networks (label ``Determ.'') or random neural networks (label ``Random''). The average solutions over $10$ runs together with their standard deviations (italic font) are compared to the MLP algorithm (label ``MLP''). On the right-hand side, the average running time on a laptop with GPU (upper row, in seconds) and the average number of function evaluations (lower row, in scientific format\textsuperscript{\ref{footnote1}}) are displayed.}
	\label{FigCounterpartyExpVG}
\end{table}

\subsection{Conclusion}

The numerical experiments in Section~\ref{SecDefaultBS}-\ref{SecCounterpartyExpVG} show that different nonlinear PDEs and PIDEs can be learned by the deep splitting method, either with deterministic neural networks or random neural networks. However, the random deep splitting algorithm significantly outperforms the original version in terms of computational efficiency. In particular, in high dimensions, the classical deep splitting algorithm with deterministic networks needs a few hours to solve a given PDE or PIDE numerically, whereas its randomized version returns a solution within a couple of seconds. This remarkable speed advantage stems from the fact that random neural networks are much faster to train than deterministic networks since only the linear readout needs to be trained, which can be performed efficiently, e.g., by the least squares method. For a detailed comparison of the empirical performance between deterministic neural networks and random neural networks, we refer to \cite[Section~6]{NeufeldSchmocker2023}.

Moreover, the deep splitting method (deterministic and random) outperforms the MLP algorithm (see \cite{NW2022}) in case of PIDEs with jumps. The reason therefore lies in the high number of simulations in the MLP algorithm, in which the jumps of every simulated path of the underlying process have to be randomly generated. The latter can be very time-consuming on an average computer.

Altogether, we conclude that the random deep splitting algorithm can solve nonlinear PDEs and PIDEs in dimensions up to 10'000 within a few seconds.

\pagebreak

\section{Proof of results}
\label{section proof}

\subsection{Auxiliary results for SDEs and PIDEs}

In this section, we present some lemmas which will be used to prove our
main results, Theorems \ref{ThmMain} and \ref{ThmMainRN},
in Section \ref{section proof main result}.
\begin{lemma}                                               \label{lemma Lyaponov}
Let Assumptions \ref{assumption Lip and growth} and \ref{assumption pointwise} hold.
Let $d\in\bN$, 
$F\in C^2(\bR^d,[0,\infty))$, $\alpha:[0,T]\to[0,\infty)$ be a Borel
function with $\int_0^T\alpha(t)\,dt<\infty$,
$\tau:\Omega\to[0,T]$ be a $\bF$-stopping time, and 
$X:[0,T]\times\Omega\to\bR^d$ be a $\bF$-adapted 
c\`adl\`ag stochastic process satisfying
$\bE[F(X_0)]<\infty$ and, $\mathbb{P}$-a.s., for all $t\in[0,T]$
$$
X_t=X_0+\int_0^t\mu^d(s,X_{s-})\,ds+\int_0^t\sigma^d(s,X_{s-})\,dW^d_s
+\int_0^t\int_{\bR^d}\eta^d_s(X_{s-},z)\,\tilde{\pi}^d(dz,ds).
$$
Moreover, assume that it holds $\mathbb{P}$-a.s.~for all $t\in[0,T]$ that
\begin{align}
&
\<\nabla F(X_{t\wedge \tau}),\mu^d(t\wedge\tau,X_{t\wedge\tau})\>
+\frac{1}{2}\operatorname{Trace}
\big(\sigma^d(t\wedge\tau,X_{t\wedge\tau})
\big[\sigma^d(t\wedge\tau,X_{t\wedge\tau})\big]^T
\operatorname{Hess}_xF(X_{t\wedge\tau})\big)
\nonumber\\
&
+\int_{\bR^d}\Big(F(X_{t\wedge\tau}+\eta^d_{t\wedge\tau}(X_{t\wedge \tau},z))
-F(X_{t\wedge\tau})
-\<\nabla F(X_{t\wedge \tau}),\eta^d_{t\wedge\tau}(X_{t\wedge\tau},z)\>
\Big)\nu(dz)
\leq \alpha(t\wedge\tau)F(X_{t\wedge\tau}).
\label{est operator SDE}
\end{align}
Then we have
\begin{equation}
\label{est exp}
\bE[F(X_\tau)]\leq e^{\int_0^T\alpha(s)\,ds}\bE[F(X_0)].
\end{equation}
\end{lemma}

\begin{proof}
The proof of this lemma is analogous to that of \cite[Lemma 2.2]{CHJ2021}.
We first define a sequence of $\bF$-stopping times $\{\tau_n\}_{n=1}^\infty$
by
\begin{align*}
\tau_n:=\inf\bigg\{
&
t\in[0,T]:
\sup_{s\in[0,t]}\|X_s\|+
\sup_{s\in[0,t]}F(X_s)+\int_0^t\sum_{i,j=1}^d
\Big|\frac{\partial}{\partial x_i}F(X_s)\sigma^{d,ij}(s,X_s)
\Big|^2\,ds
\\
&
+\int_0^t\int_{\bR^d}\big|\<\nabla F(X_s),\eta^d_s(X_s,z)\>\big|^2\,\nu(dz)\,ds
\geq n\bigg\}\wedge \tau
\end{align*}
for every $n\in\bN$. Then applying It\^o's formula (see, e.g., \cite[Theorem 3.1]{GW2021})
to $F(X_t)$ yields that $\mathbb{P}$-a.s.~for all $t\in[0,T]$
and $n\in\bN$
\begin{align*}
&
F(X_{t\wedge\tau_n})
\\
&
=
F(X_0)
+\int_0^{t\wedge\tau_n}\Big[\<\nabla F(X_s),\mu^d(s,X_s)\>
+\frac{1}{2}\operatorname{Trace}\big(\sigma^d(s,X_s)[\sigma^d(s,X_s)]^T
\operatorname{Hess}_xF(X_s)\big)\Big]\,ds
\\
& \quad
+\int_0^{t\wedge\tau_n}\sum_{i,j=1}^d\frac{\partial}{\partial x_i}F(X_{s-})
\sigma^{d,ij}(s,X_{s-})\,dW^{d,j}_s
+\int_0^{t\wedge \tau_n}\int_{\bR^d}\<\nabla F(X_{s-}),\eta^d_s(X_{s-},z))\>
\,\tilde{\pi}^d(dz,ds)
\\
& \quad
+\int_0^{t\wedge \tau_n}\int_{\bR^d}\left[
F(X_{s-}+\eta^d_s(X_{s-},z))-F(X_{s-})
-\<\nabla F(X_{s-}),\eta^d_s(X_{s-},z)\>
\right]\pi^d(dz,ds).
\end{align*}
Taking expectations to the above equation shows for all $t\in[0,T]$ 
and $n\in\bN$ that
\begin{align*}
&
\bE[F(X_{t\wedge \tau_n})]
\\
&
=\bE[F(X_0)]
+\bE\Bigg[\int_0^{t\wedge\tau_n}\Big(
\<\nabla F(X_s),\mu^d(s,X_s)\>
+\frac{1}{2}\operatorname{Trace}\big(\sigma^d(s,X_s)[\sigma^d(s,X_s)]^T
\operatorname{Hess}_xF(X_s)\big)
\\
& \quad
+\int_{\bR^d}\left[
F(X_{s}+\eta^d_s(X_{s},z))-F(X_{s})
-\<\nabla F(X_{s}),\eta^d_s(X_{s},z)\>
\right]\nu^d(dz)
\Big)\,ds
\Bigg].
\end{align*}
Therefore, by \eqref{est operator SDE} and the fact that
$\tau_n(\omega)\leq \tau(\omega)$ for all $n\in\bN$ 
and $\omega\in\Omega$,
we have for all $t\in[0,T]$ and $n\in\bN$ that 
\begin{align}
\bE[F(X_{t\wedge\tau_n})]
&
\leq \bE[F(X_0)]
+\bE\left[\int_0^{t\wedge\tau_n}\alpha(s)F(X_s)\,ds\right]
\leq \bE[F(X_0)]+\int_0^t\alpha(s)\bE[F(X_{s\wedge\tau_n})]\,ds.
\label{est alpha}
\end{align}
Furthermore, by Assumption \ref{assumption pointwise} 
we notice for all $n\in\bN$ and $\omega\in\Omega$ that
$$
\big\|X_{\tau_n}-X_{\tau_n-}\big\|
=\left\|\int_{\bR^d}\eta^d_{\tau_n}(X_{\tau_n},z)\,\pi^d(dz,\{\tau_n\})\right\| 
\leq Ld^p\int_{\bR^d}1\wedge\|z\|^2\,\pi^d(dz,\{\tau_n\})
\leq Ld^p.
$$
This implies for all $t\in[0,T]$ and $n\in\bN$ that
$$
\bE[F(X_{t\wedge\tau_n})]\leq n+\sup_{\|x\|\leq n+C_d}F(x)<\infty.
$$
Hence, applying \eqref{est alpha} and Gr\"onwall's lemma yields
for all $t\in[0,T]$ and $n\in\bN$ that
$$
\bE[F(X_{t\wedge\tau_n})]\leq e^{\int_0^t\alpha(s)\,ds}\bE[F(X_0)].
$$
This together with Fatou's lemma and by choosing $t=T$ imply \eqref{est exp}.
Therefore, the proof of this lemma is completed.
\end{proof}

\begin{corollary}                                  \label{corollary Lyaponov}
	Let Assumptions \ref{assumption Lip and growth} and \ref{assumption pointwise} hold.
	For every $d \in \bN$, let 
	$\xi^{d,0}:\Omega\to\bR^d$ be an $\mathcal{F}_0/\cB(\bR^d)$-measurable random variable such that $\bE\left[\big\|\xi^{d,0}\big\|^2\right]<\infty$, and let
	$(X^{d,0}_t)_{t\in[0,T]}: 
	[0,T]\times\Omega\to \bR^d$ be the stochastic process defined in \eqref{SDE d}.
	Then it holds for all $d\in\bN$ and $t\in[0,T]$ that
	\begin{equation}                                         \label{est SDE 0 x}
		\bE\left[d^p+\big\|X^{d,0}_t\big\|^2\right]
		\leq e^{2(L^{1/2}+L)t}\left(d^p+\bE\left[\big\|\xi^{d,0}\big\|^2\right]\right).
	\end{equation}
\end{corollary}

\begin{proof}
For each $d\in\bN$, we denote by $G^d:\bR^d\to\bR$ the function 
$G^d(x):=d^p+\|x\|^2$, $x\in\bR^d$.
Then by Cauchy-Schwarz inequality and \eqref{assumption growth}, 
we notice for all $d\in\bN$, $t\in[0,T]$, and $x\in\bR^d$ that
\begin{align*}
&
\<\nabla G^d(x),\mu^d(t,x)\>
+\frac{1}{2}\operatorname{Trace}\big(\sigma^d(t,x)[\sigma^d(t,x)]^T
\operatorname{Hess}_xG^d(x)\big)
\\
& \quad
+\int_{\bR^d}\left(
G^d(x+\eta^d_t(x,z))-G^d(x)-\<\nabla G^d(x),\eta^d_t(x,z)\>
\right)\nu^d(dz)
\\
&
=2\<x,\mu^d(t,x)\>+\|\sigma^d(t,x)\|_F^2
+\int_{\bR^d}\|\eta^d_t(x,z)\|^2\,\nu^d(dz)
\\
&
\leq 2\|x\|L^{1/2}(d^p+\|x\|^2)^{1/2}+2L(d^p+\|x\|^2)
\leq 2(L^{1/2}+L)(d^p+\|x\|^2).
\end{align*}
Hence, by Lemma \ref{lemma Lyaponov} we obtain \eqref{est SDE 0 x}.
\end{proof}

\begin{corollary}
\label{corollary Lyaponov q}
Let Assumptions \ref{assumption Lip and growth} and \ref{assumption pointwise} hold, and let Assumption~\ref{AssXiLq} hold with some $q \in (2,\infty)$. Moreover,
for every $d\in\bN$, let 
$(X^{d,0}_t)_{t\in[0,T]}: 
[0,T]\times\Omega\to \bR^d$ be the stochastic process defined in \eqref{SDE d}.
Then we have for all $d\in\bN$ and $t\in[0,T]$ that
\begin{equation}                                         \label{est SDE 0 x q}
\bE\left[\big(d^p+\big\|X^{d,0}_t\big\|^2\big)^{q/2}\right]
\leq \exp\left\{[2(L+1)]^{\frac{q-2}{2}}2(L+L^{1/2})q(q-1)t\right\}
\bE\left[\big(d^p+\big\|\xi^{d,0}\big\|^2\big)^{q/2}\right].
\end{equation}
\end{corollary}
\begin{proof}
Throughout the proof of this corollary we fix a $d\in\bN$, 
and define the function $G^d:\bR^d\to\bR$ by 
$G^d(x):=(d^p+\|x\|^2)^{q/2}$, $x\in\bR^d$.
Then we notice for all $x=(x_1,\dots,x_d)\in\bR^d$ and $i,j\in\{1,2,\dots,d\}$ that
\begin{equation}
\label{1st deri G}
\frac{\partial}{\partial x_i}G^d(x)=q(d^p+\|x\|^2)^{\frac{q-2}{2}}x_i,
\end{equation}
and
\begin{equation}
\label{2nd deri G}
\frac{\partial^2}{\partial x_i \partial x_j}G^d(x)
=q(q-2)(d^p+\|x\|^2)^{\frac{q-4}{2}}x_ix_j
+q(d^p+\|x\|^2)^{\frac{q-2}{2}}\mathds{1}_{\{i=j\}}
\end{equation}
with $\mathds{1}_{\{i=j\}}$ denoting the indicator function.
Here and in the rest of this proof we set $0/0:=0$ whenever it appears.
Next, by Taylor's formula (see, e.g., \cite[Chapter 8.4.4, Theorem 4]{zorich}), 
Assumption \ref{assumption pointwise},
the Cauchy-Schwarz inequality, \eqref{assumption growth} in Assumption \ref{assumption Lip and growth}, and \eqref{2nd deri G}
we have for all $(t,x)\in[0,T]\times\bR^d$ that
\begin{equation}
	\label{G est 1}
	\begin{aligned}
		&
		\int_{\bR^d}\left(
		G^d(x+\eta^d_t(x,z))-G^d(x)-\<\nabla G^d(x),\eta^d_t(x,z)\>
		\right)\nu^d(dz) \\
		&
		=
		\int_{\bR^d}\int_0^1(1-\theta)\left[\sum_{i,j=1}^d
		\Big(\frac{\partial^2}{\partial x_i \partial x_j}
		G^d\Big)\big(x+\theta\eta^d_t(x,z)\big)\eta^{d,i}_t(x,z)\eta^{d,j}_t(x,z)
		\right]d\theta\,\nu^d(dz) \\
		&
		=\int_{\bR^d}\int_0^1(1-\theta)\bigg[
		q\big(d^p+\|x+\theta\eta^d_t(x,z)\|^2\big)^{\frac{q-2}{2}}\|\eta^d_t(x,z)\|^2
		+
		\bigg(
		\sum_{i,j=1}^dq(q-2) \\
		& \quad
		\cdot
		\big(d^p+\|x+\theta\eta^d_t(x,z)\|^2\big)^{\frac{q-4}{2}}
		\big(x_i+\theta\eta^{d,i}_t(x,z)\big)\big(x_j+\theta\eta^{d,j}_t(x,z)\big)
		\eta^{d,i}_t(x,z)\eta^{d,j}_t(x,z)
		\bigg)
		\bigg]
		\,d\theta\,\nu^d(dz) \\
		&
		\leq
		\int_{\bR^d}\bigg[
		q(d^p+2\|x\|^2+2Ld^p)^{\frac{q-2}{2}}\|\eta^d_t(x,z)\|^2
		+
		\bigg(
		q(q-2)(d^p+2\|x\|^2+2Ld^p)^{\frac{q-4}{2}} \\
		& \quad
		\cdot
		(\|x\|+(Ld^p)^{1/2})^2 \sum_{i,j=1}^d |\eta^{d,i}_t(x,z)\eta^{d,j}_t(x,z)|
		\bigg)
		\bigg]
		\,d\theta\,\nu^d(dz) \\
		&
		\leq
		[2(L+1)]^{\frac{q-2}{2}}q(d^p+\|x\|^2)^{\frac{q-2}{2}}
		\int_{\bR^d}\|\eta^d_t(x,z)\|^2\,\nu^d(dz) \\
		& \quad
		+
		[2(L+1)]^{\frac{q-2}{2}}q(q-2)(d^p+\|x\|^2)^{\frac{q-2}{2}}
		\int_{\bR^d}\|\eta^d_t(x,z)\|^2\,\nu^d(dz) \\
		&
		\leq
		[2(L+1)]^{\frac{q-2}{2}}Lq(q-1)(d^p+\|x\|^2)^{q/2}.
	\end{aligned}
\end{equation}
Furthermore, by \eqref{assumption growth} in Assumption \ref{assumption Lip and growth}, \eqref{1st deri G}, 
and \eqref{2nd deri G} it holds for all $(t,x)\in[0,T]\times\bR^d$ that
\begin{equation}
	\label{G est 2}
	\begin{aligned}
		\langle \nabla G^d(x),\mu^d(t,x)\rangle
		&
		=q(d^p+\|x\|^2)^{\frac{q-2}{2}}\langle x,\mu^d(t,x)\rangle \\
		&
		\leq
		q(d^p+\|x\|^2)^{\frac{q-2}{2}}\|x\|\cdot\|\mu^d(t,x)\| \\
		&
		\leq
		L^{1/2}q(d^p+\|x\|^2)^{q/2}.
	\end{aligned}
\end{equation}
Similarly, by using the same arguments together with the Cauchy-Schwarz inequality, it holds for all $(t,x)\in[0,T]\times\bR^d$ that
\begin{equation}
	\label{G est 3}
	\begin{aligned}
		&
		\frac{1}{2}\operatorname{Trace}\big(\sigma^d(t,x)[\sigma^d(t,x)]^T
		\operatorname{Hess}_xG^d(x)\big) \\
		&
		=\frac{1}{2}\sum_{i,j=1}^d\sum_{k=1}^d\sigma^{d,ik}(t,x)\sigma^{d,jk}(t,x)
		\frac{\partial^2}{\partial x_i \partial x_j}G^d(x) \\
		&
		=
		\frac{q(q-2)}{2}(d^p+\|x\|^2)^{\frac{q-4}{2}} \sum_{i,j=1}^d\sum_{k=1}^d
		\sigma^{d,ik}(t,x)\sigma^{d,jk}(t,x) x_i x_j \\
		& \quad
		+\frac{q}{2}(d^p+\|x\|^2)^{\frac{q-2}{2}} \sum_{i=1}^d\sum_{k=1}^d
		|\sigma^{d,ik}(t,x)|^2 \\
		&
		\leq
		\frac{q(q-2)}{2}(d^p+\|x\|^2)^{\frac{q-2}{2}}
		\big\Vert \sigma^d(t,x) \big\Vert_F^2 +\frac{q}{2}(d^p+\|x\|^2)^{\frac{q-2}{2}} \big\Vert \sigma^d(t,x) \big\Vert_F^2 \\		
		&
		\leq \frac{Lq(q-1)}{2}(d^p+\|x\|^2)^{q/2}.
	\end{aligned}
\end{equation}
Combining \eqref{G est 1}, \eqref{G est 2}, and \eqref{G est 3} shows for all
$(t,x)\in[0,T]\times\bR^d$ that
\begin{equation*}
	\begin{aligned}
		&
		\<\nabla G^d(x),\mu^d(t,x)\>
		+\frac{1}{2}\operatorname{Trace}\big(\sigma^d(t,x)[\sigma^d(t,x)]^T
		\operatorname{Hess}_xG^d(x)\big) \\
		& \quad
		+\int_{\bR^d}\left(
		G^d(x+\eta^d_t(x,z))-G^d(x)-\<\nabla G^d(x),\eta^d_t(x,z)\>
		\right)\nu^d(dz) \\
		&
		\leq
		[2(L+1)]^{\frac{q-2}{2}}2(L+L^{1/2})q(q-1)(d^p+\|x\|^2)^{q/2}.
	\end{aligned}
\end{equation*}
Therefore, by Lemma \ref{lemma Lyaponov} we obtain \eqref{est SDE 0 x q}.
\end{proof}

\begin{lemma}                                     \label{lemma Lip u}
Let $d\in\bN$, $c,\rho,L\in[0,\infty)$, $T\in(0,\infty)$, 
let $(Z^{t,x}_{s})_{s\in[t,T]}:[t,T]\times\Omega\to\bR^d$,
$t\in[0,T]$, $x\in\bR^d$, be 
$\cB([0,T])\otimes\cF/\cB(\bR^d)$-measurable functions,
let $F:[0,T]\times\bR^d\times\bR\to\bR$, 
$G:\bR^d\to\bR$,
and 
$v:[0,T]\times\bR^d\to\bR$, be Borel functions.
For all $t\in[0,T]$, $s\in[t,T]$, and all Borel functions 
$h:\bR^d\times\bR^d\to[0,\infty)$ assume that
$
\bR^d\times\bR^d\ni(y_1,y_2)\mapsto\bE[h(Z^{t,y_1}_{s},Z^{t,y_2}_{s})]
\in[0,\infty]
$ 
is measurable. Moreover, for all $x,y\in\bR^d$, $t\in[0,T]$ and $s\in[t,T]$,
$r\in[s,T]$, $v,w\in\bR$, and all Borel functions 
$h:\bR^d\times\bR^d\to[0,\infty)$ assume that
\begin{equation}                                   \label{coca 1}
Z^{t,x}_{t}=x,
\end{equation}
\begin{equation}                                   \label{coca 2}
\bE\left[\bE\Big[h(Z^{s,x'}_{r},Z^{s,y'}_{r})\Big]
\Big|_{(x',y')=(Z^{t,x}_{s},Z^{t,y}_{s})}\right]
=\bE\Big[h(Z^{t,x}_{r},Z^{t,y}_{r})\Big],
\end{equation}
\begin{equation}                                     \label{coca 3}
\max\{T|F(t,x,v)-F(t,y,w)|,|G(x)-G(y)|\}
\leq L\Big(T|v-w|+T^{-1/2}\|x-y\|\Big),
\end{equation}
\begin{equation}                                   \label{coca 4}
\bE\left[\big|G(Z^{t,x}_{T})\big|\right]
+\int_t^T\bE\left[\big|F(s,Z^{t,x}_{s},v(s,Z^{t,x}_{s}))\big|
\right]\,ds<\infty,
\end{equation}
\begin{equation}                                     \label{coca 5}
v(t,x)=\bE\left[G(Z^{t,x}_{T})
+\int_t^TF(s,Z^{t,x}_{s},v(s,Z^{t,x}_{s}))\,ds\right],
\end{equation}
and
\begin{equation}                                      \label{cond max}
\max\left\{|v(t,x)|^2,
e^{\rho(t-s)}\bE\Big[d^p+\big\|Z^{t,x}_{s}\big\|^2\Big]\right\}
\leq c(d^p+\|x\|^2).
\end{equation}
Then for all $t\in[0,T]$ and $x_1,x_2\in\bR^d$ it holds that
\begin{equation}                                     \label{u difference}
|v(t,x_1)-v(t,x_2)|\leq e^{LT}2LT^{-1/2}\cdot\sup_{s\in[t,T]}
\bE\left[\big\|Z^{t,x_1}_s-Z^{t,x_2}_s\big\|\right].
\end{equation}
\end{lemma}

\begin{proof}
By the triangle inequality, Fubini's theorem, \eqref{coca 1}--\eqref{coca 3},
and \eqref{coca 5}
we have for all $t\in[0,T]$, $s\in[t,T]$, and $x_1,x_2\in\bR^d$ that
\begin{align*}
&
\bE\left[|v(s,Z^{t,x_1}_s)-v(s,Z^{t,x_2}_s)|\right]
\\
&
=\bE\Bigg[\Big|\bE\Big[G(Z^{s,y_1}_T)-G(Z^{s,y_2}_T)
\\
& \quad
+\int_s^T\left[F(r,Z^{s,y_1}_r,v(r,Z^{s,y_1}_r))
-F(r,Z^{s,y_2}_r,v(r,Z^{s,y_2}_r))\right]dr
\Big]\Big|\Big|_{(y_1,y_2)=(Z^{t,x_1}_s,Z^{t,x_2}_s)}\Bigg]
\\
&
\leq 
\bE\left[\big|G(Z^{t,x_1}_T)-G(Z^{t,x_2}_T)\big|\right]
+\int_s^T\bE\left[\big|F(r,Z^{t,x_1}_r,v(r,Z^{t,x_1}_r))
-F(r,Z^{t,x_2}_r,v(r,Z^{t,x_2}_r))\big|\right]dr
\\
&
\leq LT^{-1/2}\bE\left[\big\|Z^{t,x_1}_T-Z^{t,x_2}_T\big\|\right]
+L\int_s^T\bE\left[\big|v(r,Z^{t,x_1}_r)-v(r,Z^{t,x_2}_r)\big|\right]dr
\\
& \quad
+LT^{-3/2}\int_s^T\bE\left[\big\|Z^{t,x_1}_r-Z^{t,x_2}_r\big\|\right]dr
\\
&
\leq 2LT^{-1/2}\cdot\sup_{r\in[t,T]}
\bE\left[\big\|Z^{t,x_1}_r-Z^{t,x_2}_r\big\|\right]
+L\int_s^T\bE\left[\big|v(r,Z^{t,x_1}_r)-v(r,Z^{t,x_2}_r)\big|\right]dr.
\end{align*}
Hence, taking \eqref{cond max} into account we apply Gr\"onwall's lemma
to obtain for all $t\in[0,T]$, $s\in[t,T]$, and $x_1,x_2\in\bR^d$ that
$$
\bE\left[\big|v(s,Z^{t,x_1}_s)-v(s,Z^{t,x_2}_s)\big|\right]
\leq e^{L(T-s)}2LT^{-1/2}
\cdot\sup_{r\in[t,T]}\bE\left[\big\|Z^{t,x_1}_r-Z^{t,x_2}_r\big\|\right].
$$
This together with \eqref{coca 1} imply \eqref{u difference}.
The proof of this lemma is therefore completed.
\end{proof}

\begin{corollary}                           \label{corollary Lip u}
Let Assumptions \ref{assumption Lip and growth} and \ref{assumption pointwise} hold.
Moreover, for every $d\in\bN$, let 
$u^d:[0,T]\times\bR^d\to\bR$ be the viscosity solution of PIDE \eqref{PDE}
given by \eqref{Feynman Kac u}.
Then there exists a positive constant $c$ (depending only on $L$ and $T$)
satisfying
for all $d\in\bN$, $t\in[0,T]$, and $x_1,x_2\in\bR^d$ that
\begin{equation}
|u^d(t,x_1)-u^d(t,x_2)|\leq c\|x_1-x_2\|.
\end{equation}
\end{corollary}

\begin{proof}
First recall that for each $d\in\bN$, $(t,x)\in[0,T]\times\bR^d$ the
process $\big(X^{d,t,x}_s\big)_{s\in[t,T]}$ is defined by \eqref{SDE}.
By \cite[Lemma~3.2]{NW2022}, 
we notice that for each $d\in\bN$, $s\in[0,T]$, and $r\in[s,T]$ the mapping 
$
\bR^d\times\bR^d\ni(x,y)\mapsto \big(X^{d,s,x}_r,X^{d,s,y}_r\big)
\in \cL_0(\Omega,\bR^d\times\bR^d)
$
is continuous and hence measurable, 
where $\cL_0(\Omega,\bR^d\times\bR^d)$ denotes the metric space of
all measurable functions from $\Omega$ to $\bR^d\times\bR^d$ equipped with the 
metric deduced by convergence in probability.
Moreover, we have for all $d\in\bN$ and all nonnegative Borel functions 
$h:\bR^d\times\bR^d\to[0,\infty)$ that the mapping
$
\cL_0(\Omega,\bR^d\times\bR^d)\ni Z \mapsto \bE\big[h(Z)\big]\in [0,\infty]
$
is measurable. 
Hence, it holds for all $d\in\bN$, $s\in[0,T]$,
$r\in[s,T]$, and all nonnegative Borel functions 
$h:\bR^d\times\bR^d\to[0,\infty)$ that the mapping
\begin{equation}                                           \label{measurability 3}
\bR^d\times\bR^d\ni(x,y)\mapsto
\bE\Big[h\big(X^{d,s,x}_{r},X^{d,s,y}_{r}\big)\Big]\in[0,\infty]
\end{equation}
is measurable.
Furthermore, \cite[Lemma~2.2]{HJKNW2020} ensures that for all $d\in\bN$, 
$t\in[0,T]$, $s\in[t,T]$, $r\in[s,T]$, $x,y\in\bR^d$ and
all nonnegative Borel functions $h:\bR^d\times\bR^d\to[0,\infty)$
it holds that
\begin{equation}                                     \label{expectation equality}
\bE\left[\bE \left[h\Big(X^{d,s,x'}_{r},X^{d,s,y'}_{r}\Big)\right]
\Big|_{(x',y')=(X^{d,t,x}_{s},X^{d,t,y}_{s})}\right]
=\bE\left[h\Big(X^{d,t,x}_{r},X^{d,t,y}_{r}\Big)\right].
\end{equation}
Then by \eqref{assumption Lip f g},
\eqref{Feynman Kac u}, \eqref{u est}, \eqref{est SDE 0 x},
\eqref{measurability 3}, and \eqref{expectation equality}, 
the application of Lemma \ref{lemma Lip u}
(with $L\cal L^{1/2}$, $T\cal T$, $Z^{t,x}\cal X^{t,x}$, $v\cal u$, 
$\rho\cal 2(L^{1/2}+L)$, 
and $c\cal \max\{1,4Le^{2L^{1/2}T}\}$ in the notation of Lemma
\ref{lemma Lip u}) 
yields for all $d\in\bN$, $t\in [0,T]$, and $x_1,x_2\in\bR^d$
that
$$
|u^d(t,x_1)-u^d(t,x_2)|\leq 2e^{L^{1/2}T}L^{1/2}T^{-1/2}\cdot\sup_{s\in[t,T]}
\bE\left[\big\|X^{d,t,x_1}_s-X^{d,t,x_2}_s\big\|\right].
$$
Hence, by \cite[Lemma~3.2]{NW2022} we get for $d\in\bN$, $t\in [0,T]$, 
and $x_1,x_2\in\bR^d$ that
$$
|u^d(t,x_1)-u^d(t,x_2)|\leq 4L^{1/2}T^{-1/2}
\exp\big\{L^{1/2}T\big[1+2L^{1/2}(T+2)\big]\big\}
\cdot\|x_1-x_2\|.
$$
The proof of this lemma is therefore completed.
\end{proof}

\begin{lemma}                                             \label{lemma time u}
Let Assumptions \ref{assumption Lip and growth} and \ref{assumption pointwise} hold.
Then for every $d\in\bN$, let 
$u^d:[0,T]\times\bR^d\to\bR$ be the viscosity solution of PIDE \eqref{PDE}
given by \eqref{Feynman Kac u}.
There exists a positive constant $c_1$ (depending only on $T$ and $L$) satisfying
for all $d\in\bN$, $x\in\bR^d$, $t\in[0,T]$, and $s\in [t,T]$ that 
\begin{equation}                                           \label{u est t s}
|u^d(t,x)-u^d(s,x)|\leq c_1(|s-t|+|s-t|^{1/2})d^{p/2}(1+\|x\|^2)^{1/2}.
\end{equation}
\end{lemma}

\begin{proof}
By \eqref{assumption Lip f g}, \eqref{assumption growth f g},
\eqref{Feynman Kac u},
\eqref{u est} , and Corollary \ref{corollary Lip u},
we first notice for all $d\in\bN$, $x\in\bR^d$, $t\in[0,T]$, and $s\in [t,T]$
that 
\begin{align*}
&
|u^d(t,x)-u^d(s,x)|
\\
&
\leq 
\bE\left[\big|g^d(X^{d,t,x}_T)-g^d(X^{d,s,x}_T)\big|\right]
+\bE\left[\int_t^s\big|f^d(r,X^{d,t,x}_r,0)\big|\,dr\right]
\\
& \quad
+\bE\left[\int_t^s\big|f^d(r,X^{d,t,x}_r,u^d(r,X^{d,t,x}_r))
-f^d(r,X^{d,t,x}_r,0)\big|\,dr\right]
\\
& \quad
+\bE\left[\int_s^T\big|f(r,X^{d,t,x}_r,u^d(r,X^{d,t,x}_r))
-f(r,X^{d,s,x}_r,u^d(r,X^{d,s,x}_r))\big|\,dr\right]
\\
&
\leq L^{1/2}T^{-1/2}\bE\left[\big\|X^{d,t,x}_T-X^{d,s,x}_T\big\|\right]
+\bE\left[\int_t^sL^{1/2}T^{-1/2}\left(d^p+\big\|X^{d,t,x}_r\big\|^2\right)^{1/2}
dr\right]
\\
& \quad
+\bE\left[\int_t^sL^{1/2}\big|u^d(r,X^{d,t,x}_r)\big|dr\right]
\\
& \quad
+\bE\left[\int_t^s\left(
L^{1/2}\big|u^d(r,X^{d,t,x}_r)-u^d(r,X^{d,s,x}_r)\big|
+L^{1/2}T^{-3/2}\big\|X^{d,t,x}_r-X^{d,s,x}_r\big\|
\right)dr\right]
\\
&
\leq L^{1/2}T^{-1/2}\bE\left[\big\|X^{d,t,x}_T-X^{d,s,x}_T\big\|\right]
+\bE\left[\int_t^sL^{1/2}T^{-1/2}\left(d^p+\big\|X^{d,t,x}_r\big\|^2\right)^{1/2}
dr\right]
\\
& \quad
+\bE\left[\int_t^s 2L\exp\{L^{1/2}T\}\left(d^p+\big\|X^{d,t,x}_r\big\|^2
\right)^{1/2}dr\right]
\\
& \quad
+\bE\left[\int_t^s\left(
L^{1/2}c\big\|X^{d,t,x}_r-X^{d,s,x}_r\big\|
+L^{1/2}T^{-3/2}\big\|X^{d,t,x}_r-X^{d,s,x}_r\big\|
\right)dr\right].
\end{align*}
This together with Corollary \ref{corollary Lyaponov}, H\"older's inequality,
and \cite[Lemma~3.1+3.4]{NW2022} imply
for all $d\in\bN$, $x\in\bR^d$, $t\in[0,T]$, and $s\in [t,T]$ that 
\begin{align*}
&
|u^d(t,x)-u^d(s,x)|\\
&
\leq |s-t|L^{1/2}(2L^{1/2}\exp\{L^{1/2}T\}+T^{-1/2})
\left(\sup_{r\in[t,s]}\bE\left[d^{p}+\|X^{d,t,x}_r\big\|^2\right]
\right)^{1/2}\\
& \quad
+L^{1/2}(2T^{-1/2}+cT)\sup_{r\in[s,T]}
\left(\bE\left[\big\|X^{d,t,x}_r-X^{d,s,x}_r\big\|^2\right]\right)^{1/2}
\\
& 
\leq
|s-t|L^{1/2}(2L^{1/2}\exp\{L^{1/2}T\}+T^{-1/2})
\left[(d^p+\|x\|^2)^{1/2}\right]\exp\{(L^{1/2}+L)(s-t)\}
\\
& \quad
+L^{1/2}(2T^{-1/2}+cT)\cdot 3(|s-t|+|s-t|^{1/2})\exp\{3LT(T+4)\}\cdot
\left[1+3L^{1/2}d^{p/2}(1+\|x\|^2)^{1/2}\right].
\end{align*}
Hence, we obtain \eqref{u est t s} with
\begin{align*}
c_1:=
&
L^{1/2}(2L^{1/2}\exp\{L^{1/2}T\}+T^{-1/2})\exp\{(L^{1/2}+L)T\}
\\
&
+L^{1/2}(2T^{-1/2}+cT)\cdot 3\exp\{3LT(T+4)\}(3L^{1/2}+1),
\end{align*}
where $c$ is the positive constant (depending only on $L$ and $T$)
defined in Corollary \ref{corollary Lip u}.
\end{proof}

\subsection{Time discretization}

In this section, we derive an approximation error of the true solution $u^d(t_n,X^{d,0}_{t_n})$ by the time discretized value function $\cV^{d}_n(\cX^{d,0,N,\delta,\cM}_{t_n})$. In order to ease notation, we introduce for every 
$d\in\bN$ and $t\in[0,T]$ the linear
integro-differential operator $\cA^d_t$ defined 
for all $\varphi\in C^{2}(\bR^d)$
and $x\in\bR^d$ by
\begin{align*}
	\cA^d_t\varphi(x):=
	&
	\langle\nabla \varphi(x),\mu^d(t,x)\rangle
	+\frac{1}{2}\operatorname{Trace}
	\big(\sigma^{d}(t,x)[\sigma^{d}(t,x)]^T\operatorname{Hess}_x\varphi(x)\big)
	\\
	&
	+\int_{\R^{d}}\left(\varphi(x+\eta^{d}_t(z,x))-\varphi(x)
	-\langle\nabla \varphi(x),\eta^{d}_t(z,x)\rangle\right)\,\nu^{d}(dz).
\end{align*}
Then we notice that for every $d\in\bN$, $n\in\{N-1,N-2,\dots,1,0\}$,
$t\in[t_n,t_{n+1})$, and $x\in\bR^d$ it holds in the viscosity sense that
$$
u^d(t,x)=u^d(t_{n+1},x)+\int_t^{t_{n+1}}\cA^d_su^d(s,x)\,ds
+\int_t^{t_{n+1}}f^d(s,x,u^d(s,x))\,ds.
$$
Now, we first show the two following auxiliary lemma before the time discretization error.

\begin{lemma}                                            \label{lemma U sol}
	Let Assumptions \ref{assumption Lip and growth}, \ref{assumption pointwise}, 
	and \ref{assumption jacobian} hold, let $d,N\in\bN$, 
	and define $U^d_N(T,x):=g^d(x)$ for all $x\in\bR^d$.
	Then for each $n\in\{N-1,N-2,\dots,1,0\}$, there is a unique Borel function 
	$U^d_n\in C([t_n,t_{n+1})\times\bR^d)$
	satisfying for all 
	$x\in\bR^d$ and $t\in[t_n,t_{n+1})$ that
	\begin{equation}                                    \label{PIDE U n}
		U^d_n(t,x)=u^d(t_{n+1},x)+\int_t^{t_{n+1}}\cA^d_sU^d(s,x)\,ds
		+(t_{n+1}-t_n)f^d(t_{n+1},x,u^d(t_{n+1},x))
	\end{equation}
	in the viscosity sense. Moreover, it holds for
	all $n\in\{N-1,N-2,\dots,1,0\}$ and $x\in\bR^d$ that
	\begin{equation}                                     \label{Feynman Kac U n}
		U^d_n(t_n,x)=\bE\left[u^d(t_{n+1},X^{d,t_n,x}_{t_{n+1}})\right]
		+(t_{n+1}-t_n)\bE\left
		[f^d(t_{n+1},X^{d,t_n,x}_{t_{n+1}},u^d(t_{n+1},X^{d,t_n,x}_{t_{n+1}}))\right].
	\end{equation}
\end{lemma}

\begin{proof}
	Note that by \eqref{assumption Lip f g}, \eqref{assumption growth f g}, 
	\eqref{u est}, and Corollary \ref{corollary Lip u}, it holds for all 
	$n\in\{N-1,N-2,\dots,1,0\}$ and $x,y\in\bR^d$ that
	\begin{align*}
		&
		|u^d(t_{n+1},x)|+(t_{n+1}-t_n)|f^d(t_{n+1},x,u^d(t_{n+1},x))|
		\\
		&
		\leq |u^d(t_{n+1},x)|+T|f^d(t_{n+1},x,0)|
		+T|f^d(t_{n+1},x,u^d(t_{n+1},x))-f^d(t_{n+1},x,0))|
		\\
		&
		\leq |u^d(t_{n+1},x)|+L^{1/2}(d^p+\|x\|^2)^{1/2}+L^{1/2}T|u^d(t_{n+1},x)|
		\\
		&
		\leq L^{1/2}\big[1+2(1+L^{1/2}T)\exp\{L^{1/2}T\}\big](d^p+\|x\|^2)^{1/2},
	\end{align*}
	and
	\begin{align*}
		&
		|u^d(t_{n+1},x)-u^d(t_{n+1},y)|
		+(t_{n+1}-t_n)|f^d(t_{n+1},x,u^d(t_{n+1},x))-f^d(t_{n+1},y,u^d(t_{n+1},y))|
		\\
		&
		\leq c\|x-y\|+L^{1/2}T|u^d(t_{n+1},x)-u^d(t_{n+1},y)|+L^{1/2}T^{-1/2}\|x-y\|
		\\
		&
		\leq\big[c(L^{1/2}T+1)+L^{1/2}T^{-1/2}\big]\|x-y\|.
	\end{align*}
	Hence, by Assumptions \ref{assumption Lip and growth}--\ref{assumption jacobian}, 
	the application of \cite[Proposition~5.16]{NW2022} 
	(with $g^d(\cdot)\cal u^d(t_{n+1},\cdot)
	+(t_{n+1}-t_n)f^d(t_{n+1},\cdot,u^d(t_{n+1},\cdot))$
	and
	$
	f^d\cal 0
	$
	in the notation of \cite[Proposition~5.16]{NW2022})
	yields \eqref{PIDE U n} 
	and \eqref{Feynman Kac U n}, which completes the proof of this lemma.
\end{proof}

\begin{lemma}
	\label{lemma V sol}
	Let Assumptions \ref{assumption Lip and growth}, \ref{assumption pointwise},
	and \ref{assumption jacobian} hold, 
	let $d,N\in\bN$, 
	and define $V^d_N(T,x):=g^d(x)$ for all $x\in\bR^d$.
	Then for each $n\in\{N-1,N-2,\dots,1,0\}$, there is a unique Borel function 
	$V^d_n\in C([t_n,t_{n+1})\times\bR^d)$
	satisfying for all $x\in\bR^d$ and $t\in[t_n,t_{n+1})$ that
	\begin{equation}                                    \label{PIDE V n}
		V^d_n(t,x)=V^d_{n+1}(t_{n+1},x)+\int_t^{t_{n+1}}\cA^d_sV^d_n(s,x)\,ds
		+(t_{n+1}-t_n)f^d(t_{n+1},x,V^d_{n+1}(t_{n+1},x))
	\end{equation}
	in the viscosity sense. Moreover, 
	we have for all $n\in\{N-1,N-2,\dots,1,0\}$ and $x\in\bR^d$ that
	\begin{equation}                                     \label{Feynman Kac V n}
		V^d_n(t_n,x)=\bE\left[V^d_{n+1}(t_{n+1},X^{d,t_n,x}_{t_{n+1}})\right]
		+(t_{n+1}-t_n)\bE\left[f^d(t_{n+1},X^{d,t_n,x}_{t_{n+1}},
		V^d_{n+1}(t_{n+1},X^{d,t_n,x}_{t_{n+1}}))\right].
	\end{equation}
\end{lemma}

\begin{proof}
	By \eqref{assumption Lip f g} and \eqref{assumption growth f g}, 
	we first notice for all $x,y\in\bR^d$ that
	\begin{align}
		&
		|V^d_N(t_{N},x)|+(t_{N}-t_{N-1})|f^d(t_{N},x,V^d_N(t_{N},x))|
		\nonumber
		\\
		&
		\leq |g^d(x)|+T|f^d(T,x,0)|
		+T|f^d(T,x,g^d(x))-f^d(T,x,0))|
		\nonumber
		\\
		&
		\leq |g^d(x)|+L^{1/2}(d^p+\|x\|^2)^{1/2}+L^{1/2}T|g^d(x)|
		\nonumber
		\\
		&
		\leq L^{1/2}(2+L^{1/2}T)(d^p+\|x\|^2)^{1/2},
		\label{LG V N}
	\end{align}
	and
	\begin{align}
		&
		|V^d_N(t_N,x)-V^d_N(t_N,y)|
		+(t_N-t_{N-1})|f^d(t_N,x,V^d_N(t_N,x))-f^d_N(t_N,y,V^d_N(t_N,y))|
		\nonumber
		\\
		&
		\leq |g^d(x)-g^d(y)|+L^{1/2}T|g^d(x)-g^d(y)|+L^{1/2}T^{-1/2}\|x-y\|
		\nonumber
		\\
		&
		\leq L^{1/2}(L^{1/2}T^{1/2}+2T^{-1/2})\|x-y\|.
		\label{Lip V N}
	\end{align}
	Therefore, by Assumptions 
	\ref{assumption Lip and growth}--\ref{assumption jacobian},
	the application of \cite[Proposition~5.16]{NW2022} 
	(with
	$g^d(\cdot)\cal V^d_N(t_{N},\cdot)
	+(t_{N}-t_{N-1})f^d(t_{N},\cdot,V^d(t_{N},\cdot))$ 
	and
	$f^d\cal 0$
	in the notation of \cite[Proposition~5.16]{NW2022} )
	yields that 
	\eqref{PIDE V n} and \eqref{Feynman Kac V n} hold in the case of $n=N-1$.
	Furthermore, by \eqref{assumption Lip f g}, \eqref{assumption growth f g},
	\eqref{Feynman Kac V n} with $n=N-1$,
	\eqref{LG V N}, \eqref{Lip V N}, Corollary \ref{corollary Lyaponov},
	Lemma 3.2 in \citep{NW2022},
	and Jensen's inequality
	we obtain for all $x,y\in\bR^d$ that
	\begin{align*}
		|V^d_{N-1}(t_{N-1},x)|
		& \leq
		\bE\left[\big|V^d_N(t_{N},X^{d,t_{N-1},x}_{t_N})\big|\right]
		+(t_{N}-t_{N-1})
		\bE\left[
		\big|f^d(t_{N},X^{d,t_{N-1},x}_{t_N},V^d_N(t_{N},X^{d,t_{N-1},x}_{t_N}))\big|
		\right]
		\nonumber
		\\
		&
		\leq \bE\left[\big|g^d(X^{d,t_{N-1},x}_{t_N})\big|\right]
		+T\bE\left[\big|f^d(T,X^{d,t_{N-1},x}_{t_N},0)\big|\right]
		\\
		& \quad
		+T\bE\left[\big|f^d(T,X^{d,t_{N-1},x}_{t_N},g^d(X^{d,t_{N-1},x}_{t_N}))-
		f^d(T,X^{d,t_{N-1},x}_{t_N},0))\big|\right]
		\nonumber
		\\
		&
		\leq (1+L^{1/2}T)\bE\left[\big|g^d(X^{d,t_{N-1},x}_{t_N})\big|\right]
		+L^{1/2}\bE\left[\big(d^p+\big\|X^{d,t_{N-1},x}_{t_N}\big\|^2\big)^{1/2}\right]
		\nonumber
		\\
		&
		\leq L^{1/2}(2+L^{1/2}T)
		\bE\left[\big(d^p+\big\|X^{d,t_{N-1},x}_{t_N}\big\|^2\big)^{1/2}\right]
		\\
		&
		\leq L^{1/2}(2+L^{1/2}T)e^{(L^{1/2}+L)T}(d^p+\|x\|^2)^{1/2},
	\end{align*}
	and
	\begin{align*}
		&
		|V^d_{N-1}(t_{N-1},x)-V^d_{N-1}(t_{N-1},y)|
		\\
		&
		\leq
		\bE\left[\big|V^d_N(t_N,X^{t_{N-1},x}_{t_N})
		-V^d_N(t_N,X^{t_{N-1},y}_{t_N})\big|\right]
		\\
		& \quad
		+(t_N-t_{N-1})\bE\left[\big|
		f^d(t_N,X^{t_{N-1},x}_{t_N},V^d_N(t_N,X^{t_{N-1},x}_{t_N}))
		-f^d(t_N,X^{t_{N-1},y}_{t_N},V^d_N(t_N,X^{t_{N-1},y}_{t_N}))
		\big|\right]
		\nonumber
		\\
		&
		\leq (1+L^{1/2}T)\bE\left[\big|g^d(X^{t_{N-1},x}_{t_N})
		-g^d(X^{t_{N-1},y}_{t_N})\big|\right]
		+L^{1/2}T^{-1/2}\bE\left[\big\|X^{t_{N-1},x}_{t_N}
		-X^{t_{N-1},y}_{t_N}\big\|\right]
		\nonumber
		\\
		&
		\leq L^{1/2}(L^{1/2}T^{1/2}+2T^{-1/2})
		\bE\left[
		\big\|X^{t_{N-1},x}_{t_N}-X^{t_{N-1},y}_{t_N}\big\|\right]
		\\
		&
		\leq 2L^{1/2}(L^{1/2}T^{1/2}+2T^{-1/2})e^{2LT(T+2)}\|x-y\|.
	\end{align*}
	This together with \eqref{assumption Lip f g} and \eqref{assumption growth f g}
	imply for all $x,y\in\bR^d$ that
	\begin{align*}
		&
		|V^d_{N-1}(t_{N-1},x)|+(t_{N-1}-t_{N-2})|f^d(t_{N-1},x,V^d_{N-1}(t_{N-1},x))|
		\nonumber
		\\
		&
		\leq |V^d_{N-1}(t_{N-1},x)|
		+T|f^d(t_{N-1},x,0)|
		+T|f^d(t_{N-1},x,V^d_{N-1}(t_{N-1},x))-f^d(t_{N-1},x,0))|
		\nonumber
		\\
		&
		\leq |V^d_{N-1}(t_{N-1},x)|+L^{1/2}(d^p+\|x\|^2)^{1/2}
		+L^{1/2}T|V^d_{N-1}(t_{N-1},x)|
		\nonumber
		\\
		&
		= L^{1/2}(d^p+\|x\|^2)^{1/2}+(1+L^{1/2}T)|V^d(t_{N-1},x)|
		\\
		&
		\leq L^{1/2}\left[(2+L^{1/2}T)^2e^{(L^{1/2}+L)T}+1\right]
		\cdot(d^p+\|x\|^2)^{1/2},
	\end{align*}
	and
	\begin{align*}
		&
		|V^d_{N-1}(t_{N-1},x)-V^d_{N-1}(t_{N-1},y)|
		\\
		& \quad
		+(t_{N-1}-t_{N-2})|f^d(t_{N-1},x,V^d_{N-1}(t_{N-1},x))
		-f^d(t_{N-1},y,V^d_{N-1}(t_{N-1},y))|
		\nonumber
		\\
		&
		\leq (1+L^{1/2}T)|V^d_{N-1}(t_{N-1},x)-V^d_{N-1}(t_{N-1},y)|
		+L^{1/2}T^{-1/2}\|x-y\|
		\nonumber
		\\
		&
		\leq (1+L^{1/2}T)2L^{1/2}(L^{1/2}T^{1/2}+2T^{-1/2})e^{2LT(T+2)}\|x-y\|
		+L^{1/2}T^{-1/2}\|x-y\|.
	\end{align*}
	Hence, by Assumptions 
	\ref{assumption Lip and growth}--\ref{assumption jacobian},
	the application of \cite[Proposition~5.16]{NW2022} 
	(with
	$g^d(\cdot)\cal V^d_{N-1}(t_{N-1},\cdot)
	+(t_{N-1}-t_{N-2})f^d(t_{N-1},\cdot,V^d_{N-1}(t_{N-1},\cdot))$ 
	and
	$f^d\cal 0$
	in the notation of \cite[Proposition~5.16]{NW2022} )
	yields that 
	\eqref{PIDE V n} and \eqref{Feynman Kac V n} hold in the case of $n=N-2$.
	Then by backward recursion and \cite[Proposition~5.16]{NW2022},
	we obtain \eqref{PIDE V n} and \eqref{Feynman Kac V n}.
	Thus, the proof of this lemma is completed.
\end{proof}

\begin{proposition}
	\label{PropTimeDiscr}
	Let Assumptions \ref{assumption Lip and growth}--\ref{assumption time Holder} hold. Then, there exists a constant $\widehat{C} > 0$ (depending only on $T$, $L$, $L_1$, $L_2$, and $C_\eta$)
	such that for every $d,N,\cM \in \bN$, $\delta \in (0,1)$ with $\cM\geq \delta^{-2} C_\eta d^p$, and $n\in\{N,N-1,\dots,1,0\}$
	\begin{equation*}
		\sup_{n\in\{N,N-1,\dots,1,0\}}\bE\left[\left\vert u^d(t_n,X^{d,0}_{t_n})
		-\cV^{d}_n(\cX^{d,0,N,\delta,\cM}_{t_n}) \right\vert^2\right]
		\leq \widehat{C} d^p\left(d^p+\bE\big[\big\|\xi^{d,0}\big\|^2\right)e^{d,N,\delta,\cM},
	\end{equation*}
\end{proposition}
\begin{proof}
	Throughout the proof of this theorem, 
	we fix $d \in \bN$, $x\in\bR^d$, $\delta\in(0,1)$,
	and $\cM\in \bN$ with $\cM\geq \delta^{-2} C_\eta d^p$.
	Then let $\{U^d_n\}_{n=0}^N$
	and $\{V^d_n\}_{n=0}^N$ be continuous functions introduced in
	Lemmas \ref{lemma U sol} and \ref{lemma V sol}, respectively.
	\quad\\
	\textit{Step 1. }
	By \eqref{assumption Lip f g}, \eqref{Feynman Kac u}, and \eqref{Feynman Kac U n}, 
	it holds for all $N\in\bN$ and $n\in\{N-1,N-2,\dots,1,0\}$ that
	\begin{align}
		&
		\big|u^d(t_n,x)-U^d_n(t_n,x)\big|
		\nonumber\\
		&
		=\left|
		\bE\left[
		\int_{t_n}^{t_{n+1}}\left(
		f^d(s,X^{d,t_n,x}_s,u^d(s,X^{d,t_n,x}_s))
		-f^d(t_{n+1},X^{d,t_n,x}_{t_{n+1}},u^d(t_{n+1},X^{d,t_n,x}_{t_{n+1}}))
		\right)ds
		\right]
		\right|
		\nonumber\\
		&
		\leq
		\int_{t_n}^{t_{n+1}}T^{-1}L^{1/2}\Big(
		T(t_{n+1}-t_n)^{1/2}
		+T\bE\left[\big|u^d(s,X^{d,t_n,x}_{s})
		-u^d(t_{n+1},X^{d,t_n,x}_{t_{n+1}})\big|\right]
		\nonumber\\
		& \quad
		+T^{-1/2}\bE\left[\big\|X^{d,t_n,x}_s-X^{d,t_n,x}_{t_{n+1}}\big\|\right]
		\Big)\,ds.                                      \label{est step 1.1}
	\end{align}
	Furthermore, H\"older's inequality, \cite[Lemma~3.3]{NW2022},
	Corollary \ref{corollary Lyaponov} and \ref{corollary Lip u}, 
	and Lemma \ref{lemma time u} ensure 
	for all $N\in\bN$ and $n\in\{0,\dots,N-1\}$ that
	\begin{align}
		&
		\bE\left[\big|u^d(s,X^{d,t_n,x}_s)-u^d(t_{n+1},X^{d,t_n,x}_{t_{n+1}})\big|\right]
		\nonumber\\
		&
		\leq 
		\bE\left[\big|u^d(s,X^{d,t_n,x}_s)-u^d(s,X^{d,t_n,x}_{t_{n+1}})\big|\right]
		+
		\bE\left[\big|u^d(s,X^{d,t_n,x}_{t_{n+1}})
		-u^d(t_{n+1},X^{d,t_n,x}_{t_{n+1}})\big|\right]
		\nonumber\\
		&
		\leq c\left(\bE\left[\big\|X^{d,t_n,x}_s-X^{d,t_n,x}_{t_{n+1}}\big\|^2
		\right]\right)^{1/2}
		+c_1[(t_{n+1}-t_n)+(t_{n+1}-t_n)^{1/2}]d^{p/2}
		\left(\bE\left[1+\big\|X^{d,t_n,x}_{t_{n+1}}\big\|^2\right]\right)^{1/2}
		\nonumber\\
		&
		\leq c_2^{1/2}c[(t_{n+1}-t_n)+(t_{n+1}-t_n)^2]^{1/2}(d^p+\|x\|^2)^{1/2}
		\nonumber\\
		& \quad
		+c_1[(t_{n+1}-t_n)+(t_{n+1}-t_n)^{1/2}]d^{p/2}e^{(L^{1/2}+L)T}(d^p+\|x\|^2)^{1/2}
		\nonumber\\
		&
		\leq [c_1e^{(L^{1/2}+L)T}+c_2^{1/2}c]
		[(t_{n+1}-t_n)+(t_{n+1}-t_n)^{1/2}]d^{p/2}(d^p+\|x\|^2)^{1/2},
		\label{est step 1.2}
	\end{align}
	where $c_2:=12L(1+6LT)e^{(1+6L)T}$ (see \cite[Equation~(3.6)]{NW2022}).
	By H\"older's inequality and \cite[Lemma~3.3]{NW2022}, we also obtain for
	all $N\in\bN$ and $n\in\{0,\dots,N-1\}$ that
	\begin{align}
		\bE\left[\big\|X^{d,t_n,x}_s-X^{d,t_n,x}_{t_{n+1}}\big\|\right]
		&
		\leq \left(\bE\left[\big\|X^{d,t_n,x}_s-X^{d,t_n,x}_{t_{n+1}}\big\|^2
		\right]\right)^{1/2}
		\nonumber\\
		&
		\leq c^{1/2}_2[(t_{n+1}-t_n)+(t_{n+1}-t_n)^{1/2}](d^p+\|x\|^2)^{1/2}.
		\label{est step 1.3}
	\end{align}
	Then combining \eqref{est step 1.1}--\eqref{est step 1.3} yields for all
	$N\in\bN$ and $n\in\{0,\dots,N-1\}$ that
	\begin{align}
		\big|u^d(t_n,x)-U^d_n(t_n,x)\big|
		&
		\leq L^{1/2}(t_{n+1}-t_n)\Big\{(t_{n+1}-t_n)^{1/2}
		\nonumber\\
		& \quad
		+
		[(t_{n+1}-t_n)+(t_{n+1}-t_n)^{1/2}]d^{p/2}(d^p+\|x\|^2)^{1/2}
		\big[c_1e^{(L^{1/2}+L)T}+c^{1/2}_2c\big]
		\nonumber\\
		& \quad
		+c^{1/2}_2T^{-3/2}[(t_{n+1}-t_n)+(t_{n+1}-t_n)^{1/2}](d^p+\|x\|^2)^{1/2}
		\Big\}
		\nonumber\\
		&
		\leq c_3(t_{n+1}-t_n)
		[(t_{n+1}-t_n)+(t_{n+1}-t_n)^{1/2}]d^{p/2}(d^p+\|x\|^2)^{1/2},
		\label{est step 1.4}
	\end{align}
	where 
	$c_3:=L^{1/2}\big(1+\big[c_1e^{(L^{1/2}+L)T}
	+c^{1/2}_2c\big]+c_2^{1/2}T^{-3/2}\big)$.
	\quad\\
	\textit{Step 2. }
	Notice that $V^d_{N-1}(t_{N-1},x)=U^d_{N-1}(t_{N-1},x)$ for all $N\in\bN$.
	Thus, by
	\eqref{est step 1.4} we have for all $N\in\bN$ that
	\begin{align}                                      
		|V^d_{N-1}(t_{N-1},x)-u^d(t_{N-1},x)|
		&
		=|U^d_{N-1}(t_{N-1},x)-u^d(t_{N-1},x)|
		\nonumber\\
		&
		\leq c_3TN^{-1}(TN^{-1}+T^{1/2}N^{-1/2})d^{p/2}(d^p+\|x\|^2)^{1/2}.
		\label{est V N-1}
	\end{align}
	Therefore, by \eqref{assumption Lip f g}, 
	\eqref{Feynman Kac U n}, \eqref{Feynman Kac V n},
	\eqref{est step 1.4}, H\"older's inequality,
	and Corollary \ref{corollary Lyaponov}, it holds for all 
	$N\in\bN\cap[2,\infty)$ that
	\begin{align*}
		&
		|V^d_{N-2}(t_{N-2},x)-U^d_{N-2}(t_{N-2},x)|
		\\
		&
		\leq
		\bE\left[\left|V^d_{N-1}(t_{N-1},X^{t_{N-2},x}_{t_{N-1}})
		-u^d(t_{N-1},X^{t_{N-2},x}_{t_{N-1}})\right|\right]
		\\ 
		& \quad
		+(t_{N-1}-t_{N-2})L^{1/2}
		\bE\left[\left|V^d_{N-1}(t_{N-1},X^{t_{N-2},x}_{t_{N-1}})
		-u^d(t_{N-1},X^{t_{N-2},x}_{t_{N-1}})\right|\right]
		\\
		& 
		\leq c_3(1+L^{1/2}TN^{-1})TN^{-1}(TN^{-1}+T^{1/2}N^{-1/2})d^{p/2}
		\bE\Big[\Big(d^p+\big\|X^{t_{N-2},x}_{t_{N-1}}\big\|^2\Big)^{1/2}\Big]
		\\
		& 
		\leq c_3(1+L^{1/2}TN^{-1})TN^{-1}(TN^{-1}+T^{1/2}N^{-1/2})d^{p/2}
		\exp\{(L^{1/2}+L)(t_{N-1}-t_{N-2})\}(d^p+\|x\|^2)^{1/2}.
	\end{align*}
This together with \eqref{est step 1.4} yields for all $N\in\bN\cap[2,\infty)$ that
\begin{align*}
|V^d_{N-2}(t_{N-2},x)-u^d(t_{N-2},x)|
\leq
&
c_3[(1+L^{1/2}TN^{-1})+1]TN^{-1}(TN^{-1}+T^{1/2}N^{-1/2})
\\
& \cdot
d^{p/2}\exp\{(L^{1/2}+L)(t_{N-1}-t_{N-2})\}(d^p+\|x\|^2)^{1/2}.
\end{align*}	
	Hence, by iteration we obtain for all $N\in\bN\cap[2,\infty)$ and
	$k\in\{2,\dots,N\}$ that
	\begin{align}
		&
		|V^d_{N-k}(t_{N-k},x)-u^d(t_{N-k},x)|
		\nonumber\\
		&
		\leq c_3\left(\sum_{j=0}^{k-1}a_j(N)\right)
		TN^{-1}(TN^{-1}+T^{1/2}N^{-1/2})d^{p/2}
		\exp\{(L^{1/2}+L)(k-1)TN^{-1}\}(d^p+\|x\|^2)^{1/2},
		\label{iteration V U 1}
	\end{align}
where
$$
a_j(N):=(1+L^{1/2}TN^{-1})^j, \quad j\in\{0,1,\dots,N-1\}.
$$	
By the fact that 
$(1+ak^{-1})^k\leq e^a$ for all $a\geq 0$ and $k\in\bN$, we notice for all $N\in\bN\cap[2,\infty)$ that
$$
\sum_{j=1}^{N-1}a_j(N)=\frac{(1+L^{1/2}TN^{-1})^{N-1}-1}{L^{1/2}TN^{-1}}
\leq
\frac{e^{L^{1/2}T}-1}{L^{1/2}TN^{-1}}.
$$
This together with \eqref{est V N-1} and \eqref{iteration V U 1}
	ensure for all $N\in\bN$ and $n\in\{N,N-1,\dots,1,0\}$ that
	\begin{equation}                                           \label{est sup V u}
		|V^d_n(t_n,x)-u^d(t_n,x)|
		\leq (L^{-1/2}+T)c_3(TN^{-1}+T^{1/2}N^{-1/2})d^{p/2}
		\exp\{(2L^{1/2}+L)T\}(d^p+\|x\|^2)^{1/2}.
	\end{equation}
	Hence, by Corollary \ref{corollary Lyaponov} we obtain for all $N\in\bN$ that
	\begin{align}
		&
		\sup_{n\in\{N,N-1,\dots,1,0\}}\bE\left[\big|V^d_n(t_n,X^{d,0}_{t_n})
		-u^d(t_n,X^{d,0}_{t_n})\big|^2\right]
		\nonumber\\
		&
		\leq 4c_3^2(L^{-1/2}+T)^2(TN^{-1}+T^{1/2}N^{-1/2})^2d^{p}
		\exp\{2(2L^{1/2}+L)T\}\bE\left[d^p+\big\|X^{d,0}_{t_n}\big\|^2\right]
		\nonumber\\
		&
		\leq 4c_3^2(L^{-1/2}+T)^2(TN^{-1}+T^{1/2}N^{-1/2})^2d^{p}
		\exp\{6(L^{1/2}+L)T\}\left(d^p+\bE\big[\big\|\xi^{d,0}\big\|^2\big]\right)
		\nonumber\\
		&
		\leq 4c_3^2(L^{-1/2}+T)^2T(1+T^{1/2})^2
		N^{-1} \cdot d^p
		\exp\{6(L^{1/2}+L)T\}\left(d^p+\bE\big[\big\|\xi^{d,0}\big\|^2\big]\right).
		\label{est sup expectation V u}
	\end{align}
	
	\quad\\
	\textit{Step 3. }
	By \eqref{Feynman Kac V n} and the Markov property
	of $(X^{d,0}_t)_{t\in[0,T]}$ 
	(see, e.g., \cite[Theorem~17.2.3]{CE}), we notice that 
	it holds for all $N\in\bN$ and $n\in\{0,\dots,N-1\}$ that
	\begin{align}
		&
		\bE\left[
		V^d_{n+1}(t_{n+1},X^{d,0}_{t_{n+1}})
		+(t_{n+1}-t_n)f^d(t_{n+1},X^{d,0}_{t_{n+1}},V^d_{n+1}(t_{n+1},X^{d,0}_{t_{n+1}}))
		\big|\cF_{t_n}
		\right]
		\nonumber\\
		&
		=\bE\left[
		V^d_{n+1}(t_{n+1},X^{d,0}_{t_{n+1}})
		+(t_{n+1}-t_n)f^d(t_{n+1},X^{d,0}_{t_{n+1}},V^d_{n+1}(t_{n+1},X^{d,0}_{t_{n+1}}))
		\big|X^{d,0}_{t_n}
		\right]
		\nonumber\\
		&
		=V^d_n(t_n,X^{d,0}_{t_n}).
		\label{Markov}
	\end{align}
	For all $N\in\bN$ and $n\in\{N,N-1,\dots,1,0\}$, define
	\begin{equation}
		\label{def error E}
		E^d_{N,\delta,\cM}(n):=\bE\left[\big|\cV^{d}_n(\cX^{d,0,N,\delta,\cM}_{t_n})
		-V^d(t_n,X^{d,0}_{t_n})\big|^2\right],
	\end{equation}
	where $\cV^{d}_n(\cdot):\bR^d\to\bR$, $n\in\{N,N-1,\dots,1,0\}$ is defined by
	\eqref{def cV}.
	Note that the Markov property of 
	$\big(\cX^{d,0,N,\delta,\cM}_t\big)_{t\in[0,T]}$ ensures for all
	$N\in\bN$ and $n\in\{N-1,N-2,\dots,1,0\}$ that
	\begin{align}
		&
		\cV^{d}_n(\cX^{d,0,N,\delta,\cM}_{t_{n}})
		\nonumber\\
		&
		=
		\bE\Big[
		\cV^{d}_{n+1}(\cX^{d,0,N,\delta,\cM}_{t_{n+1}})
		+(t_{n+1}-t_n)f^d(t_{n+1},\cX^{d,0,N,\delta,\cM}_{t_{n+1}},
		\cV^{d}_{n+1}(\cX^{d,0,N,\delta,\cM}_{t_{n+1}}))
		\Big|\cX^{d,0,N,\delta,\cM}_{t_{n}}
		\Big]
		\nonumber\\
		&
		=\bE\Big[
		\cV^{d}_{n+1}(\cX^{d,0,N,\delta,\cM}_{t_{n+1}})
		+(t_{n+1}-t_n)f^d(t_{n+1},\cX^{d,0,N,\delta,\cM}_{t_{n+1}},
		\cV^{d}_{n+1}(\cX^{d,0,N,\delta,\cM}_{t_{n+1}}))
		\Big|\cF_{t_n}
		\Big].
		\label{markov Euler}
	\end{align}
	Then by the tower property and Jensen's inequality for conditional expectations, 
	Young's inequality, 
	\eqref{assumption Lip f g}, \eqref{error Euler} with constant $\widetilde{C} > 0$ (depending only on $T$, $L$, $L_1$, $L_2$, and $C_\eta$, see \eqref{EqDefCTilde} below), \eqref{Feynman Kac V n},
	\eqref{Markov}, and \eqref{markov Euler},
	it holds for all $N\in\bN$ and $n\in\{N-1,N-2,\dots,1,0\}$ that
	\begin{align}
		&
		E^d_{N,\delta,\cM}(n)
		\nonumber\\
		&
		\leq 
		\bE\Big[\big|\cV^d_{n+1}(\cX^{d,0,N,\delta,\cM}_{t_{n+1}})
		-V^d_{n+1}(t_{n+1},X^{d,0}_{t_{n+1}})
		+(t_{n+1}-t_n)f^d(t_{n+1},\cX^{d,0,N,\delta,\cM}_{t_{n+1}}
		,\cV^d_{n+1}(\cX^{d,0,N,\delta,\cM}_{t_{n+1}}))
		\nonumber\\
		& \quad
		-(t_{n+1}-t_n)f^d(t_{n+1},X^{d,0}_{t_{n+1}}
		,V^d_{n+1}(t_{n+1},X^{d,0}_{t_{n+1}}))
		\big|^2\Big]
		\nonumber\\
		&
		\leq E^d_{N,\delta,\cM}(n+1)
		+2N^{-1}T\bE\Big[
		\big|\cV^d_{n+1}(\cX^{d,0,N,\delta,\cM}_{t_{n+1}})
		-V^d_{n+1}(t_{n+1},X^{d,0}_{t_{n+1}})\big|
		\nonumber\\
		& \quad
		\cdot
		\big|
		f^d(t_{n+1},\cX^{d,0,N,\delta,\cM}_{t_{n+1}}
		,\cV^d_{n+1}(\cX^{d,0,N,\delta,\cM}_{t_{n+1}}))
		-f^d(t_{n+1},X^{d,0}_{t_{n+1}}
		,V^d_{n+1}(t_{n+1},X^{d,0}_{t_{n+1}}))
		\big|
		\Big]
		\nonumber\\
		& \quad
		+N^{-2}T^2\bE\Big[\big|
		f^d(t_{n+1},\cX^{d,0,N,\delta,\cM}_{t_{n+1}}
		,\cV^d_{n+1}(\cX^{d,0,N,\delta,\cM}_{t_{n+1}}))
		-f^d(t_{n+1},X^{d,0}_{t_{n+1}}
		,V^d_{n+1}(t_{n+1},X^{d,0}_{t_{n+1}}))
		\big|^2
		\Big]
		\nonumber\\
		&
		\leq (1+N^{-1}T)E^d_{N,\delta,\cM}(n+1)
		+N^{-1}T(1+N^{-1}T)\bE\Big[\big|
		f^d(t_{n+1},\cX^{d,0,N,\delta,\cM}_{t_{n+1}}
		,\cV^d_{n+1}(\cX^{d,0,N,\delta,\cM}_{t_{n+1}}))
		\nonumber\\
		& \quad
		-f^d(t_{n+1},X^{d,0}_{t_{n+1}}
		,V^d_{n+1}(t_{n+1},X^{d,0}_{t_{n+1}}))
		\big|^2
		\Big]
		\nonumber\\
		&
		\leq (1+N^{-1}T)E^d_{N,\delta,\cM}(n+1)
		\nonumber\\
		& \quad
		+N^{-1}T(1+N^{-1}T)\left( 
		2LE^d_{N,\delta,\cM}(n+1)+2LT^{-3}\bE\left[\big\|
		\cX^{d,0,N,\delta,\cM}_{t_{n+1}}-X^{d,0}_{t_{n+1}}
		\big\|^2\right]
		\right)
		\nonumber\\
		&
		\leq 
		[1+N^{-1}T(1+2L(T+1))]E^d_{N,\delta,\cM}(n+1)+2N^{-1}L(T+1)T^{-2}
		\widetilde{C}\left(d^p+\bE\big[\big\|\xi^{d,0}\big\|^2\big]\right)e^{d,N,\delta,\cM},
		\label{est n n+1}
	\end{align}
	where $e^{d,N,\delta,\cM}$ is defined in Remark~\ref{RemTimeDiscr}.
	Moreover, by \eqref{assumption Lip f g} and \eqref{est n n+1}, and
	the fact that $1+y\leq e^y$ for all $y\geq 0$, 
	we notice for all $N\in\bN$ and $n\in\{N-1,N-2,\dots,1,0\}$ that
	\begin{align}
		&
		\sup_{n\in\{N,N-1,\dots,1,0\}}E^d_{N,\delta,\cM}(n)
		\nonumber\\
		&
		\leq 
		[1+N^{-1}T(1+2L(T+1))]^NE^d_{N,\delta,\cM}(N)
		+\frac{[1+N^{-1}T(1+2L(T+1))]^N-1}{N^{-1}T(1+2L(T+1))}
		\nonumber\\
		& \quad
		\cdot 2N^{-1}L(T+1)T^{-2}
		\widetilde{C}\left(d^p+\bE\big[\big\|\xi^{d,0}\big\|^2\big]\right)e^{d,N,\delta,\cM}
		\nonumber\\
		&
		\leq e^{[1+2L(T+1)]T}E^d_{N,\delta,\cM}(N)
		+\big(e^{[1+2L(T+1)]T}-1\big)2T^{-3}
		\widetilde{C}\left(d^p+\bE\big[\big\|\xi^{d,0}\big\|^2\big]\right)e^{d,N,\delta,\cM}.
		\label{error sup 1}
	\end{align}
	Furthermore, by \eqref{assumption Lip f g}, \eqref{error Euler}, 
	and \eqref{def error E} we obtain for all $N\in\bN$ that
	\begin{align*}
		E^d_{N,\delta,\cM}(N)
		&
		=\bE\left[
		\big|
		g^d(\cX^{d,0,N,\delta,\cM}_T)-g^d(X^{d,0}_T)
		\big|^2
		\right]
		\leq LT^{-1}\bE\left[\big\|
		\cX^{d,0,N,\delta,\cM}_T-X^{d,0}_T
		\big\|^2\right]
		\\
		&
		\leq LT^{-1}
		\widetilde{C}\left(d^p+\bE\big[\big\|\xi^{d,0}\big\|^2\big]\right)e^{d,N,\delta,\cM}.
	\end{align*}
	This together with \eqref{error sup 1} imply for all
	$N\in\bN$ that
	$$
	\sup_{n\in\{N,N-1,\dots,1,0\}}E^d_{N,\delta,\cM}(n)
	\leq e^{[1+2L(T+1)]T}\widetilde{C}(2T^{-3}+LT^{-1})
	\left(d^p+\bE\big[\big\|\xi^{d,0}\big\|^2\big]\right)e^{d,N,\delta,\cM}.
	$$
	Combining this with \eqref{est sup expectation V u} yields 
	for all $N\in\bN$ that
	\begin{equation*}
		\sup_{n\in\{N,N-1,\dots,1,0\}}\bE\left[\big|u^d(t_n,X^{d,0}_{t_n})
		-\cV^{d}_n(\cX^{d,0,N,\delta,\cM}_{t_n})\big|^2\right]
		\leq \widehat{C} d^p\left(d^p+\bE\big[\big\|\xi^{d,0}\big\|^2\right)e^{d,N,\delta,\cM},
	\end{equation*}
	where
	\begin{equation}
		\label{EqDefCHat}
		\widehat{C}:=2\big( 4c_3^2(L^{-1/2}+T)^2T(1+T^{1/2})^2e^{6(L^{1/2}+L)T}
		+\widetilde{C}(2T^{-3}+LT^{-1})e^{[1+2L(T+1)]T}\big)
	\end{equation}
	with
	\begin{equation}
		\label{EqDefCTilde}
		\begin{aligned}
			\widetilde{C} & := \max\Big( 27 T^2 \big( 38 L_1 + 37 L_2 + 150 \cdot 12 L (1+6LT) e^{(1+6L)T} (T+1) L \big) e^{(1 + 225L)T}, \\
			& \quad\quad\quad\quad\quad 24\big( 9 \max\{150LT, 1\} + 1 \big) C_\eta T e^{9(1 + 150L)T} e^{3(T+12)L}, \\
			& \quad\quad\quad\quad\quad 16 L T^2 \cdot 5 \max\{1, 4LT (T + 8)\} e^{8L(16+T)} \Big) > 0,
		\end{aligned}
	\end{equation}
	see \cite[Lemma~3.6-3.8]{NW2022} and their proofs for the explicit derivation of $\widetilde{C} > 0$.
\end{proof}

\subsection{Integrability of time-discretized approximation for the value function}

In this section, we show that the Borel functions $\{\mathcal{V}^d_n\}_{n=0}^N$ introduced in \eqref{def cV} are $L^2$-integrable. To this end, we fix some $d,N,\mathcal{M} \in \mathbb{N}$, and $\delta \in (0,1)$ throughout this section,
and consider the time discretization $( \cX^{d,0,N,\delta,\mathcal{M}}_s )_{s \in [0,T]}$ introduced in \eqref{def time discretization}. In order to ease notation, we drop here some indices and denote by $\cX^{d,0}_s  
:= \cX^{d,0,N,\delta,\mathcal{M}}_s$ for all $s\in[0,T]$.

Then, for every $n \in \lbrace 0,1,\dots,N-1 \rbrace$, we define the push-forward measure of $\cX^{d,0}_{t_n}$ under $\mathbb{P}$ as the measure $\mathcal{B}(\mathbb{R}^d) \ni E \mapsto \big(\mathbb{P} \circ (\cX^{d,0}_{t_n})^{-1}\big)(E) := \mathbb{P}\big[ \cX^{d,0}_{t_n} \in E \big] \in [0,1]$, and aim to show that
\begin{equation*}
	\Vert \mathcal{V}^d_n \Vert_{L_2(\mathbb{R}^d,\mathcal{B}(\mathbb{R}^d),\mathbb{P} \circ (\cX^{d,0}_{t_n})^{-1})}^2 = \int_{\mathbb{R}^d} \left\vert \mathcal{V}^d_n(x) \right\vert^2 \left(\mathbb{P} \circ (\cX^{d,0}_{t_n})^{-1} \right)(dx) = \mathbb{E}\left[ \left\vert \mathcal{V}^d_n\left( \cX^{d,0}_{t_n} \right) \right\vert^2 \right] < \infty.
\end{equation*}
For this purpose, we first show the following auxiliary results.

\begin{lemma}
	\label{LemmaFinMom}
	Let $c \in \mathbb{N}_0$ and $x\in\bR^d$, 
	and let $X,Y: \Omega \rightarrow \mathbb{R}^d$ be $d$-dimensional random variables. Moreover, let $\mathbb{R}^d \ni \zeta \mapsto \phi_{Y|X=x}(\zeta) := \mathbb{E}\big[ e^{\mathbf{i} \zeta^\top Y} \Big\vert X = x \big] \in \mathbb{C}$ be the characteristic function of $Y$ conditioned on $X = x$. If $\phi_{Y|X=x}: \mathbb{R}^d \rightarrow \mathbb{C}$ is $(2c)$-times (not necessarily continuously) differentiable at $0 \in \mathbb{R}^d$, then
	\begin{equation}
		\label{EqLemmaFinMom1}
		\mathbb{E}\left[ \Vert Y \Vert^{2c} \big\vert X = x \right] \leq d^{2c-1} \sum_{i=1}^d \left\vert \frac{\partial^{2c} \phi_{Y|X=x}}{\partial \zeta_i^{2c}}(0) \right\vert < \infty.
	\end{equation}
\end{lemma}
\begin{proof}
	We generalize the proof of \cite[Theorem~15.34]{Klenke2014} to this multivariate setting with conditional expectation. Fix some $c \in \mathbb{N}_0$, $d \in \mathbb{N}$, $x \in \mathbb{R}^d$, $i = 1,\dots,d$, and let $X: \Omega \rightarrow \mathbb{R}^d$ and $Y := (Y_1,\dots,Y_d)^\top: \Omega \rightarrow \mathbb{R}^d$ be random variables. Then, by induction on $k = 0,\dots,c$, we first show that
	\begin{equation}
		\label{EqLemmaFinMomProof1}
		\mathbb{E}\left[ \vert Y_i \vert^{2k} \Big\vert X = x \right] = (-1)^k \frac{\partial^{2k} \phi_{Y|X=x}}{\partial \zeta_i^{2k}}(0).
	\end{equation}
	Indeed, the induction initialization in \eqref{EqLemmaFinMomProof1} with $k = 0$ holds trivially. Now, for the induction step, we fix some $k \in \lbrace 1,\dots,c \rbrace$ and assume that \eqref{EqLemmaFinMomProof1} holds for all $j \in \lbrace 0,\dots,k-1 \rbrace$ instead of $k$. Then, by using that $\phi_{Y|X=x}: \mathbb{R}^d \rightarrow \mathbb{C}$ is $(2c)$-times differentiable at $0 \in \mathbb{R}^d$, the function $\mathbb{R} \ni \lambda \mapsto h_i(\lambda) := \re(\phi_{Y|X=x}(\lambda e_i)) = \mathbb{E}\left[ \cos(\lambda Y_i) \big\vert X = x \right] \in \mathbb{R}$ is $(2k)$-times differentiable at $0 \in \mathbb{R}$, where $e_i \in \mathbb{R}^d$ denotes the $i$-th unit vector of $\mathbb{R}^d$. Moreover, since $h_i: \mathbb{R} \rightarrow \mathbb{R}$ is even, we have $h_i^{(2j-1)}(0) = 0$ for all $j \in \lbrace 1,\dots,k \rbrace$. In addition, by using that the $j$-th derivative $h_i^{(j)}: \mathbb{R} \rightarrow \mathbb{R}$ exists on $(-\varepsilon,\varepsilon)$ for all $j \in \lbrace 0,\dots,2k-2 \rbrace$ and some $\varepsilon > 0$, we can apply Taylor's theorem to conclude that for every $\lambda \in (-\varepsilon,\varepsilon)$ it holds that
	\begin{equation*}
		\left\vert h_i(\lambda) - \sum_{j=1}^{k-1} \frac{h_i^{(2j)}(0)}{(2j)!} \lambda^{2j} \right\vert \leq \frac{\vert \lambda \vert^{2k-1}}{(2k-1)!} \sup_{s \in (0,1]} \left\vert h_i^{(2k-1)}(s \lambda) \right\vert.
	\end{equation*}
	From this, we define the continuous function
	\begin{equation*}
		\mathbb{R} \ni \lambda \quad \mapsto \quad f_{i,k}(\lambda) :=
		\begin{cases}
			1, & \text{if } \lambda = 0, \\
			\frac{(-1)^k (2k)!}{\lambda^{2k}} \left( \cos(\lambda) - \sum_{j=0}^{k-1} \frac{(-1)^j}{(2j)!} \lambda^{2j} \right), & \text{otherwise}.
		\end{cases}
	\end{equation*}
	Then, by using the induction hypothesis, i.e.~that $\mathbb{E}\left[ \vert Y_i \vert^{2j} \big\vert X = x \right] = (-1)^{2j} \frac{\partial^{2j} \phi_{Y|X=x}}{\partial \zeta_i^{2j}}(0) = h_i^{(2j)}(0)$ for all $j \in \lbrace 0,\dots,k-1 \rbrace$, it follows that
	\begin{equation*}
		\mathbb{E}\left[ f_{i,k}(\lambda Y_i) \vert Y_i \vert^{2k} \Big\vert X = x \right] \leq \frac{2k}{\vert \lambda \vert} \sup_{s \in (0,1]} \left\vert h_i^{(2k-1)}(s \lambda) \right\vert \leq 2k \sup_{s \in (0,1]} \frac{\left\vert h_i^{(2k-1)}(s \lambda) \right\vert}{s \vert \lambda \vert}.
	\end{equation*}
	Hence, by applying Fatou's lemma, it follows that
	\begin{equation}
		\label{EqLemmaFinMomProof2}
		\begin{aligned}
			\mathbb{E}\left[ \vert Y_i \vert^{2k} \Big\vert X = x \right] & = \mathbb{E}\left[ f_{i,k}(0) \vert Y_i \vert^{2k} \Big\vert X = x \right] \leq \liminf_{\lambda \rightarrow 0} \mathbb{E}\left[ f_{i,k}(\lambda Y_i) \vert Y_i \vert^{2k} \Big\vert X = x \right] \\
			& \leq 2k \liminf_{\lambda \rightarrow 0} \sup_{s \in (0,1]} \frac{\left\vert h_i^{(2k-1)}(s \lambda) \right\vert}{s \vert \lambda \vert} = 2k \left\vert h_i^{(2k)}(0) \right\vert < \infty.
		\end{aligned}
	\end{equation}
	Thus, we can take limits of difference quotients and apply the dominated convergence theorem (using the bound in \eqref{EqLemmaFinMomProof2}) to conclude iteratively that
	\begin{equation*}
		\begin{aligned}
			\mathbb{E}\left[ \vert Y_i \vert^{2k} \Big\vert X = x \right] & = \frac{1}{\mathbf{i}^{2k}} \mathbb{E}\left[ \frac{\partial^{2k}}{\partial \zeta_i^{2k}}\bigg\vert_{\zeta = 0} \left( e^{\mathbf{i} \zeta^\top Y} \right) \bigg\vert X = x \right] = (-\mathbf{i})^{2k} \frac{\partial^{2k}}{\partial \zeta_i^{2k}}\bigg\vert_{\zeta = 0} \mathbb{E}\left[ \left( e^{\mathbf{i} \zeta^\top Y} \right) \bigg\vert X = x \right] \\
			& = (-1)^k \frac{\partial^{2k} \phi_{Y|X=x}}{\partial \zeta_i^{2k}}(0),
		\end{aligned}
	\end{equation*}
	which shows the induction step and that \eqref{EqLemmaFinMomProof1} is true for all $k \in \lbrace 0,\dots,c \rbrace$. Finally, by using that norms are equivalent on $\mathbb{R}^d$, that $\big( \sum_{i=1}^d x_i \big)^{2c} \leq d^{2c-1} \sum_{i=1}^d x_i^{2c}$ for any $x_1,\dots,x_d \geq 0$, the identity \eqref{EqLemmaFinMomProof1}, and that $\mathbb{E}\left[ \vert Y_i \vert^{2c} \big\vert X = x \right] = (-1)^k \frac{\partial^{2k} \phi_{Y|X=x}}{\partial \zeta_i^{2k}}(0)$ is non-negative, it follows that
	\begin{equation*}
		\begin{aligned}
			\mathbb{E}\left[ \left\Vert Y \right\Vert^{2c} \Big\vert X = x \right] & \leq \mathbb{E}\left[ \bigg( \sum_{i=1}^d \vert Y_i \vert \bigg)^{2c} \Bigg\vert X = x \right] \leq d^{2c-1} \sum_{i=1}^d \mathbb{E}\left[ \vert Y_i \vert^{2c} \Big\vert X = x \right] \\
			& = d^{2c-1} \sum_{i=1}^d (-1)^k \frac{\partial^{2k} \phi_{Y|X=x}}{\partial \zeta_i^{2k}}(0) = d^{2c-1} \sum_{i=1}^d \left\vert \frac{\partial^{2k} \phi_{Y|X=x}}{\partial \zeta_i^{2k}}(0) \right\vert < \infty,
		\end{aligned}
	\end{equation*}
	which completes the proof.
\end{proof}

\begin{lemma}
	\label{LemmaYMom}
	Let Assumption~\ref{AssEtaMoments} hold with some $c \in \mathbb{N}_0$. Then, for every $n \in \lbrace 0,1,\dots,N \rbrace$ it holds that $\mathbb{E}\left[ \big\Vert \cX^{d,0}_{t_n} \big\Vert^{2c} \right] < \infty$.
\end{lemma}
\begin{proof}
	For some fixed $d \in \mathbb{N}$, we first observe that $\Delta \cX^{d,0}_{t_{n+1}} := \cX^{d,0}_{t_{n+1}} - \cX^{d,0}_{t_n}$ conditioned on $\cX^{d,0}_{t_n} = x$ has the same distribution as $Z_1 + Z_2$, where $Z_1 \sim \mathcal{N}_d((t_{n+1}-t_n) \mu^d(t_n,x), (t_{n+1}-t_n) a^d(t_n,x))$ is a multivariate normal random vector with $a^d(t_n,x) := \sigma^d(t_n,x) \sigma^d(t_n,x)^\top$, which is independent of the compensated Poisson integral $Z_2 := \int_{A^d_\delta} \eta^d_{t_n}(x,z) \widetilde{N}(t_{n+1}-t_n,dz)$ (see \cite[p.~94]{Applebaum2009}). 
	
	Now, we use the independence and the characteristic functions of $Z_1$ and $Z_2$ (see \cite[Equation~2.9]{Applebaum2009} for $Z_2$) to conclude that the characteristic function of $\Delta \cX^{d,0}_{t_{n+1}} \big\vert \cX^{d,0}_{t_n} = x$ is for every $\zeta \in \mathbb{R}^d$ given by 
	\begin{equation*}
		\begin{aligned}
			\chi^d_{n,x}(\zeta) & := \mathbb{E}\left[ e^{\mathbf{i} \zeta^\top \Delta \cX^{d,0}_{t_{n+1}}} \bigg\vert \cX^{d,0}_{t_n} = x \right] = \mathbb{E}\left[ e^{\mathbf{i} \zeta^\top (Z_1 + Z_2)} \right] = \mathbb{E}\left[ e^{\mathbf{i} \zeta^\top Z_1} \right] \mathbb{E}\left[ e^{\mathbf{i} \zeta^\top Z_2} \right] \\
			& = e^{(t_{n+1}-t_n) \mathbf{i} \zeta^\top \mu^d(t_n,x) - \frac{1}{2} (t_{n+1}-t_n) \zeta^\top a^{d}(t_n,x) \zeta} e^{(t_{n+1}-t_n) \int_{A^d_\delta} \left( \exp\left( \mathbf{i} \zeta^\top \eta^d_{t_n}(x,z) \right) - 1 - \mathbf{i} \zeta^\top \eta^d_{t_n}(x,z) \right) \nu(dz)} \\
			& = e^{(t_{n+1}-t_n) \left( \mathbf{i} \zeta^\top \mu^d(t_n,x) - \frac{1}{2} \zeta^\top a^{d}(t_n,x) \zeta + \int_{A^d_\delta} \theta^d_{t_n}(\zeta,x,z) \nu^d(dz) \right)},
		\end{aligned}
	\end{equation*}
	where $\theta^d_{t_n}(\zeta,x,z) := \exp\left( \mathbf{i} \zeta^\top \eta^d_{t_n}(x,z) \right) - 1 - \mathbf{i} \zeta^\top \eta^d_{t_n}(x,z)$. Next, we prove that the characteristic function $\chi^d_{n,x}: \mathbb{R}^d \rightarrow \mathbb{C}$ is $(2c)$-times differentiable at $0 \in \mathbb{R}^d$. Indeed, for every $\alpha \in \mathbb{N}^d_{0,2c}$, it holds that
	\begin{equation*}
		\frac{\partial^{\vert \alpha \vert}}{\partial \zeta^\alpha}\bigg\vert_{\zeta = 0} \theta^d_{t_n}(\zeta,x,z) =
		\begin{cases}
			0, & \text{if } \vert \alpha \vert \leq 1, \\
			\mathbf{i}^{\vert \alpha \vert} \prod_{i=1}^d \eta^d_{t_n}(x,z)_i^{\alpha_i} & \text{otherwise.}
		\end{cases}
	\end{equation*}
	Hence, by induction on $\vert\alpha\vert = 0,\dots,2c$, we can take limits of difference quotients and apply the dominated convergence theorem (using Assumption~\ref{AssEtaMoments}) to conclude for every $\alpha \in \mathbb{N}^d_{0,2c}$ that
	\begin{equation*}
		\frac{\partial^{\vert \alpha \vert}}{\partial \zeta^\alpha}\bigg\vert_{\zeta = 0} \int_{\mathbb{R}^d} \theta^d_{t_n}(\zeta,x,z) \nu^d(dz) = 
		\begin{cases}
			0, & \text{if } \vert \alpha \vert \leq 1, \\
			\mathbf{i}^{\vert \alpha \vert} \int_{A^d_\delta} \prod_{i=1}^d \eta^d_{t_n}(x,z)_i^{\alpha_i} \nu^d(dz) & \text{otherwise.}
		\end{cases}
	\end{equation*}
	This shows that the characteristic function $\chi^d_{n,x}: \mathbb{R}^d \rightarrow \mathbb{C}$ is $c$-times differentiable at $0 \in \mathbb{R}^d$ as concatenation of functions that are $c$-times differentiable at $0 \in \mathbb{R}^d$. Thus, by using Lemma~\ref{LemmaFinMom}, i.e.~that $\mathbb{E}\big[ \big\Vert \cX^{d,0}_{t_{n+1}} \big\Vert^{2c} \Big\vert \cX^{d,0}_{t_n} = x \big] < \infty$ for all $x \in \mathbb{R}^d$, and the tower property, it follows that
	\begin{equation*}
		\mathbb{E}\left[ \big\Vert \cX^{d,0}_{t_{n+1}} \big\Vert^{2c} \right] = \mathbb{E}\bigg[ \mathbb{E}\left[ \big\Vert \cX^{d,0}_{t_{n+1}} \big\Vert^{2c} \Big\vert \cX^{d,0}_{t_n} \right] \bigg] < \infty,
	\end{equation*}
	which completes the proof.
\end{proof}

\begin{lemma}
	\label{LemmaL2Y}
	Let $n \in \lbrace 0,1,\dots,N \rbrace$, $x \in \mathbb{R}^d$, and $\gamma \in (0,\infty)$. Moreover, let Assumption~\ref{AssEtaMoments} hold with some $c \in \mathbb{N}_0 \cap [\gamma,\infty)$. Then, the following holds true:
	\begin{enumerate}
		\item\label{LemmaL2Y1} The identity $(C_b(\mathbb{R}^d),\Vert \cdot \Vert_{C_{pol,\gamma}(\mathbb{R}^d)}) \hookrightarrow (L_2(\mathbb{R}^d,\mathcal{B}(\mathbb{R}^d),\mathbb{P} \circ (\cX^{d,0}_{t_n})^{-1}),\Vert \cdot \Vert_{L_2(\mathbb{R}^d,\mathcal{B}(\mathbb{R}^d),\mathbb{P} \circ (\cX^{d,0}_{t_n})^{-1})})$ is a continuous dense embedding.
		\item\label{LemmaL2Y2} If Assumption~\ref{assumption Lip and growth} additionally holds true, then $\mathcal{V}^d_n \in L_2(\mathbb{R}^d,\mathcal{B}(\mathbb{R}^d),\mathbb{P} \circ (\cX^{d,0}_{t_n})^{-1})$.
	\end{enumerate}
\end{lemma}
\begin{proof}
	For \ref{LemmaL2Y1}, we use that $(x+y)^{2\lceil\gamma\rceil} \leq 2^{2\lceil\gamma\rceil-1} \left( x^{2\lceil\gamma\rceil} + y^{2\lceil\gamma\rceil} \right)$ for any $x,y \geq 0$, and Lemma~\ref{LemmaYMom} to conclude that
	\begin{equation*}
		\begin{aligned}
			\int_{\mathbb{R}^d} (1+\Vert x \Vert)^{2\gamma} \left( \mathbb{P} \circ (\cX^{d,0}_{t_n})^{-1} \right)(dx) & \leq \mathbb{E}\left[ \left( 1 + \left\Vert \cX^{d,0}_{t_n} \right\Vert \right)^{2\lceil\gamma\rceil} \right] \\
			& \leq 2^{2\lceil\gamma\rceil-1} \left( 1 + \mathbb{E}\left[ \left\Vert \cX^{d,0}_{t_n} \right\Vert^{2\lceil\gamma\rceil} \right] \right) < \infty.
		\end{aligned}
	\end{equation*}
	Hence, by following \cite[Example~2.6~(c)]{NeufeldSchmocker2024}, the identity $(C_b(\mathbb{R}^d),\Vert \cdot \Vert_{C_{pol,\gamma}(\mathbb{R}^d)}) \hookrightarrow (L_2(\mathbb{R}^d,\mathcal{B}(\mathbb{R}^d),\mathbb{P} \circ (\cX^{d,0}_{t_n})^{-1}),\Vert \cdot \Vert_{L_2(\mathbb{R}^d,\mathcal{B}(\mathbb{R}^d),\mathbb{P} \circ (\cX^{d,0}_{t_n})^{-1})})$ is a continuous dense embedding.
	
	In order to show \ref{LemmaL2Y2}, we first observe that $\mathbb{R}^d \ni x \mapsto \mathcal{V}^d_n(x) \in \mathbb{R}$ is by definition $\mathcal{L}(\mathbb{R}^d)/\mathcal{B}(\mathbb{R})$-measurable, for all $n \in \lbrace 0,1,\dots,N \rbrace$. Furthermore, by \eqref{assumption Lip f g} and \eqref{assumption growth f g},
	we have for all $d\in\bN$, $t\in[0,T]$, $x\in\bR^d$, and $v\in\bR$ that
	\begin{equation}
	\label{f d LG}
	|f^d(t,x,v)|^2 
	\leq 2|f^d(t,x,0)|^2 + 2L|v|^2
	\leq 2L(d^pT^{-2}+1)(1+\|x\|^2+|v|^2).
	\end{equation} 	
	
	 Now, we show by induction on $n = N,N-1,\dots,1,0$ that $\mathcal{V}^d_n \in L_2(\mathbb{R}^d,\mathcal{B}(\mathbb{R}^d),\mathbb{P} \circ (\cX^{d,0}_{t_n})^{-1})$ for all $n \in \lbrace 0,1,\dots,N \rbrace$. For the induction initialization $n = N$, we use \eqref{assumption growth f g}, the inequality $x^2 \leq 1 + x^{2\lceil\gamma\rceil}$ for any $x \geq 0$, and Lemma~\ref{LemmaYMom} to obtain that
	\begin{equation*}
		\begin{aligned}
			\Vert \mathcal{V}^d_N \Vert_{L_2(\mathbb{R}^d,\mathcal{B}(\mathbb{R}^d),\mathbb{P} \circ (\cX^{d,0}_{t_N})^{-1})}^2 & = \int_{\mathbb{R}^d} \left\vert \mathcal{V}^d_N(x) \right\vert^2 \left( \mathbb{P} \circ (\cX^{d,0}_{t_N})^{-1} \right)(dx) = \mathbb{E}\left[ \left\vert g\left( \cX^{d,0}_{t_N} \right) \right\vert^2 \right] \\
			& \leq Ld^p \mathbb{E}\left[ 1 + \left\Vert \cX^{d,0}_{t_N} \right\Vert^2 \right] \leq Ld^p \left( 2 + \mathbb{E}\left[ \left\Vert \cX^{d,0}_{t_N} \right\Vert^{2\lceil\gamma\rceil} \right] \right) < \infty,
		\end{aligned}
	\end{equation*}
	which proves that $\mathcal{V}^d_N \in L_2(\mathbb{R}^d,\mathcal{B}(\mathbb{R}^d),\mathbb{P} \circ (\cX^{d,0}_{t_N})^{-1})$. Now, for the induction step, we fix some $n \in \lbrace 0,1,\dots,N-1 \rbrace$ and assume that $\mathcal{V}^d_{n+1} \in L_2(\mathbb{R}^d,\mathcal{B}(\mathbb{R}^d),\mathbb{P} \circ (\cX^{d,0}_{t_{n+1}})^{-1})$. Then, by using Jensen's inequality, the inequality $x^2 \leq 1 + x^{2\lceil\gamma\rceil}$ for any $x \geq 0$, and \eqref{f d LG}, it follows that
	\begin{equation*}
		\begin{aligned}
			& \Vert \mathcal{V}^d_n \Vert_{L_2(\mathbb{R}^d,\mathcal{B}(\mathbb{R}^d),\mathbb{P} \circ (\cX^{d,0}_{t_n})^{-1})}^2 = \int_{\mathbb{R}^d} \left\vert \mathcal{V}^d_n(x) \right\vert^2 \left( \mathbb{P} \circ (\cX^{d,0}_{t_n})^{-1} \right)(dx) = \mathbb{E}\left[ \left\vert \mathcal{V}^d_n\left(\cX^{d,0}_{t_n}\right) \right\vert^2 \right] \\
			& \quad\quad = \mathbb{E}\left[ \left\vert \mathbb{E}\left[ \mathcal{V}^d_{n+1}(\cX^{d,0}_{t_{n+1}}) + (t_{n+1}-t_n) f^d(t_{n+1},\cX^{d,0}_{t_{n+1}},\mathcal{V}^d_{n+1}(\cX^{d,0}_{t_{n+1}})) \Big\vert \cX^{d,0}_{t_n} \right] \right\vert^2 \right] \\
			& \quad\quad \leq \mathbb{E}\left[ \left\vert \mathcal{V}^d_{n+1}(\cX^{d,0}_{t_{n+1}}) + (t_{n+1}-t_n) f^d(t_{n+1},\cX^{d,0}_{t_{n+1}},\mathcal{V}^d_{n+1}(\cX^{d,0}_{t_{n+1}})) \right\vert^2 \right] \\
			& \quad\quad \leq 2 \mathbb{E}\left[ \left\vert \mathcal{V}^d_{n+1}(\cX^{d,0}_{t_{n+1}}) \right\vert^2 \right] + 2 (t_{n+1}-t_n)^2 \mathbb{E}\left[ \left\vert f^d(t_{n+1},\cX^{d,0}_{t_{n+1}},\mathcal{V}^d_{n+1}(\cX^{d,0}_{t_{n+1}})) \right\vert^2 \right] \\
			& \quad\quad \leq 2 \mathbb{E}\left[ \left\vert \mathcal{V}^d_{n+1}(\cX^{d,0}_{t_{n+1}}) \right\vert^2 \right] + 4L(d^pT^{-2}+1) (t_{n+1}-t_n)^2 \mathbb{E}\left[ 1 + \left\Vert \cX^{d,0}_{t_{n+1}} \right\Vert^2 + \left\vert \mathcal{V}^d_{n+1}(\cX^{d,0}_{t_{n+1}}) \right\vert^2 \right] \\
			& \quad\quad \leq 2 \left( 1 + 2L(d^pT^{-2}+1) (t_{n+1}-t_n)^2 \right) \left( 1 + \mathbb{E}\left[ \left\Vert \cX^{d,0}_{t_{n+1}} \right\Vert^2 \right] + \left\Vert \mathcal{V}^d_{n+1} \right\Vert_{L_2(\mathbb{R}^d,\mathcal{B}(\mathbb{R}^d),\mathbb{P} \circ (\cX^{d,0}_{t_{n+1}})^{-1})}^2 \right) < \infty,
		\end{aligned}
	\end{equation*}
	which shows the induction step, i.e.~that $\mathcal{V}^d_n \in L_2(\mathbb{R}^d,\mathcal{B}(\mathbb{R}^d),\mathbb{P} \circ (\cX^{d,0}_{t_n})^{-1})$.
\end{proof}

\subsection{\textbf{Proof of Theorem \ref{ThmMain}}, Theorem \ref{ThmMainRN}, and Corollary \ref{CorMainRN}}
\label{section proof main result}

In this section, we prove the main results in Theorem~\ref{ThmMain}, Theorem~\ref{ThmMainRN}, and Corollary~\ref{CorMainRN}. To this end, we consider again the time discretization $(\cX^{d,j}_s)_{s \in [0,T]} := ( \cX^{d,j,N,\delta,\mathcal{M}}_s )_{s \in [0,T]}$ introduced in \eqref{def time discretization}, where we drop again some indices here in order to ease the notation. Moreover, we recall that the push-forward measure of $\cX^{d,0}_{t_n}$ under $\mathbb{P}$ is defined as $\mathcal{B}(\mathbb{R}^d) \ni E \mapsto \big(\mathbb{P} \circ (\cX^{d,0}_{t_n})^{-1}\big)(E) := \mathbb{P}\big[ \cX^{d,0}_{t_n} \in E \big] \in [0,1]$.

\begin{proof}[Proof of Theorem~\ref{ThmMain}]
	Fix $d,N,\cM \in \bN$, $\delta,\varepsilon \in (0,1)$ with $\cM\geq \delta^{-2} C_\eta d^p$, and $n\in\{N-1,N-2,\dots,1,0\}$. Then, \ref{ThmMain1} follows directly from Proposition~\ref{PropTimeDiscr}.
	
	For \ref{ThmMain2}, we first use Lemma~\ref{LemmaL2Y}~\ref{LemmaL2Y1} to conclude that the Banach space $(L_2(\mathbb{R}^d,\mathcal{B}(\mathbb{R}^d),\mathbb{P} \circ (\cX^{d,0}_{t_n})^{-1}),\Vert \cdot \Vert_{L_2(\mathbb{R}^d,\mathcal{B}(\mathbb{R}^d),\mathbb{P} \circ (\cX^{d,0}_{t_n})^{-1})})$ satisfies the conditions in \cite[Definition~2.3]{NeufeldSchmocker2024}. Hence, by applying the universal approximation result for deterministic neural networks in \cite[Theorem~2.8]{NeufeldSchmocker2024} to $\mathcal{V}^d_n \in L_2(\mathbb{R}^d,\mathcal{B}(\mathbb{R}^d),\mathbb{P} \circ (\cX^{d,0}_{t_n})^{-1})$, there exists a deterministic neural network 
	$\mathbb{V}^d_n(\theta_n,\cdot) \in L_2(\mathbb{R}^d,\mathcal{B}(\mathbb{R}^d),\mathbb{P} \circ (\cX^{d,0}_{t_n})^{-1})$ with network parameters $\theta_n \in \mathbb{R}^r$ such that
	\begin{equation*}
		\Vert \mathcal{V}^d_n - \mathbb{V}^d_n(\theta_n,\cdot) \Vert^2_{L_2(\mathbb{R}^d,\mathcal{B}(\mathbb{R}^d),\mathbb{P} \circ (\cX^{d,0}_{t_n})^{-1})} \leq \varepsilon.
	\end{equation*}
	Thus, by using the push-forward measure $\mathbb{P} \circ (\cX^{d,0}_{t_n})^{-1}: \mathcal{B}(\mathbb{R}^d) \rightarrow [0,1]$, it follows that
	\begin{equation*}
		\begin{aligned}
			\err^{d,N,\delta,\mathcal{M}}_{UAT,n} & := \mathbb{E}\left[ \left\vert \mathcal{V}^d_n\left( \mathcal{X}^{d,0}_{t_n} \right) - \mathbb{V}^d_n\left( \theta_n, \mathcal{X}^{d,0}_{t_n} \right) \right\vert^2 \right] \\
			& = \int_{\mathbb{R}^d} \left\vert \mathcal{V}^d_n(x) - \mathbb{V}^d_n(\theta_n,x) \right\vert^2 \left(\mathbb{P} \circ (\cX^{d,0}_{t_n})^{-1} \right)(dx) \\
			& = \Vert \mathcal{V}^d_n - \mathbb{V}^d_n(\theta_n,\cdot) \Vert^2_{L_2(\mathbb{R}^d,\mathcal{B}(\mathbb{R}^d),\mathbb{P} \circ (\cX^{d,0}_{t_n})^{-1})} \leq \varepsilon,
		\end{aligned}
	\end{equation*}
	which proves the inequality in \ref{ThmMain2}.
	
	Finally, by combining the inequalities in \ref{ThmMain1}+\ref{ThmMain2} with the inequality $(x+y+z)^2 \leq 3\left( x^2 + y^2 + z^2 \right)$, we obtain the inequality \eqref{EqThmMain1}.
\end{proof}

\begin{proof}[Proof of Theorem~\ref{ThmMainRN}]
	Fix $d,N,\cM,J \in \bN$, $\delta,\varepsilon \in (0,1)$ with $\cM\geq \delta^{-2} C_\eta d^p$, $\vartheta > 0$, and $n\in\{N-1,N-2,\dots,1,0\}$. Then, \ref{ThmMainRN1} follows directly from Proposition~\ref{PropTimeDiscr}.
	
	For \ref{ThmMainRN2}, we first use Lemma~\ref{LemmaL2Y}~\ref{LemmaL2Y1} to conclude that the Banach space $(L_2(\mathbb{R}^d,\mathcal{B}(\mathbb{R}^d),\mathbb{P} \circ (\cX^{d,0}_{t_n})^{-1}),\Vert \cdot \Vert_{L_2(\mathbb{R}^d,\mathcal{B}(\mathbb{R}^d),\mathbb{P} \circ (\cX^{d,0}_{t_n})^{-1})})$ satisfies \cite[Assumption~3.5]{NeufeldSchmocker2023}. Hence, by applying the universal approximation result for random neural networks in \cite[Corollary~3.8]{NeufeldSchmocker2023} to $\mathcal{V}^d_n \in L_2(\mathbb{R}^d,\mathcal{B}(\mathbb{R}^d),\mathbb{P} \circ (\cX^{d,0}_{t_n})^{-1})$, there exists a random neural network 
	$\mathscr{V}^d_n(Y_n,\cdot) \in L_2(\Omega,\mathcal{F},\mathbb{P};L_2(\mathbb{R}^d,\mathcal{B}(\mathbb{R}^d),\mathbb{P} \circ (\cX^{d,0}_{t_n})^{-1}))$ with linear readout $Y_n: \Omega \rightarrow \mathbb{R}^{K_\varepsilon}$ such that
	\begin{equation}
		\label{EqThmRUATProof1}
		\mathbb{E}\left[ \Vert \mathcal{V}^d_n - \mathscr{V}^d_n(Y_n,\cdot) \Vert^2_{L_2(\mathbb{R}^d,\mathcal{B}(\mathbb{R}^d),\mathbb{P} \circ (\cX^{d,0}_{t_n})^{-1})} \right] \leq \varepsilon.
	\end{equation}
	Thus, by conditioning on $\mathcal{F}_{A,B} := \sigma\left( \lbrace A_k, B_k: k \in \mathbb{N} \rbrace \right)$ and by using that $\cX^{d,0}_{t_n}$ is independent of $\mathcal{F}_{A,B}$ as well as \eqref{EqThmRUATProof1}, it follows that
	\begin{equation*}
		\begin{aligned}
			\err^{d,N,\delta,\mathcal{M}}_{UAT,n} & := \mathbb{E}\left[ \left\vert \mathcal{V}^d_n\left(\cX^{d,0}_{t_n}\right) - \mathscr{V}^d_n\left(Y_n,\cX^{d,0}_{t_n}\right) \right\vert^2 \right] \\
			& = \mathbb{E}\Bigg[ \mathbb{E}\left[ \left\vert \mathcal{V}^d_n\left(\cX^{d,0}_{t_n}\right) - \mathscr{V}^d_n\left(Y_n,\cX^{d,0}_{t_n}\right) \right\vert^2 \bigg\vert \mathcal{F}_{A,B} \right] \Bigg] \\
			& = \mathbb{E}\bigg[ \mathbb{E}\left[ \Vert \mathcal{V}^d_n - \mathscr{V}^d_n(Y_n,\cdot) \Vert^2_{L_2(\mathbb{R}^d,\mathcal{B}(\mathbb{R}^d),\mathbb{P} \circ (\cX^{d,0}_{t_n})^{-1})} \Big\vert \mathcal{F}_{A,B} \right] \bigg] \\
			& = \mathbb{E}\left[ \Vert \mathcal{V}^d_n - \mathscr{V}^d_n(Y_n,\cdot) \Vert^2_{L_2(\mathbb{R}^d,\mathcal{B}(\mathbb{R}^d),\mathbb{P} \circ (\cX^{d,0}_{t_n})^{-1})} \right] \leq \varepsilon,
		\end{aligned}
	\end{equation*}
	which shows the inequality in \ref{ThmMainRN2}.
	
	For \ref{ThmMainRN3}, we aim to apply the bound on the generalization error of the least squares method in \cite[Theorem~11.3]{Gyoerfi2002}. To this end, we define for every fixed $j = 1,\dots,J$ the random variables $\overline{X_j} := \mathcal{X}^{d,j}_{t_n}$ and
	\begin{equation*}
		\overline{Y_j} := T_\vartheta\left( \mathscr{V}^d_{n+1}\left(\Upsilon_{n+1},\mathcal{X}^{d,j}_{t_{n+1}}\right) + (t_{n+1} - t_n) \left(F^d \circ \mathscr{V}^d_{n+1}(\Upsilon_{n+1}, \cdot) \right)\left(t_{n+1},\mathcal{X}^{d,j}_{t_{n+1}}\right) \right).
	\end{equation*}
	Then, we observe that the conditional variance satisfies
	\begin{equation}
		\label{EqThmRUATProof2b}
		\begin{aligned}
			\sup_{x \in \mathbb{R}^d} \operatorname{Var}\left[ \overline{Y_j} \big\vert \overline{X_j} = x \right] & = \sup_{x \in \mathbb{R}^d} \mathbb{E}\left[ \left( \overline{Y_j} - \mathbb{E}\left[ \overline{Y_j} \big\vert \overline{X_j} = x \right] \right)^2 \Big\vert \overline{X_j} = x \right] \\
			& = \sup_{x \in \mathbb{R}^d} \left( \mathbb{E}\left[ \vert \overline{Y_j} \vert^2 \big\vert \overline{X_j} = x \right] - \mathbb{E}\left[ \overline{Y_j} \big\vert \overline{X_j} = x \right]^2 \right) \\
			& \leq \sup_{x \in \mathbb{R}^d} \mathbb{E}\left[ \vert \overline{Y_j} \vert^2 \big\vert \overline{X_j} = x \right] \leq \vartheta^2.
		\end{aligned}
	\end{equation}
	Moreover, we define $\mathcal{G}_{K_\varepsilon}$ as the vector space of random neural networks with $K_\varepsilon$ neurons, i.e.
	\begin{equation*}
		\mathcal{G}_{K_\varepsilon} := \left\lbrace \Omega \ni \omega \mapsto \sum_{k=1}^{K_\varepsilon} \widetilde{Y}_k(\omega) \rho\left( A_k(\omega)^\top \cdot - B_k(\omega) \right) \in \overline{C_b(\mathbb{R}^d)}^\gamma:
		\begin{matrix}
			\widetilde{Y} := (\widetilde{Y}_1,\dots,\widetilde{Y}_{K_\varepsilon})^\top: \Omega \rightarrow \mathbb{R}^{K_\varepsilon}
		\end{matrix}
		\right\rbrace.
	\end{equation*}
	Then, the truncated learning problem of \eqref{EqLeastSquares} corresponds to
	\begin{equation*}
		\begin{aligned}
			\Upsilon_n(\omega) & := \argmin_{y \in \mathbb{R}^{K_\varepsilon}} \frac{1}{J} \sum_{j=1}^J \Bigg\vert \mathscr{V}^d_n\left( y, \mathcal{X}^{d,j}_{t_n}(\omega) \right) - T_\vartheta\bigg( \mathscr{V}^d_{n+1}\left( \Upsilon_{n+1}(\omega), \mathcal{X}^{d,j}_{t_{n+1}}(\omega) \right) \\
			& \quad\quad\quad\quad\quad\quad\quad\quad\quad\quad\quad + (t_{n+1} - t_n) \left( F^d \circ \mathscr{V}^d_{n+1}\left( \Upsilon_{n+1}(\omega), \cdot \right) \right)\left(t_{n+1}, \mathcal{X}^{d,j}_{t_{n+1}}(\omega) \right) \bigg) \Bigg\vert^2 \\
			& = \argmin_{y \in \mathbb{R}^{K_\varepsilon}} \frac{1}{J} \sum_{j=1}^J \left\vert \sum_{k=1}^{K_\varepsilon} y_k \rho\left( A_k(\omega)^\top \overline{X_j}(\omega) - B_k(\omega) \right) - \overline{Y_j}(\omega) \right\vert^2.
		\end{aligned}
	\end{equation*}	
	Hence, by applying \cite[Theorem~11.3]{Gyoerfi2002} to the function $m(\cdot) := T_\vartheta\left( \mathscr{V}^d_n(Y_n,\cdot) \right)$ as well as setting
	\begin{equation}
		\label{EqDefC0}
		C_0 := 8 + 8 \cdot 2304 \ln(3 \cdot 12 e) > 0,
	\end{equation}
	corresponding to the universal constant $c > 0$ in \cite[Theorem~11.3]{Gyoerfi2002} (see \cite[p.~192-194]{Gyoerfi2002}), where $\mathcal{G}_{K_\varepsilon}$ has vector dimension $K_\varepsilon$ in the sense of \cite[p.~184]{Gyoerfi2002}, and by using \eqref{EqThmRUATProof2b} and that $\mathscr{V}^d_n(Y_n,\cdot) \in \mathcal{G}_{K_\varepsilon}$ for any $n\in\{N-1,N-2,\dots,1,0\}$, it follows for every $n\in\{N-1,N-2,\dots,1,0\}$ that
	\begin{equation*}
		\begin{aligned}
			\err^{d,N,\delta,\mathcal{M}}_{gen,n} & := \mathbb{E}\left[ \left\vert T_\vartheta\left( \mathscr{V}^d_n\left( Y_n, \mathcal{X}^{d,0}_{t_n} \right) \right) - T_\vartheta \left( \mathscr{V}^d_n\left( \Upsilon_n, \mathcal{X}^{d,0}_{t_n} \right) \right) \right\vert^2 \right] \\
			& \leq C_0 \vartheta^2 \frac{(\ln(J)+1) K_\varepsilon}{J} + 8 \inf_{g \in \mathcal{G}_{K_\varepsilon}} \mathbb{E}\left[ \left\vert T_\vartheta\left( \mathscr{V}^d_n\left( Y_n, \mathcal{X}^{d,0}_{t_n} \right) \right) - g\left( \mathcal{X}^{d,0}_{t_n} \right) \right\vert^2 \right] \\
			& = C_0 \vartheta^2 \frac{(\ln(J)+1) K_\varepsilon}{J} + 8 \mathbb{E}\left[ \mathds{1}_{\lbrace \vert \mathscr{V}^d_n(Y_n,\cX^{d,0}_{t_n}) \vert > \vartheta \rbrace} \left\vert \mathscr{V}^d_n\left(Y_n,\cX^{d,0}_{t_n}\right) \right\vert^2 \right],
		\end{aligned}
	\end{equation*}
	which shows the inequality in \ref{ThmMainRN3}.
	
	For \eqref{EqThmMainRN1}, we introduce the constant
	\begin{equation}
		\label{EqDefCBar}
		\overline{C} := 34 \cdot 2^q L^{q/2} \exp\big\{q L^{1/2} T\big\} \exp\left\{[2(L+1)]^{\frac{q-2}{2}}2(L+L^{1/2})q(q-1)T\right\} > 0,
	\end{equation}
	which only depends on $q$, $T$, and $L$. Then, by using H\"older's inequality, Markov's inequality, that~\eqref{u est}, and Corollary~\ref{corollary Lyaponov q}, we obtain for every $n\in\{N-1,N-2,\dots,1,0\}$ that
	\begin{equation}
		\label{EqThmRUATProof2}
		\begin{aligned}
			& \bE\left[\mathds{1}
			_{\lbrace \vert u^d(t_n,X^{d,0}_{t_n}) \vert > \vartheta \rbrace}
			\left\vert u^d\left(t_n,X^{d,0}_{t_n}\right) \right\vert^2 \right] \leq \bP\left[\left\vert u^d(t_n,X^{d,0}_{t_n}) \right\vert > \vartheta\right]^\frac{q-2}{q} 			\bE\left[\left\vert u^d\left(t_n,X^{d,0}_{t_n}\right)\right\vert^q\right]^\frac{2}{q} \\
			& \quad\quad \leq \left( \frac{1}{\vartheta^q} \bE\left[\left\vert u^d\left(t_n,X^{d,0}_{t_n}\right)\right\vert^q\right] \right)^\frac{q-2}{q} \bE\left[\left\vert u^d\left(t_n,X^{d,0}_{t_n}\right)\right\vert^q\right]^\frac{2}{q} \leq \frac{1}{\vartheta^{q-2}} \bE\left[\left\vert u^d\left(t_n,X^{d,0}_{t_n}\right)\right\vert^q\right] \\
			& \quad\quad \leq \frac{2^q L^{q/2} \exp\big\{q L^{1/2} T\big\}}{\vartheta^{q-2}} \bE\left[\big(d^p+\big\|X^{d,0}_{t_n}\big\|^2\big)^{q/2}\right] \\
			& \quad\quad \leq \frac{2^q L^{q/2} \exp\big\{q L^{1/2} T\big\}}{\vartheta^{q-2}} \exp\left\{[2(L+1)]^{\frac{q-2}{2}}2(L+L^{1/2})q(q-1)T\right\} \bE\left[\big(d^p+\big\|\xi^{d,0}\big\|^2\big)^{q/2}\right] \\
			& \quad\quad \leq \frac{\overline{C}}{34 \vartheta^{q-2}} \bE\left[\big(d^p+\big\|\xi^{d,0}\big\|^2\big)^{q/2}\right].
		\end{aligned}
	\end{equation}
	Finally, by inserting three times the inequality $(x+y)^2 \leq 2\left( x^2 + y^2 \right)$ for any $x,y \geq 0$, by applying \cite[Theorem~11.3]{Gyoerfi2002} now to $m(\cdot) := T_\vartheta\left( u^d_n(t_n,\cdot) \right)$ (with universal constant $C_0 > 0$ defined in \eqref{EqDefC0}, where $\mathcal{G}_{K_\varepsilon}$ has vector dimension $K_\varepsilon$ in the sense of \cite[p.~184]{Gyoerfi2002}), by using that $\mathscr{V}^d_n(Y_n,\cdot) \in \mathcal{G}_{K_\varepsilon}$, the inequality \eqref{EqThmRUATProof2}, and the inequalities in \ref{ThmMainRN1}+\ref{ThmMain2}, we conclude for every $n\in\{N-1,N-2,\dots,1,0\}$ that
	\begin{equation*}
		\begin{aligned}
			& \mathbb{E}\left[ \left\vert u^d\left( t_n, X^{d,0}_{t_n} \right) - \mathscr{V}^d_n\left( \Upsilon_n, \mathcal{X}^{d,0}_{t_n} \right) \right\vert^2 \right] \\
			& \quad\quad \leq 2 \mathbb{E}\left[ \left\vert u^d\left( t_n, X^{d,0}_{t_n} \right) - T_\vartheta\left( u^d\left( t_n, X^{d,0}_{t_n} \right) \right) \right\vert^2 \right] + 2 \mathbb{E}\left[ \left\vert T_\vartheta\left( u^d\left( t_n, X^{d,0}_{t_n} \right) \right) - \mathscr{V}^d_n\left( \Upsilon_n, \mathcal{X}^{d,0}_{t_n} \right) \right\vert^2 \right] \\
			& \quad\quad \leq 2 \mathbb{E}\left[ \left\vert u^d\left( t_n, X^{d,0}_{t_n} \right) - T_\vartheta\left( u^d\left( t_n, X^{d,0}_{t_n} \right) \right) \right\vert^2 \right] + 2 C_0 \vartheta^2 \frac{(\ln(J)+1) K_\varepsilon}{J} \\
			& \quad\quad\quad\quad + 16 \inf_{g \in \mathcal{G}_{K_\varepsilon}} \mathbb{E}\left[ \left\vert T_\vartheta\left( u^d\left( t_n, X^{d,0}_{t_n} \right) \right) - g\left( \mathcal{X}^{d,0}_{t_n} \right) \right\vert^2 \right] \\
			& \quad\quad \leq 34 \mathbb{E}\left[ \left\vert u^d\left( t_n, X^{d,0}_{t_n} \right) - T_\vartheta\left( u^d\left( t_n, X^{d,0}_{t_n} \right) \right) \right\vert^2 \right] + 2 C_0 \vartheta^2 \frac{(\ln(J)+1) K_\varepsilon}{J} \\
			& \quad\quad\quad\quad + 32 \inf_{g \in \mathcal{G}_{K_\varepsilon}} \mathbb{E}\left[ \left\vert u^d\left( t_n, X^{d,0}_{t_n} \right) - g\left( \mathcal{X}^{d,0}_{t_n} \right) \right\vert^2 \right] \\
			& \quad\quad \leq 34 \mathbb{E}\left[ \mathds{1}_{\lbrace \vert u^d(t_n,X^{d,0}_{t_n}) \vert > \vartheta \rbrace} \left\vert u^d\left(t_n,X^{d,0}_{t_n}\right) \right\vert^2 \right] + 2 C_0 \vartheta^2 \frac{(\ln(J)+1) K_\varepsilon}{J} \\
			& \quad\quad\quad\quad + 32 \mathbb{E}\left[ \left\vert u^d\left( t_n, X^{d,0}_{t_n} \right) - \mathscr{V}^d_n\left(Y_n,\cX^{d,0}_{t_n}\right) \right\vert^2 \right] \\
			& \quad\quad \leq 34 \frac{\overline{C}}{34 \vartheta^{q-2}} \bE\left[\big(d^p+\big\|\xi^{d,0}\big\|^2\big)^{q/2}\right] + 2 C_0 \vartheta^2 \frac{(\ln(J)+1) K_\varepsilon}{J} \\
			& \quad\quad\quad\quad + 64 \mathbb{E}\left[ \left\vert u^d\left( t_n, X^{d,0}_{t_n} \right) - \mathcal{V}^d_n\left( \mathcal{X}^{d,0}_{t_n} \right) \right\vert^2 \right] + 64 \mathbb{E}\left[ \left\vert \mathcal{V}^d_n\left( \mathcal{X}^{d,0}_{t_n} \right) - \mathscr{V}^d_n\left(Y_n,\cX^{d,0}_{t_n}\right) \right\vert^2 \right] \\
			& \quad\quad \leq \frac{\overline{C}}{\vartheta^{q-2}} \bE\left[\big(d^p+\big\|\xi^{d,0}\big\|^2\big)^{q/2}\right] + 2 C_0 \vartheta^2 \frac{(\ln(J)+1) K_\varepsilon}{J} \\
			& \quad\quad\quad\quad + 64 \widehat{C} d^p \left( d^p+\bE\left[ \big\|\xi^{d,0}\big\|^2 \right] \right) e^{d,N,\delta,\cM} + 64 \varepsilon,
		\end{aligned}
	\end{equation*}
	which shows the inequality in \eqref{EqThmMainRN1} and completes the proof.
\end{proof}

\begin{proof}[Proof of Corollary~\ref{CorMainRN}]
	Corollary~\ref{CorMainRN} follows directly from inserting the bounds in Line~2-8 of Algorithm~\ref{AlgDS} into inequality \eqref{EqThmMainRN1} of Theorem~\ref{ThmMainRN}.
\end{proof}

\appendix

\section{Notation}
\label{App}
For the convenience of the readers, we provide a list below to present some notations of important
quantities used throughout this paper.
\begin{center}
\begin{tabular}{ |p{2.5cm}|p{11.5cm}|  }
 \hline
 \multicolumn{2}{|c|}{Notation List} \\
 \hline
 $X^{d,t,x}$ & Solution of SDE with jumps starting at $(t,x)$, see \eqref{SDE}\\
 $X^{d,m}$ & Solution of SDE with jumps starting at $(0,\xi^{d,m})$, see \eqref{SDE d} \\
 $\cX^{d,m,N,\delta,\cM}$ & Euler-Maruyama approximation of $X^{d,m}$, see \eqref{def time discretization}\\
 $\mu^d$, $\sigma^d$, $\eta^d$, $\nu^d$ & Drift coefficient, diffusion coefficient, jump coefficient, and L\'evy measure \\
 \hline
 $u^d$ & Unique viscosity solution of PIDE \eqref{PDE}\\
 $g^d$ & Terminal condition of PIDE \eqref{PDE}\\
 $\cV^d$ & Time-discretized approximation for $u^d$, see \eqref{def cV argmin}+\eqref{def cV}\\
 $\bV^d$ & Deterministic neural network approximation for $u^d$, 
 see Section \ref{section Deter NN}\\
 $\mathscr{V}^d$ & Random neural network approximation for $u^d$, 
 see Section \ref{section RNN}\\
 \hline
\end{tabular}
\end{center}

\smallskip
\noindent 

\bibliographystyle{plain}
\bibliography{Reference}

\end{document}